\documentclass[english, usenames,dvipsnames,svgnames,table]{article}
\usepackage{preamble}

\title{Higher Semiadditive Character Theory}

\begin{document}

\begin{titlepage}
    \maketitle
    \thispagestyle{empty}
    \begin{abstract}
    We introduce and develop the theory of semiadditive characters in the higher semiadditive setting, generalizing both the $T(n)$-local monoidal character and the $K(t)$-local transchromatic character. Such characters are natural transformations
    \begin{equation*}
        X^{(-)} \to Y^{\L_p^{\,n-t}(-)}
    \end{equation*}
    compatible with restriction and transfer along $\pi$-finite spaces. We show that every $\infty$-commutative monoid admits a universal $(n-t)$-fold character. This universal character has several strong structural properties: it exhibits blue shift, satisfies higher cyclotomic descent, and is compatible with the semiadditive Fourier transform. We compute it for an arbitrary $K(n)$-local object and show that, for Morava $E$-theory, it recovers the $K(t)$-local transchromatic character. By functoriality, the universal character carries a natural action of the profinite group $\GL_{n-t}(\ZZ_p)$. When $t=0$, the fixed points of this action recover rationalization. As a consequence, we derive an explicit description of $L_\QQ(\SS_{K(n)}^A)$ for every $\pi$-finite space $A$ and compute the ring of rational $K(n)$-local power operations.  
     
     \medskip
    
    \begin{figure}[h]
    \centering
    \includegraphics[width=.90\linewidth]{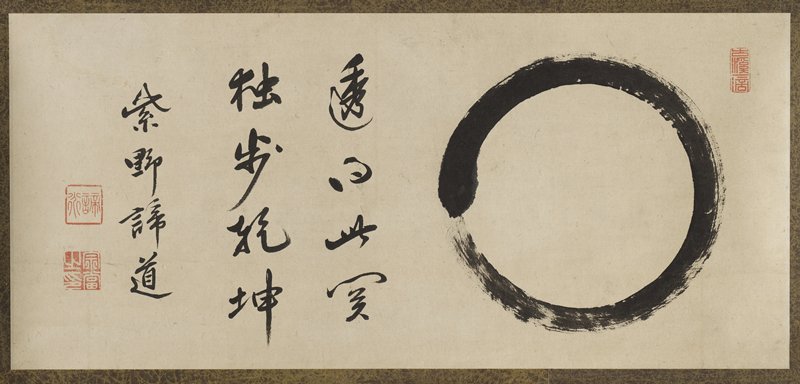}
    \caption*{Enso, Taido Soto (Japanese, 1776--1836), 19th century. 
Minneapolis Institute of Art, Public Domain.}
    \end{figure}
    
\end{abstract}

\end{titlepage}

\tableofcontents
\setcounter{page}{1}
\newpage

\section{Introduction}

\subsection{Background and Motivation}

In their foundational work on generalized characters, Hopkins--Kuhn--Ravenel
\cite{Hopkins-Kuhn-Ravanel-2000-HKR} constructed, for Morava $E$-theory at height $n$,
a character map comparing the $E_n$-cohomology of classifying spaces with a height $0$ theory of generalized class functions.

The height $1$ case provides a useful model for the general picture. Taking
\[
E_1\simeq \KU_p^\wedge,
\]
the Atiyah-Segal completion theorem identifies, for a finite $p$-group $G$, the
spectrum $(\KU_p^\wedge)^{\B G}$ with the completed representation ring. The resulting character map is the usual
$p$-adic character map
\[
\chi\colon (\KU_p^\wedge)^{\B G}\to
\QQ_p(\omega_{p^\infty})[\beta^{\pm1}]^{\L\B G},
\qquad |\beta|=2.
\]
Equivalently, one may describe the same map in terms of $p$-divisible groups. The formal group associated to $\KU_p^\wedge$ is $\widehat{\Gm}$, whose $p^r$-torsion is $\mu_{p^r}$. After base change to $\QQ_p$, this becomes an etale $p$-divisible
group, and adjoining all of its torsion points produces the extension $\QQ_p(\omega_{p^\infty})$.

Hopkins--Kuhn--Ravenel observed that this picture persists in higher heights. For Morava $E$-theory, one starts with the canonical $p$-divisible group on $E_n$, after rationalization it becomes etale, and trivializing it produces the generalized character map.

Stapleton \cite{Stapleton-2013-HKR} later showed that this story admits intermediate stages. Rather than passing directly from height $n$ to height $0$, one may descend to any intermediate height $0\le t\le n$. The target is obtained by taking the canonical $p$-divisible group on $E_n$, base changing to $L_{K(t)}E_n$, splitting its
connected-etale sequence, and trivializing the etale part.

\medskip

From the point of view of this paper, the striking feature of transchromatic character theory is not only the existence of these maps, but also the additional structure carried by both their source and target. By the ambidexterity theorem of
Hopkins--Lurie for $K(n)$-local spectra, $K(n)$-local objects admit
canonical restriction and transfer maps along $\pi$-finite spaces
\cite{Hopkins-Lurie-2013-ambi}. These transfer operations have become an important organizing principle in chromatic homotopy theory, and it is natural to ask
how the transchromatic character interacts with them.

A convenient way to package such transfer data is as follows. Let
\[
\PCMoninf(\cD)\coloneqq \Fun(\Span(\spcpi),\cD),
\]
and let
\[
\CMoninf(\cD)\subseteq \PCMoninf(\cD)
\]
denote the full subcategory of those functors satisfying the Segal condition, namely
those functors $M^{(-)}$ for which the natural map
\[
M^A \to (M^{\pt})^A
\]
is an equivalence for every $A\in \spcpi$. Thus an object of $\CMoninf(\cD)$ is an
object of $\cD$ equipped with restriction and transfer maps along all maps of
$\pi$-finite spaces. In particular, the higher semiadditivity results above allow us to regard $\Sp_{K(n)}$ as a full subcategory of $\CMoninf(\Sp)$. The same is also true for $\Sp_{T(n)}$.

\begin{theorem}[\cite{CSY-teleambi}]
Every $T(n)$-local object admits canonical restriction and transfer maps along $\pi$-finite spaces. Equivalently, the
forgetful functor
\[
\CMoninf(\Sp_{T(n)})\xto{\simeq}\Sp_{T(n)}
\]
is an equivalence.
\end{theorem}

At first sight, however, there is an evident tension. The transchromatic character lowers chromatic height, and so it cannot preserve semiadditive structure naively, see
\cite[Proposition 5.3.8]{CSY-ambiheight}. Recent work of Ben-Moshe shows that the correct compatibility is restored only after inserting a $p$-typical free loop space
for each unit of height drop \cite{ben-mosh-Transchromatic}.

\begin{theorem}[\cite{Hopkins-Kuhn-Ravanel-2000-HKR,Stapleton-2013-HKR,Lurie-2019-Elliptic3,ben-mosh-Transchromatic}]
Let $E_n$ be Morava $E$-theory at height $n$ and choose  $0\le t\le n$. Then there is a canonical transchromatic character map
\[
\chi^{\mathrm{tch}}_{n,t}\colon E_n^{\,(-)}
\to
(C_t^n)^{\,\L_p^{\,n-t}(-)}
\qquad \in \qquad \Fun(\Span(\spcpi),\Sp),
\]
where $C_t^n$ is the target
appearing in transchromatic character theory  \footnote{In this work we will only use the $T(t)$-localization, equivalently the $K(t)$-localization, of the relevant splitting algebra.}, and
\[
\L_p^{\,n-t}A\coloneqq \Map(\B\ZZ_p^{\,n-t},A).
\]
\end{theorem}

This suggests reversing the story. Rather than starting with a specific character map
and then discovering that a free loop shift is forced upon us, one may instead begin by
asking for a natural transformation of the form
\[
X^{(-)}\to Y^{\L_p^{\,n-t}(-)}
\]
which is compatible with restriction and transfer along $\pi$-finite spaces, and ask
what such a requirement forces. Equivalently, once the loop-space correction is built
into the problem from the outset, it is natural to ask whether there is a theory
parallel to transchromatic character theory for a general object
\[
X\in \CMoninf(\Sp),
\]
rather than only for Morava $E$-theory.

The main result of this paper is that such a theory exists in much greater generality. We show that every $\infty$-commutative monoid admits a universal $(n-t)$-fold character, and that for Morava $E$-theory this recovers the transchromatic character.
This universal character enjoys strong structural properties, including blue shift, finite cyclotomic descent, and compatibility with the semiadditive Fourier transform. It also allows many results previously known for Morava $E$-theory to be extended to arbitrary $K(n)$-local spectra. As applications, we compute $L_\QQ(\SS_{K(n)}^A)$ for every $\pi$-finite space $A$, and determine the ring of rational power operations in the $K(n)$-local category.

\subsection{Higher Semiadditive Character Theory}

Fix a prime $p$ and integers $0\le c\le n$, and call $c$ the coheight. For a $\pi$-finite space $A$, the
$c$-fold $p$-typical loop space is
\[
\L_p^cA\coloneqq \Map(\B\ZZ_p^c,A).
\]
This construction is functorial not only in maps but also in spans of $\pi$-finite
spaces. Hence it determines an endofunctor
\[
\L_p^c\colon \Span(\spcpi)\to \Span(\spcpi),
\]
and therefore, by precomposition, a functor
\[
(\L_p^c)^*\colon \CMoninf(\Sp)\to \PCMoninf(\Sp),
\qquad
X^{(-)}\mapsto X^{\L_p^c(-)}.
\]

Recall that an object $X\in \CMoninf(\Sp)$ is a functor
\[
X^{(-)}\in \Fun\!\bigl(\Span(\spcpi),\Sp\bigr)=:\PCMoninf(\Sp)
\]
satisfying the Segal condition
\[
X^A\simeq (X^{\pt})^A.
\]
In particular, such an object encodes restriction and transfer along all maps of
$\pi$-finite spaces. 

\begin{definition}
Let $X,Y\in \CMoninf(\Sp)$. A \emph{$c$-fold character of $X$ with values in $Y$}
is a natural transformation
\[
\chi\colon X^{(-)} \to Y^{\L_p^c(-)}
\qquad \in \qquad \PCMoninf(\Sp).
\]
Equivalently, it is a family of maps
\[
\chi_A\colon X^A\to Y^{\L_p^cA},
\]
natural in $A$ and compatible with both restriction and transfer.
\end{definition}

This notion simultaneously encompasses the classical monoidal character, for rational
commutative ring spectra,\footnote{In work in progress by
Carmeli--Cnossen--Ramzi--Yanovski, following up on
\cite{Carmeli-Cnossen-Ramzi-Yanovski-2022-characters}, it is shown that the monoidal
character of any $T(n)$-local commutative algebra defines a character in the sense of the
above definition.} and the transchromatic characters $\chi^{\mathrm{tch}}_{n,n-c}$.

A key feature of the theory is that $(\L_p^c)^*$ admits a symmetric monoidal left
adjoint
\[
(\L_p^c)_!\colon \PCMoninf(\Sp)\to \CMoninf(\Sp).
\]
For $X\in \CMoninf(\Sp)$, the unit of this adjunction defines the
\emph{universal $c$-fold character}
\[
\chi^{\mrm{uni}}_{n,c}\colon
X^{(-)}\to
\bigl((\L_p^c)_!X\bigr)^{\L_p^c(-)}.
\]
By construction, every $c$-fold character out of $X$ factors uniquely through
$\chi^{\mrm{uni}}_{n,c}$. We show that this universal character enjoys several strong structural properties.

\subsubsection*{Blue-shift}

The red-shift philosophy predicts that categorified constructions, such as algebraic $K$-theory, raise chromatic height. Character theory goes in the opposite direction. It extracts functions from objects with symmetries, and therefore behaves as a form of
decategorification. From this perspective, one expects character theory to lower chromatic height. The first structural property of the universal character makes this precise.

\begin{alphtheorem}[\cref{cor: F chr single height T local}]
Let $R\in \Sp_{T(n)}$. Then
\[
(\L_p^c)_!R\in \Sp_{T(n-c)}.
\]
Consequently, if
\[
R^{(-)}\to S^{\L_p^c(-)}
\]
is a $c$-fold character of commutative rings, then $S\in \Sp_{T(n-c)}$.
\end{alphtheorem}

In particular, the universal $c$-fold character carries $T(n)$-local information to $T(n-c)$-local information.

\subsubsection*{Cyclotomic blue-shift}

The second structural property concerns higher cyclotomic extensions. In \cite{CSY-cyclotomic}, Carmeli, Schlank, and Yanovski construct, for every $R\in \Sp_{T(n)}$, higher cyclotomic extensions
\[
R[\omega_{p^r}^{(n)}]\in \CAlg(\Sp_{T(n)})^{\B(\ZZ/p^r)\units},
\]
which play the role of adjoining $p^r$-th roots of unity in the $T(n)$-local setting. 

We show that the universal character is compatible with these extensions, up to the height shift predicted by blue shift. 

\begin{alphtheorem}[\cref{thm:cyclo blue shift}] \label{alphthm: cyclotomic blushift}
Let $R\in \CAlg(\Sp_{T(n)})$. Then there is a canonical equivalence
\[
(\L_p^c)_!\bigl(R[\omega_{p^r}^{(n)}]\bigr)
\simeq
(\L_p^c)_!(R)[\omega_{p^r}^{(n-c)}]
\qquad
\in \qquad
\CAlg(\Sp_{T(n-c)})^{\B(\ZZ/p^r)\units}.
\]
\end{alphtheorem}

Thus the universal character carries height $n$ cyclotomic extensions to the corresponding height $n-c$ cyclotomic extensions, compatibly with the cyclotomic Galois actions.

\subsubsection*{Fourier transform}

The third structural property is compatibility with the semiadditive Fourier transform.
One of the most striking features of the higher cyclotomic extensions is that a choice
of ring map
\[
R[\omega_{p^r}^{(n)}]\to S
\]
is precisely the datum needed to construct a height $n$ semiadditive Fourier transform, see \cite[Definition 4.1, Proposition 6.6]{BCSY-Fourier}. Concretely, if $M\in \Mod_{\ZZ/p^r}(\Ab)^{\fin}$, then such a map gives canonical equivalences of $R$-algebras
\[
S^{\B^kM^\vee}\simeq S[\B^{n-k}M],
\qquad 0\le k\le n.
\]

Using the cyclotomic compatibility above, we show that the universal character carries
the semiadditive Fourier transform at height $n$ to the corresponding Fourier transform
at height $n-c$. More precisely, if $R$ admits a primitive $p^r$-th root of unity of height
$n$, then $(\L_p^c)_!R$ admits a primitive $p^r$-th root of unity of height $n-c$. The induced height $n-c$ Fourier transform can be computed explicitly from the height $n$ Fourier transform on $R$.

\begin{alphtheorem}[\cref{prop: Fourier in chr}]
Let $R\in \CAlg(\Sp_{T(n)})$ and assume that $R$ admits a primitive $p^r$-th root of
unity of height $n$. Let $0\le k\le n-c$, and let $M$ be a finite $p^r$-torsion abelian
group. If
\[
\mathcal F\colon R[\B^kM]\xto{\simeq} R^{\B^{n-k}M^\vee}
\]
denotes the height $n$ semiadditive Fourier transform, then the height $n-c$ Fourier transform of $(\L_p^c)_!R$ is induced from $\mathcal F$ by the following diagram:
\[
\begin{tikzcd}[column sep=large,row sep=large]
(\L_p^c)_!(R)[\B^kM]
    \arrow[r]
    \arrow[d,dashed,"\simeq"']
&
(\L_p^c)_!(R[\B^kM])
    \arrow[d,"(\L_p^c)_!(\mathcal F)"]
\\
(\L_p^c)_!(R)^{\B^{n-c-k}M^\vee}
&
(\L_p^c)_!(R^{\B^{n-k}M^\vee})
    \arrow[l]
\end{tikzcd}
\]
Here the top horizontal map is the assembly map, the right vertical map is obtained by
applying $(\L_p^c)_!$ to $\mathcal F$, and the bottom horizontal map is induced by
the universal character together with the projection
\[
(\B^{n-c-k}M^\vee)\to \L_p^c\B^{n-k}M^\vee
\]
coming from a choice of a degree $1$ map $\TT^c \to S^c$ \footnote{Note that there is a unique such map up to non-canonical homotopy, induced by the canonical orientation of $S^1$.}.
\end{alphtheorem}

\subsubsection*{Base change}

Finally, the universal character is compatible with exponentiation by $\pi$-finite
spaces.

\begin{proposition}[\cref{thm: F! of R A}]
Let $R\in \CAlg(\Sp_{T(n)})$ and let $A\in \spcpi$. Then there is a canonical
equivalence, functorial in $A$,
\[
(\L_p^c)_!(R^A)
\simeq
\bigl((\L_p^c)_!R\bigr)^{\L_p^cA}.
\]
\end{proposition}

All of the results above hold in greater generality: one may replace $\L_p^c$ by an arbitrary exact functor satisfying mild finiteness and orientability conditions: see \cref{def: dim F} and \cref{def: F orientable}. One may also replace $\Sp$ by any
additive category.

\subsection{The $K(n)$-Local Universal Character}

We now turn to applying the general theory to the $K(n)$-local category. Our general approach is first to identify the transchromatic character as the
universal character of Morava $E$-theory and then use the descent theory developed by Mathew
\cite{Mathew-2016-Galois}, together with the structural results of the previous
section, to extend many results about the transchromatic character from Morava
$E$-theory to arbitrary $K(n)$-local objects. The key input for the first step is the tempered refinement of the transchromatic character, \cite{Lurie-2019-Elliptic3,ben-mosh-Transchromatic}, which gives the transchromatic character the right kind of semiadditive naturality.

\subsubsection*{Formula for the $K(n)$-local universal character}

We begin by revisiting the $K(n-c)=K(t)$-local splitting algebra appearing in
transchromatic character theory. Classically, this algebra is obtained by inverting a
determinant which detects the relevant splitting condition on the $p$-divisible group
associated to Morava $E$-theory.

The following proposition gives a semiadditive description of this determinant.

\begin{proposition}[\cref{cor: indtifying the spliting algebra}]
There is an idempotent
\[
e_r\in
\pi_0 L_{T(n-c)}\!\bigl(\En^{\B(\ZZ/p^r)^c}\bigr)
\]
constructed from semiadditive data and a choice of a height $n$ primitive $p^r$-th root of unity. After inverting this
idempotent, one obtains the finite-level splitting algebra
\[
C_{n-c,r}^n\simeq
L_{T(n-c)}(\En^{\B(\ZZ/p^r)^c})[e_r^{-1}].
\]
Consequently,
\[
C_{n-c}^n \simeq
\colim_r\,L_{T(n-c)}(\En^{\B(\ZZ/p^r)^c})[e_r^{-1}].
\]
\end{proposition}

As this construction uses only higher semiadditivity and higher roots of unity, it applies to any $K(n)$-local algebra equipped with all height $n$ primitive $p^r$-th roots of unity.

\begin{construction}\label{cons:CtR}
Assume that $R$ is a $K(n)$-local algebra which admits all height $n$ primitive
$p^r$-th roots of unity for all $r$, i.e.
\[
R\in \CAlg_{\SS_{K(n)}[\omega^{(n)}_{p^\infty}]}(\Sp_{T(n)}).
\]
For each $r\ge 1$, let $e_r$ denote the above idempotent in $L_{T(n-c)}\!\bigl(R^{\B(\ZZ/p^r)^c}\bigr)$. Set
\[
C_{n-c,r}(R)\coloneqq
L_{T(n-c)}\!\bigl(R^{\B(\ZZ/p^r)^c}\bigr)[e_r^{-1}]
\]
and define
\[
C_{n-c}(R)\coloneqq \colim_r C_{n-c,r}(R),
\]
where the colimit is taken in $\CAlg(\Sp_{T(n-c)})$.
\end{construction}

We prove that this construction computes the universal character and in particular identifies the universal character of $E_n$ with the transchromatic character.

\begin{alphtheorem}[\cref{thm: k(n) local character with roots of unity}]\label{thm:Ct-is-L!}
Assume that $R$ is $K(n)$-local and admits all height $n$ primitive $p^r$-th roots of
unity for all $r$. Then there is a canonical equivalence
\[
(\L_p^c)_!R \simeq C_{n-c}(R)
\]
in  $\Sp_{T(n-c)}$. Moreover, under this equivalence, the universal character of $E_n$ agrees with the transchromatic character.
\end{alphtheorem}

The same formula holds for the universal character of any $R$-module, for $R$ as above. For a general $K(n)$-local object $M$, one first adjoins all height $n$ roots of unity, applies the theorem, and then takes fixed points for the induced $\ZZ_p\units$-action. This is carried out in \cref{cor: L!-as-hZ}.

Since the universal character is compatible with exponentiation, the above theorem also gives a character isomorphism generalizing the character isomorphism of \cite{Hopkins-Kuhn-Ravanel-2000-HKR}.

\begin{proposition}[\cref{cor: char-iso}]
Assume that $R$ is $K(n)$-local and admits all height $n$ primitive $p^r$-th roots of unity for all $r$. Then, for any $\pi$-finite space $A$, there is a canonical
equivalence
\[
C_{n-c}(R^A)\simeq C_{n-c}(R)^{\L_p^cA}
\]
in  $\Sp_{T(n-c)}$.
\end{proposition}

\subsubsection*{The $\GL_n(\ZZ_p)$-action}

The universal character also has many symmetries. The group $\ZZ_p^c$ has a natural
action of the profinite group $\GL_c(\ZZ_p)$, and by functoriality the universal
character
\[
\chi^{\mrm{uni}}_{n,c}\colon
R^{(-)}\to
\bigl((\L_p^c)_!R\bigr)^{\L_p^c(-)}
\]
is $\GL_c(\ZZ_p)$-equivariant. In the case $c=n$, this gives a
$\GL_n(\ZZ_p)$-action. We show that, after taking fixed points for this action,
one recovers rationalization.

\begin{alphtheorem}[\cref{thm: unit-rational-hGL}] \label{thm: fixed point intro}
Let $M\in \Sp_{K(n)}$ and let $A$ be a $\pi$-finite space.
Then there is a canonical equivalence
\[
\bigl((\L_p^n)_!M^{\L_p^nA}\bigr)^{h\GL_n(\ZZ_p)}
\simeq
L_\QQ(M^A).
\]
\end{alphtheorem}

\subsubsection*{The universal character of the $K(n)$-local sphere}

We next apply the general theory to the $K(n)$-local sphere. In this case, the
universal $n$-fold character admits a particularly simple description.

\begin{proposition}[\cref{cor: universal character SK(n)}]
The universal $n$-fold character of the $K(n)$-local sphere is its rationalization.
More precisely,
\[
(\L_p^n)_!(\SS_{K(n)})\simeq L_\QQ(\SS_{K(n)}),
\]
and the induced $\GL_n(\ZZ_p)$-action is trivial.
\end{proposition}

A key step in the computation is to give a geometric interpretation of the finite-level covers appearing in the universal character formula as covers of the rigid-analytic generic fiber of the Lubin-Tate formal scheme. Using the two-towers isomorphism, we then compute these covers by moving them to the Drinfeld side.

As a consequence, by passing to $\GL_n(\ZZ_p)$-fixed points, we compute the rationalization of $\SS_{K(n)}^A$ for every $\pi$-finite space $A$.

\begin{alphtheorem}[\cref{thm: rationalization SK(n)A}]
Let $A$ be a $\pi$-finite space. The profinite group $\GL_n(\ZZ_p)$ acts by precomposition on
\[
\Map(\B\ZZ_p^n,A)
\simeq
\colim_{r}\Map(\B(\ZZ/p^r)^n,A),
\]
where the equivalence follows from \cite[Proposition~3.4.7]{Lurie-2019-Elliptic3}. Write
\[
S =
\pi_0\Map(\B\ZZ_p^n,A)/\GL_n(\ZZ_p).
\]
Then there is an equivalence
\[
L_\QQ(\SS_{K(n)}^A)
\simeq
L_\QQ(\SS_{K(n)})^S.
\]
\end{alphtheorem}

For classifying spaces of finite groups this gives a concrete counting formula.

\begin{corollary}[\cref{cor: rationalization SK(n)BG}]
Let $G$ be a finite group, and let
\[
S:=
\left\{
\text{abelian }p\text{-subgroups of }G
\text{ generated by at most }n\text{ elements}
\right\}\big/\text{conjugacy}.
\]
Then
\[
L_\QQ(\SS_{K(n)}^{\B G})
\simeq
L_\QQ(\SS_{K(n)})^S.
\]
\end{corollary}

\subsubsection*{Rational $K(n)$-local power operations}

We use the above computation to identify the ring of rational
$K(n)$-local power operations. Consider the functor
\[
\pi_0^\QQ\colon\CAlg(\SpKn)\to\Set,
\qquad
A\mapsto\pi_0(L_\QQ A),
\]
where we regard $\pi_0(L_\QQ A)$ as its underlying set. Pointwise
addition and multiplication endow
$\pi_0\Nat(\pi_0^\QQ,\pi_0^\QQ)$ with a ring structure. We refer to this ring as the ring of rational $K(n)$-local power operations, since it acts
naturally on $\pi_0(L_\QQ A)$ for every
$A\in\CAlg(\SpKn)$.

\begin{proposition}[\cref{lem:yoneda-identifies-power-operations,prop: K(n) local rational power operations}]
The ring of rational $K(n)$-local power operations admits a canonical
grading, with respect to which there are isomorphisms of graded
rings
\[
\pi_0\Nat(\pi_0^\QQ,\pi_0^\QQ)
\simeq
\pi_0\left(
\bigoplus_{k\geq0}L_\QQ\SSKn^{\B\Sn[k]}
\right)
\simeq
\QQ_p[x_H\mid H\leq_o\ZZ_p^n]^{\GL_n(\ZZ_p)}.
\]
Where the grading on $\QQ_p[x_H\mid H\leq_o\ZZ_p^n]^{\GL_n(\ZZ_p)}$ is
induced by declaring
\[
\deg(x_H)=[\ZZ_p^n:H]
\]
before taking fixed points.
\end{proposition}

\subsection{Organization}

In \cref{sec: Higher Semi-Additive Character} we develop the general formalism of $F$-characters, for an exact endofunctor $F\colon \spcpi\to \spcpi$ commuting with finite coproducts. We define $F$-characters as span-natural transformations, show that the functor $F^*$ admits a symmetric monoidal left adjoint $F_!$, and obtain the universal $F$-character $X\to F^*F_!X$. We then prove the basic structural consequences of universality: we establish a blue-shift theorem, prove a base-change/exponentiation formula, and a cyclotomic blue-shift equivalence.

In \cref{sec: Fourier Transform and the Character of Higher Roots of Unity} we study the interaction between $F$-characters and the semiadditive Fourier transform. We first analyze multiplicative structure and explain how maps into the underlying multiplicative commutative monoid transport along an $F$-character. We then show that characters send primitive higher roots of unity to primitive ones, and use this to  express the Fourier transform on the target of an $F$-character explicitly in terms of the Fourier transform on the source.

In \cref{sec: splitting algebra} we revisit the construction of the
$T(t)$-local splitting algebra appearing in transchromatic character theory. Using the chromatic Fourier transform together with the theory of exterior powers of $p$-divisible groups, we give a semiadditive description of the determinant inversion that trivializes the etale part, recovering the splitting algebra from this perspective.

In \cref{sec: The $K(n)$-Local Universal Character} we turn to the $K(n)$-local universal character. We show that the transchromatic character is the universal character of $\En$. We then bootstrap this to compute the universal $(n-t)$-fold character of any $K(n)$-local object. We then investigate the natural profinite group action of $\GL_{n-t}(\ZZ_p)$ on the universal character, and show that when $t=0$ the resulting fixed points recover rationalization.

Finally, in \cref{sec: Character Theory for the K(n) Local Sphere} we specialize the general theory to the $K(n)$-local sphere. We study the Lubin-Tate and Drinfeld sides of the relevant finite-level covers and compute their rational condensed global sections. We then compare the two constructions using the two-towers isomorphism, which allows us to identify the finite-level pieces appearing in the universal character formula. This lets us determine the universal $n$-fold character of $\SS_{K(n)}$, and we deduce from this an explicit formula for the rationalization of $\SS_{K(n)}^A$ for every $\pi$-finite space $A$. We then use this to determine the ring of rational $K(n)$-local power operations.

\subsection{Conventions}
We use the following terminology and notation:

\begin{enumerate}
    \item We use the term \emph{category} to mean an $(\infty,1)$-category, and the term \emph{space} to mean an $\infty$-groupoid. We write $\spc$ for the $\infty$-category of spaces and $\Cat$
    for the $\infty$-category of (small) categories.

    \item We write $\Sp$ for the $\infty$-category of spectra and $\Sp\cn\subseteq \Sp$ for the full subcategory of connective spectra.

    \item For a category $\cC$ we write $\cC^{\simeq}\subseteq \cC$ for its maximal subgroupoid.
    For $X,Y\in \cC$ we write $\Map_{\cC}(X,Y)$ for the mapping space (and omit the subscript when $\cC$ is clear).

    \item We write $\PrL$ for the $\infty$-category of presentable categories and colimit-preserving functors, and $\CAlg(\PrL)$ for presentably symmetric monoidal categories. If
    $\cC\in \CAlg(\PrL)$, we write $\Mod_{\cC}(\PrL)$ for the $\infty$-category of $\cC$-module categories.

    \item We write $\spcpi\subseteq \spc$ for the full subcategory of $\pi$-finite spaces (spaces with finitely many nontrivial homotopy groups, all finite, and finitely many components). When we
    need to emphasize $p$-locality we will say so explicitly.

    \item For $A\in \spc$ and $X\in \cC$, we denote by $X^{A}$ (resp.\ $X[A]$) the constant $A$-indexed limit (resp.\ colimit) of $X$, whenever it exists.

    \item We write $\Span(\spcpi)$ for the category of spans in $\spcpi$: objects are $\pi$-finite spaces and morphisms are spans $A\xleftarrow{}B\xrightarrow{}C$, with composition by pullback. We regard $\Span(\spcpi)$ as symmetric monoidal under cartesian product.

    \item We write
    \[
    \PCMoninf(\cC) := \Fun(\Span(\spcpi),\cC)
    \]
    and endow it with the Day convolution symmetric monoidal structure induced from the cartesian product on $\Span(\spcpi)$ and the tensor product on $\cC$.

    \item We write $\CMoninf(\cC)\subseteq \PCMoninf(\cC)$ for the full subcategory of \emph{$\infty$-commutative monoids} (i.e.\ those functors satisfying the Segal condition).

    \item For chromatic localization we write $L_{K(n)}$ and $L_{T(n)}$ for Bousfield localization  at $K(n)$ and $T(n)$, and $\Sp_{K(n)}$, $\Sp_{T(n)}$ for the corresponding local categories.

    \item We write $\tsadi:=\CMon_\infty(\Sp)$ for the universal $\infty$-semiadditive stable category. For \(n\ge 0\), we write
    $\tsadi_n\subseteq \tsadi $ for the full subcategory of objects of semiadditive height exactly $n$, and similarly $\tsadi_{>n}\subseteq \tsadi$ for the full subcategory of objects of semiadditive height $>n$.

    \item We denote by $C_t^n$ the $T(t)$-local, equivalently $K(t)$-local, splitting algebra from transchromatic character theory. Note that this is usually denoted by $\widehat{C^n_t}$.

\end{enumerate}

\subsection{Acknowledgments}

I would like to thank my advisor, Tomer Schlank, for his patience, valuable insights, and many helpful conversations. I am especially grateful to Lior Yanovski, whose ideas both inspired this project and shaped its development, and whose suggestions helped guide it throughout. I am also grateful to my academic brother, Shai Keidar, for numerous insightful discussions and for reading earlier drafts of this paper. I thank Beckham Myers, Shay Ben-Moshe, and Lior Yanovski for their comments on previous drafts. \\
LLMs were used to improve the writing. All mathematical content, and any errors, are my own.

\section{Higher Semiadditive Characters}\label{sec: Higher Semi-Additive Character}

This section defines and records the basic properties of $F$-characters that we will use throughout the paper. In \cref{subsec: F Characters} we describe the general setup, introduce $F$-characters, and prove the existence of a universal $F$-character. In \cref{subsec: blue shift} we show that an $F$-character lowers semiadditive height (a blue-shift phenomenon) by the ``dimension'' of $F$ (see \cref{def: dim F}). In \cref{subsec: Base Change} we prove a base-change formula relating the universal $F$-character of $R^A$ to that of $R$. Finally, in \cref{sec: Cyclotomic Blue-Shift} we apply these results to obtain a cyclotomic blue-shift equivalence, see \cref{thm:cyclo blue shift}.

\subsection{$F$-Characters}\label{subsec: F Characters}

Fix a prime $p$, a presentably symmetric monoidal $p$-local category $\cC \in \CAlg(\PrL_{(p)})$, and a left exact endofunctor
\[
F\colon \spcpi \to \spcpi
\]
which commutes with finite coproducts. The functor $F$ induces a symmetric monoidal functor (with respect to the product on $\pi$-finite spaces)
\[
F\colon \Span(\spcpi)\to \Span(\spcpi),
\]
which, by abuse of notation, we again denote by $F$. Precomposition with $F$ then defines a lax monoidal functor
\[
F^*\colon \CMoninf(\cC)\to \PCMoninf(\cC),
\qquad
X^{(-)} \mapsto X^{F(-)}.
\]

The prototypical examples of such functors $F$ arise from manifolds.

\begin{example}
    Let $M$ be a compact connected space. Then the functor
    \[
    \Map(M,-)\colon \Span(\spcpi)\to \Span(\spcpi)
    \]
    is exact and commutes with finite coproducts. In the special case $M=(S^1)^{n-t}$ with $n-t\in \NN$, we obtain a functor
    \[
    (\L^{n-t})^*\colon \CMoninf(\cC)\to \PCMoninf(\cC),
    \qquad
    X^{(-)}\to X^{\L^{n-t}(-)}.
    \]
\end{example}

We now define \emph{$F$-characters}, the central objects of study in this work.

\begin{definition}
    Let $X,Y \in \CMoninf(\cC)$.
    \begin{itemize}
        \item An \emph{$F$-character of $X$ with values in $Y$} is a morphism
        \[
        X \to F^*Y
        \]
        in $\PCMoninf(\cC)$. When the source and target are clear from context, we may simply speak of an \emph{$F$-character of $X$} or an \emph{$F$-character}.
        
        \item If, moreover, $X,Y \in \CAlg(\CMoninf(\cC))$ and the above morphism is a morphism of commutative algebras, then we call it an \emph{$F$-character of commutative algebras}.
        
        \item When $F=\L^{c}_p$ or $F=\L^{c}$, we refer to such a morphism as a \emph{$p$-typical $c$-fold character} or a \emph{$c$-fold character}, respectively, of $X$ with values in $Y$. When the prime $p$ is clear from context, we omit it from the terminology.
        
        \item In the special case $c=1$, we simply call it a \emph{character} of $X$ with values in $Y$.
    \end{itemize}
\end{definition}
The following constructions fit into this general framework.

\begin{example}
    Let $R$ be a rational ring spectrum. The induced character formula shows that the monoidal character map
    \[
    (\Mod_R\dbl)^{A} \to R^{\L A}
    \]
    is functorial in $\Span(\spcpi)$. In particular, in our terminology this monoidal character is a character of $\Mod_R\dbl$ with values in $R$.

    Moreover, work in progress of Carmeli--Cnossen--Ramzi--Yanovski, building on
    \cite{Carmeli-Cnossen-Ramzi-Yanovski-2022-characters}, shows that for any
    $R\in \CAlg(\Sp_{T(n)})$ the monoidal character map
    \[
    (\Mod_R(\Sp_{T(n)})\dbl)^{A} \to R^{\L A}
    \]
    is functorial in $\Span(\spcpi)$, and hence defines a character in our sense, extending the case $n=0$.
\end{example}

\begin{example}
    Let $\En$ denote the Morava $E$-theory associated to a height $n$ formal group $\Gamma$ over a perfect field $\kappa$. By \cite[Theorem~A]{ben-mosh-Transchromatic}, the $(n-t)$-fold transchromatic character defines an $(n-t)$-fold character map. For the definition of the transchromatic character, see \cite{Stapleton-2013-HKR,Hopkins-Kuhn-Ravanel-2000-HKR,Lurie-2019-Elliptic3}\footnote{We revisit its construction in \cref{sec: splitting algebra}.}.
\end{example}

Another example, of particular importance for us, is the universal one.

\begin{lemma}\label{lem: exists of L! and mult st}
    The functor $F^*$ admits a symmetric monoidal left adjoint.
\end{lemma}

\begin{proof}
    We may write $F^*$ as the composite
    \[
    \CMoninf(\cC)\xto{i} \PCMoninf(\cC)\xto{F^*}\PCMoninf(\cC).
    \]
    Since limits in a functor category are computed pointwise, the functor $F^*$ admits a left adjoint. The inclusion $i$ also admits a left adjoint, given by Segalification, see, for example, \cite[Definition~4.7]{Ben-Moshe-Schlank-2024-K-theory}.

    Moreover, both Segalification \cite[Proposition~4.24]{Ben-Moshe-Schlank-2024-K-theory} and the functor $\Fun((-)\op,\cC)\colon \Cat\to \PrL$ are symmetric monoidal. Since $\Fun(F,\cC)=F_!$, it follows that the resulting left adjoint is symmetric monoidal as well.
\end{proof}

We denote the left adjoint from \cref{lem: exists of L! and mult st} by $F_{!}$, and we omit the subscript $\cC$ when it is clear from context.

\begin{corollary}
    Let $X,Y\in \CMoninf(\cC)$. Any $F$-character $X\to F^*Y$ factors uniquely through the unit map
    \[
    X \to F^*F_!X.
    \]
    We call this unit map the \emph{$F$-universal character}.
\end{corollary}

\subsection{Blue-Shift}\label{subsec: blue shift}

We turn to show that $\L^{n-t}_!$ blue-shifts (lowers height) by $n-t$, and more generally that $F_!$ blue-shifts by its \emph{dimension} which we define now.

\begin{definition}\label{def: dim F}
    We say that $F$ is of $\ZZ/p^r$ dimension $n-t$ if, for every $k \ge n-t$, the space $F(\B^{k}\ZZ/p^r)$ is $(k-n+t-1)$-connected but not $(k-n+t)$-connected, equivalently, the first nonzero homotopy group of $F(\B^{k}\ZZ/p^r)$ occurs in degree $k-n+t$. 
\end{definition}

\begin{remark}
    If $F=\Map(M,-)$ for $M$ a compact connected space, then 
    \[
    \Map(M,\B^k\ZZ/p^r)\simeq \prod_{i=0}^{k} \B^i H^{k-i}(M,\ZZ/p^r).
    \]
    In particular, if $M$ is a closed connected $\ZZ/p^r$-oriented manifold we have  $\dim_{\ZZ/p^r}(F)=\dim(M)$.
\end{remark}

\begin{lemma}\label{lem: F(BkCp) Cp module}
Let $k$ be a natural number. Then $F(\B^{k}\ZZ/p)$ is equivalent to a finite product of Eilenberg-MacLane spaces of the form $\B^{i}\ZZ/p$.
\end{lemma}

\begin{proof}
Since $F \colon \spcpi \to \spcpi$ is symmetric monoidal, it preserves commutative group objects and their module objects. In particular, $F(\B^{k}\ZZ/p)$ carries
a canonical structure of module object over $F(\ZZ/p)$. Moreover, as $F$ preserves finite coproducts, for every $m$ we have $F([m]) \simeq \coprod_{m} F(\pt)$. This identifies $F(\ZZ/p)\simeq \ZZ/p$. Hence $F(\B^{k}\ZZ/p)$ is naturally a $\ZZ/p$-module object, equivalently an $\FF_p$-module object in connective spectra. Such objects decompose into products of Eilenberg-MacLane spaces $\B^{i}\ZZ/p$, since $F(\B^{k}\ZZ/p)$ is $\pi$-finite, only finitely many degrees occur.
\end{proof}

Our main tool for proving height reduction will be \cref{prop: img of cardinalty}, which describes how semiadditive cardinality behaves under an $F$-character. We start by recalling the definition of semiadditive cardinality and semiadditive height developed in \cite{CSY-ambiheight}.

\begin{definition}
    Let $X\in \CMoninf(\cC)$ and $A\in\spcpi$. Define $|A|_X$ as the composition
    \[
    X\to X^A=\lim_A X\simeq \colim_A X \simeq X[A]\to X
    \]
    where the first map is the diagonal, the second is the inverse of the norm map and the third is the fold map. 
\end{definition}

This definition allows one to define the height of an object in any $\infty$-semiadditive category.

\begin{definition}
    Let $X\in \CMoninf(\cC)$. We say that 
    \begin{enumerate}
        \item $X$ is of height $\leq n$ if $|\B^n C_p|_X$ is invertible. 
        \item $X$ is of height $>n$ if $X$ is $|\B^k C_p|_X$ complete for $k\leq n$. 
    \end{enumerate}
    We denote the full subcategory on objects of height exactly $n$ by $\CMoninf(\cC)_n$. 
\end{definition}

We investigate how $F$-characters interact with semiadditive cardinality.

\begin{proposition}\label{prop: img of cardinalty}
     Let $\cC\in\CAlg(\PrL)$. For $R,S\in \CAlg(\CMoninf(\cC))$,
     \[
     \chi \colon R \to F^*S
     \]
     an $F$-character of commutative algebras and $A\in \spcpi$. Then, $\chi$ sends $|A|_R$ to $|F(A)|_{S}$.
\end{proposition}

\begin{proof}
   By naturality, we have a commutative diagram:
   \[
        \begin{tikzcd}
	       R & {R^A} & R \\
	       {S} & {S^{F(A)}} & {S}
	       \arrow[from=1-1, to=1-2]
	       \arrow["{|A|_R}", shift left, curve={height=-12pt}, from=1-1, to=1-3]
	       \arrow[from=1-1, to=2-1]
	       \arrow[from=1-2, to=1-3]
	       \arrow[from=1-2, to=2-2]
	       \arrow[from=1-3, to=2-3]
	       \arrow[from=2-1, to=2-2]
	       \arrow["{|F(A)|_{S}}"', shift right, curve={height=12pt}, from=2-1, to=2-3]
	       \arrow[from=2-2, to=2-3]
        \end{tikzcd}
    \]
    Since the map $R\to S$ is a ring homomorphism and thus preserves units, the claim follows.
\end{proof}

With the above we are ready to prove our blue shift results.

\begin{lemma}\label{lem: height leq}
    Let $\cC\in\CAlg(\PrL)$, $R,S\in \CAlg(\CMoninf(\cC))$ and 
    \[
    \chi:R \to F^*S
    \]
     an $F$-character of commutative rings. Assume that $R$ is of height $\leq n$ at $p$ and $\dim_{\ZZ/p}(F)=n-t$. Then $S$ is of height $\leq t$ at $p$.
\end{lemma}

\begin{proof}
 We aim to show that $|\B^{t}\ZZ/p|_{S} \in \pi_0(S)$ is invertible. By \cref{prop: img of cardinalty} when evaluated at the point $\chi$ sends
    $|\B^n\ZZ/p|_R$ to $|F(\B^n\ZZ/p)|_{S}$. By \cref{lem: F(BkCp) Cp module} 
    \[
    |F(\B^n\ZZ/p)|_{S}=|A|_{S}|\B^{t}\ZZ/p|_{S}.
    \]
    Since $R$ has height $\leq n$, we know that $|\B^n\ZZ/p|_R$ is invertible and therefore $|\B^t\ZZ/p|_{S}$ is invertible.\\
\end{proof}

\begin{lemma}\label{lem: height geq}
    Let $\cC\in\CAlg(\Mod_{\Sp_{(p)}}(\PrL))$, $R,S\in \CAlg(\CMoninf(\cC))$ and 
    \[
    \chi:R \to F^*S
    \]
     a $F$-character of commutative rings. Assume that $R$ is of height $> n$ at $p$ and $\dim_{\ZZ/p}(F)=n-t$. Then $S$ is of height $>t$ at $p$.
\end{lemma}

\begin{proof}
    By \cite[Theorem 4.2.7]{CSY-ambiheight}, the category $\CMoninf(\cC)$ decomposes as a product
    \[
        \CMoninf(\cC)=\CMoninf(\cC)_{>t}\times \prod_{k=0}^{t}\CMoninf(\cC)_k.
    \]
    We denote by $S_k$ the projection of $S$ onto the factor $\CMoninf(\cC)_k$.
    The statement is thus equivalent to showing that
    \[
        S_k=0, \quad \textrm{for} \quad 0\leq k\leq t.
    \]
    By \cref{prop: img of cardinalty} we have a commutative diagram 
    \[
    \begin{tikzcd}[column sep=4em]
    R &&& R \\
    S_k &&& S_k \\
    & {}
    \arrow["{1-|\B^{n}\ZZ/p|_R|\B^{n+1}\ZZ/p|_R}", from=1-1, to=1-4]
    \arrow[from=1-1, to=2-1]
    \arrow[from=1-4, to=2-4]
    \arrow["{1-|F(\B^{n}\ZZ/p)|_{S_k}|F(\B^{n+1}\ZZ/p)|_{S_k}}"', from=2-1, to=2-4]
    \end{tikzcd}
    \]
    Since $R$ is of height $>n$, using \cite[Proposition 4.1.14]{CSY-ambiheight}, the upper horizontal map is an equivalence. As $F$ is exact, and finite limits in $\spcpi$ are computed in $\spc$, we have that $F(B^{n-1}\ZZ/p)=\Omega F(B^{n}\ZZ/p)$. Since $S_k$ is of height $k \leq t$, it follows from \cite[Proposition 4.2.2]{CSY-ambiheight} that the lower horizontal map is null. As the map $R \to S$ is a map of rings we get that 
    \[
    1-|F(\B^{n}\ZZ/p)|_{S_k}|F(\B^{n+1}\ZZ/p)|_{S_k}
    \] 
    is an invertible element that acts like zero on $S_k$ which implies $S_k=0$.
\end{proof}

We now specialize to $\cC=\Sp$ and adopt the usual notation $\CMoninf(\Sp)=\tsadi$ and $\CMoninf(\Sp)_n = \tsadi_n$. 

\begin{corollary}\label{cor: chr single height}
    Let $R\in \CAlg(\tsadi_n)$, $S\in\CAlg(\tsadi)$, $\dim_{\ZZ/p}(F)=n-t$ and 
    \[
    \chi:R \to F^*S
    \]
    a character of commutative rings. Then, $S\in\tsadi_{t}$.  In particular for any $X\in \tsadi_n$, $F_!X\in \tsadi_{t}$.   
\end{corollary}

\begin{proof}
    The first part follows immediately from \cref{lem: exists of L! and mult st}, \cref{lem: height geq}, and \cref{lem: height leq}. For the second part, apply the first part to the universal character
    \[
    \ounit_{\tsadi_n} \to F^*F_!\ounit_{\tsadi_n}.
    \]
    This shows that $F_!\ounit_{\tsadi_n}\in \tsadi_{t}$. Since $X$ is a module over $F_!\ounit_{\tsadi_n}$, the claim follows.
\end{proof}

\begin{corollary}\label{cor: F chr single height T local}
    Let $R\in \CAlg(\Sp_{T(n)})\subseteq \CAlg(\tsadi_n)$ where we think of $\Sp_{T(n)}$ as a symmetric monoidal full subcategory of $\tsadi_n$ by \cite[Corollary 5.5.14]{CSY-ambiheight}. Let $S\in \CAlg(\tsadi)$,  $\dim_{\ZZ/p}(F)=n-t$ and
    \[
    \chi \colon R \to F^*S
    \]
    be an $F$-character of commutative rings. Then $S\in \Sp_{T(t)}$. In particular, for any $X\in \Sp_{T(n)}$
    $F_!X\in \Sp_{T(t)}$.
\end{corollary}

\begin{proof}
    By \cite[Theorem 1.3]{burklund-multiplicative-moore} there exists a type $n+1$ spectrum, $F(n+1)$ equipped with an $\EE_1$-ring structure. Let $G\colon \Sp \to \tsadi$ denote the free $\infty$-commutative monoid functor, and write
    $\underline{R}$ for the underlying spectrum of $R$. Since $\underline{R}\in \Sp_{T(n)}$, we have $\underline{R}\otimes F(n+1)\simeq 0$, and hence
    \[
    0 \simeq G(\underline{R}\otimes F(n+1))
    \simeq G(\underline{R})\otimes F(n+1)
    \to R\otimes F(n+1),
    \]
    where we use that the functors in sight are spectrum-linear. As this is a map of algebras we conclude that $R\otimes F(n+1)\simeq 0$.
    Tensoring $\chi$ with $F(n+1)$ yields a map of $\EE_1$-rings
    \[
    0 \simeq R\otimes F(n+1) \to F^*S\otimes F(n+1),
    \]
    and thus $F^*S\otimes F(n+1)\simeq 0$. Consequently by \cite[Theorem 5.4.10]{CSY-ambiheight} (see also \cite[Lemma 8.2]{Ben-Moshe-Schlank-2024-K-theory} and the discussion preceding it),
    \[
    S \in \prod_{k=0}^n \Sp_{T(k)}.
    \]
    By \cref{lem: height geq} we moreover have $S\in \tsadi_{t}$, so $S\in \Sp_{T(t)}$. For the second part, apply the first part to the universal character
    \[
    \SS_{T(n)} \to F^*F_!\SS_{T(n)}.
    \]
    We get that $F_!\SS_{T(n)}\in \Sp_{T(t)}$. Since $F_!X$ is a module over $F_!\SS_{T(n)}$, the claim follows.
\end{proof}

\subsection{Base Change}\label{subsec: Base Change}

Our next goal is to establish a base change formula. Namely, for $R\in \CMoninf(\cC)$ and
$A\in \spcpi$, we will construct a natural equivalence
\[
R^{A\times -}\otimes_{R^{(-)}} F^*F_!R \simeq (F_!R)^{F(A)\times F(-)},
\]
functorial in $A$.

The proof will follow from some purely categorical statement about tensor products with respect to Day convolution.

\begin{lemma}\label{lem: tensor left kan}
    Let $\cC\in \CAlg(\PrL)$ and let $I$ be a small category. There is a natural equivalence in
    \[
    \Fun(I{\op},\Fun(\cC,\Fun(I,\cC)))
    \]
    between the functors
    \[
    i \mapsto (X \mapsto \Map(i,-)\otimes X),
    \qquad
    i \mapsto (X \mapsto i_!X),
    \]
    where $i_!X$ denotes the left Kan extension of $X\colon \pt \to \cC$ along $i\colon \pt \to I$.
\end{lemma}

\begin{proof}
Fix $i\in I$ and let $\ev_i\colon \Fun(I,\cC)\to \cC$ be evaluation at $i$.
Since $\cC$ is presentable, $\Fun(I,\cC)$ is presentable and $\ev_i$ admits a left adjoint given by left Kan extension along $i$, namely $i_!\colon \cC\to \Fun(I,\cC)$.

For $X\in \cC$ and $G\in \Fun(I,\cC)$ we have a natural equivalence
\[
\Map_{\Fun(I,\cC)}(\Map(i,-)\otimes X,\,G)\simeq \Map_{\cC}(X,\,G(i)).
\]
Thus the functor $X \mapsto \Map(i,-)\otimes X$ is also left adjoint to $\ev_i$. By uniqueness of left adjoints, there is a canonical equivalence of functors
\[
\Map(i,-)\otimes (-)\simeq i_!(-)\colon \cC\to \Fun(I,\cC).
\]
These equivalences are natural in $i$, and therefore assemble to the claimed equivalence in
$\Fun(I{\op},\Fun(\cC,\Fun(I,\cC)))$.
    
\end{proof}

\begin{definition}
Let $\cC\in \CAlg(\PrL)$. We denote by:
\begin{enumerate}
    \item $\Fun_\cC^L(-,-)$ the category of $\cC$-linear left adjoints,
    \item $\Fun_{\cC}^{L,R}(-,-)$ the category of $\cC$-linear left adjoints whose right adjoint, equipped with its induced lax $\cC$-linear structure, is strong,
    \item $\Fun_{\cC}^{R,L}(-,-)$ the category of $\cC$-linear right adjoints whose left adjoint, equipped with its induced op-lax $\cC$-linear structure, is strong.
\end{enumerate}
Note that passing to right adjoints induces an equivalence
\[
\Fun_{\cC}^{L,R}(\cD,\cE){\op}\simeq \Fun_{\cC}^{R,L}(\cE,\cD).
\]
\end{definition}

\begin{lemma}\label{lem: pointwise left kan otimes}
Let $\cC\in \CAlg(\PrL)$ and let $I\in \CAlg(\Cat)$. Endow $\Fun(I,\cC)$ with the Day convolution symmetric monoidal structure. Then for any $G\in \Fun(I,\cC)$ and any $i\in I$ there is an equivalence
\[
(\Map_I(i,-)\otimes \ounit_{\cC})\otimes G \simeq (i\otimes(-))_!G.
\]
\end{lemma}

\begin{proof}
Let $\mu=\otimes_I\colon I\times I\to I$ denote the monoidal product. Recall that Day convolution on $\Fun(I,\cC)$ is given by
\[
A\otimes B \simeq \mu_!(A\boxtimes B),
\qquad
(A\boxtimes B)(a,b)=A(a)\otimes_{\cC}B(b),
\]
where $\boxtimes$ denotes the external tensor product.

Fix $i\in I$ and write 
\[
s=i\times \id_I\colon I\simeq \pt\times I\to I\times I.
\]
Then $\mu\circ s=i\otimes(-)$. Since $\boxtimes$ preserves colimits separately in each variable, there is a natural equivalence
\[
(i_!\ounit_{\cC})\boxtimes G \simeq s_!(\ounit_{\cC}\boxtimes G).
\]
Using \cref{lem: tensor left kan}, we identify $\Map_I(i,-)\otimes \ounit_{\cC}\simeq i_!\ounit_{\cC}$, and therefore
\begin{align*}
(\Map_I(i,-)\otimes \ounit_{\cC})\otimes G
&\simeq \mu_!((\Map_I(i,-)\otimes \ounit_{\cC})\boxtimes G) \\
&\simeq \mu_!((i_!\ounit_{\cC})\boxtimes G) \\
&\simeq \mu_!(s_!(\ounit_{\cC}\boxtimes G)) \\
&\simeq (i\otimes(-))_!(\ounit_{\cC}\boxtimes G) \\
&\simeq (i\otimes(-))_!G,
\end{align*}
where the last step uses the canonical identification $\ounit_{\cC}\boxtimes G\simeq G$.
\end{proof}

\begin{lemma}\label{lem:Left-Kan-otimes-Day}
Let $\cC\in \CAlg(\PrL)$ and let $I\in \CAlg(\Cat)$. Endow $\Fun(I,\cC)$ with the Day convolution symmetric monoidal structure. Then there is a natural equivalence of symmetric monoidal functors in
\[
\Fun(I{\op},\Fun(I,\cC))\simeq
\Fun(I{\op},\Fun_{\Fun(I,\cC)}^{L}(\Fun(I,\cC),\Fun(I,\cC))),\footnote{Note that the symmetric monoidal structure on $\Fun_{\Fun(I,\cC)}^{L}(\Fun(I,\cC),\Fun(I,\cC))$ is given by composition.}
\]
between the functors
\[
i \mapsto (\Map_I(i,-)\otimes \ounit_{\cC})\otimes(-)
\qquad \text{and} \qquad
i \mapsto (i\otimes(-))_!(-).
\]
\end{lemma}

\begin{proof}
We check that the equivalence is compatible with the symmetric monoidal structures. By \cite[Corollary 4.8.1.12]{Lurie-HA}, there is a symmetric monoidal functor
\[
\Mod_{I{\op}}(\Cat) \to \Mod_{\Fun(I,\cC)}(\PrL),
\qquad
M\mapsto \Fun(M{\op},\cC).
\]
Applying this to the endomorphism object of $I{\op}$ yields a symmetric monoidal map
\[
I{\op}\simeq \Fun_{I{\op}}(I{\op},I{\op})
\to
\Fun_{\Fun(I,\cC)}^{L}(\Fun(I,\cC),\Fun(I,\cC))
\simeq \Fun(I,\cC).
\]
By construction, an object $i\in I$ is sent to the left Kan extension functor $(i\otimes(-))_!$. Moreover, the underlying functor of this symmetric monoidal embedding agrees with the functor
\[
i \mapsto (\Map_I(i,-)\otimes \ounit_{\cC})\otimes(-)
\]
by \cref{lem: pointwise left kan otimes}.

Let $\cD\subset \Fun_{\Fun(I,\cC)}^{L}(\Fun(I,\cC),\Fun(I,\cC))$ be the full subcategory spanned by the endofunctors
\[
(\Map_I(i,-)\otimes \ounit_{\cC})\otimes(-)
\]
for $i\in I$. It contains the unit and is closed under tensor product. Hence there is a unique symmetric monoidal structure on $\cD$ for which the inclusion $\cD\hookrightarrow \Fun_{\Fun(I,\cC)}^{L}(\Fun(I,\cC),\Fun(I,\cC))$ lifts to a symmetric monoidal functor. It follows that the two symmetric monoidal functors in the statement are naturally equivalent.
\end{proof}

\begin{remark}
    Assume that every object of $I$ is dualizable. Then \cref{lem:Left-Kan-otimes-Day} upgrades to an equivalence of symmetric monoidal functors
    \[
    I\op \to \Fun(I,\cC)\dbl\simeq \Fun_{\Fun(I,\cC)}^{R,L}(\Fun(I,\cC),\Fun(I,\cC)).
    \]
    Here the symmetric monoidal structure on $\Fun_{\Fun(I,\cC)}^{R,L}(\Fun(I,\cC),\Fun(I,\cC))$ is given by composition.
\end{remark}

\begin{lemma}\label{lem: representable_tensor_dual}
Let $\cC\in \CAlg(\PrL)$ and let $I\in \CAlg(\Cat)$ and assume every object in $I$ is dualizable. Equip
$\Fun(I,\cC)$ with its Day convolution symmetric monoidal structure. Then, for every
$i\in I$, tensoring with the representable object $\Map_I(i,-)\otimes \ounit_{\cC}\in \Fun(I,\cC)$
is canonically equivalent to left Kan extension along the functor $i\otimes(-)\colon I\to I$.
That is, there is a natural equivalence of endofunctors
\[
\bigl(\Map_I(i,-)\otimes \ounit_{\cC}\bigr)\otimes(-)
\simeq
(i\otimes(-))_!(-).
\]

These equivalences assemble naturally in $i\in I$ and extend to an equivalence of symmetric monoidal functors
\[
I{\op}\to
\Fun_{\Fun(I,\cC)}^L\bigl(\Fun(I,\cC),\Fun(I,\cC)\bigr).
\]
Here the symmetric monoidal structure on
\[
\Fun_{\Fun(I,\cC)}^L\bigl(\Fun(I,\cC),\Fun(I,\cC)\bigr)
\]
is given by composition.
\end{lemma}


\begin{proof}
    Taking right adjoints to \cref{lem:Left-Kan-otimes-Day} we get an equivalence of symmetric monoidal functors between
    \begin{gather*}
    I \to \Fun^{R,L}(\Fun(I,\cC),\Fun(I,\cC)), \\ 
    i \mapsto (\Map_I(i^{\vee},-)\otimes \ounit_{\cC})\otimes(-) \quad \textrm{and} \quad  
    i\mapsto (-)\circ(i\otimes (-)).
        \end{gather*}

    Precomposing with $I\op\xto{(-)^\vee} I$ gives the desired claim.
\end{proof}

\begin{remark}\label{rem: R to is symmatric monoidal}
Note that it follows from \cref{lem: representable_tensor_dual} that for $R\in \CAlg(\CMoninf(\cC))$ the functor 
\[
\spcpi \to \Mod_R(\PCMoninf), \quad A \mapsto R^{A \times (-)}
\]
is symmetric monoidal.
\end{remark}

We are now finally ready to prove our base change result.

\begin{corollary}\label{cor: base change}
Let $R\in \CAlg(\CMoninf(\cC))$ and $A\in \spcpi$. Then there is a natural equivalence of symmetric monoidal functors
\begin{gather*}
\Span(\spcpi)\op\to \Mod_{F^*F_!(R)}(\PCMoninf(\cC))\simeq \End_{\Mod_{F^*F_!(R)}(\PCMoninf(\cC))}^L(\Mod_{F^*F_!(R)}(\PCMoninf(\cC)),\\
\quad\quad A \mapsto (G\mapsto R^{A\times -}\otimes_{R^{(-)}} G) \quad \textrm{and} \quad A\mapsto (G\mapsto G({A\times (-)})).
\end{gather*}
\end{corollary}

\begin{proof}
Every object of $\Span(\spcpi)$ is canonically self-dual. Hence, applying \cref{lem: representable_tensor_dual} to the object $A$, we obtain an identification
\[
R^{A\times -}\simeq (\Map_{\Span(\spcpi)}(A,-)\otimes \ounit_\cC)\otimes R^{(-)},
\]
where $\otimes$ denotes Day convolution.

Now compute for $G\in \Mod_{F^*F_!(R)}(\PCMoninf(\cC))$:
\begin{align*}
R^{A\times -}\otimes_{R^{(-)}} G
&\simeq ((\Map(A,-)\otimes \ounit_\cC)\otimes R^{(-)})\otimes_{R^{(-)}} G \\
&\simeq (\Map(A,-)\otimes \ounit_\cC)\otimes G \\
&\simeq G(A\times -).
\end{align*}
Since the isomorphism in \cref{lem: representable_tensor_dual} is symmetric monoidal this  equivalence is symmetric monoidal as well. 
\end{proof}

\begin{proposition}\label{thm: F! of R A}
Let $R\in \CAlg(\CMoninf(\cC))$. Then there exists an equivalence of functors
\[
F_!(R^{A\times (-)})\simeq (F_!R)^{F(A)\times (-)}.
\] 
functorial in $A$. That is an equivalence of symmetric monoidal functors
\[
\Span(\spcpi)\op \to  \Mod_{F_!(R)}(\CMoninf(\cC)).
\]
\end{proposition}

\begin{proof}
    By \cref{lem: representable_tensor_dual} and \cref{lem: tensor left kan} we have
    \[
    R^{A\times (-)} \simeq R^{(-)} \otimes (\Map(A,-)\otimes \ounit_\cC)\simeq R^{(-)} \otimes A_!\ounit_\cC.
    \]
    As $F_!$ is symmetric monoidal, we get 
    \[
    F_!(R^{A\times (-)})=F_!R^{(-)} \otimes F(A)_!\ounit_\cC\simeq F_!(R^{(-)}) \otimes (\Map(F(A),-)\otimes \ounit_\cC)\simeq F_!(R)^{F(A)\times (-)}.
    \]
\end{proof}




\begin{remark}
    None of the chromatic applications in this paper depend essentially on \cref{thm: F! of R A}, in fact, an earlier draft was written without it. Instead, one may argue directly, as in \cref{rem: char root of unity}, to identify the character of higher roots of unity. This is the main ingredient needed to establish compatibility of the character with the semiadditive Fourier transform, which already suffices to identify the transchromatic character with the universal character of $\En$. In the chromatic case, taking $F=\L_p^{n-t}$, one may then recover the base-change result of \cref{thm: F! of R A} a posteriori.
\end{remark}


\subsection{Cyclotomic Blue-Shift}\label{sec: Cyclotomic Blue-Shift}

The goal of this subsection is to show that, under mild hypotheses on $F$, the universal $F$-character is compatible with the formation of higher cyclotomic extensions. We begin with a brief review of higher cyclotomic extensions, then record the additional assumptions we will impose on $F$, and finally prove that $F_!$ sends higher cyclotomic extensions to higher cyclotomic extensions up to a height-dependent shift.

To that end, we recall the construction of higher cyclotomic extensions from \cite[Section 4.1]{BMCSY-cycloshift}. Assume that $\cD$ is a presentably symmetric monoidal, stable, $\infty$-semiadditive category of semiadditive height $n$ (for example, $\Sp_{T(n)}$). The map $\Sigma^n \ZZ/p \to 0$ induces a canonical splitting in commutative algebras 
\[
\ounit_{\cD}[\Sigma^n \ZZ/p]\simeq \ounit_{\cD}\times \ounit_{\cD}[\omega_p^{(n)}]
\]
which is $(\ZZ/p^r)\units$ equivariant. The factor $\ounit_{\cD}[\omega_p^{(n)}]$ is called the height $n$ $p$-th cyclotomic extension.

If $M\in \Sp_{\ge 0}$ is equipped with a map $\Sigma^n \ZZ/p \to M$, then $\ounit_{\cD}[M]$ is naturally an algebra over $\ounit_{\cD}[\Sigma^n \ZZ/p]$.

\begin{definition}
    Let $X\in \cD$ and let $M\in \Sp_{\ge 0}$ be under $\Sigma^n \ZZ/p$. Define the two base changes
    \[
    X[M]_0 := X\otimes_{\ounit_{\cD}[\Sigma^n \ZZ/p]} \ounit_{\cD},
    \qquad
    X[M]_\omega := X\otimes_{\ounit_{\cD}[\Sigma^n \ZZ/p]} \ounit_{\cD}[\omega_p^{(n)}].
    \]
    Then $X[M]:=X\otimes \ounit_{\cD}[M]$ decomposes as a product
    \[
    X[M]\simeq X[M]_0 \times X[M]_\omega.
    \]
    This splitting is functorial in $M$ (under $\Sigma^n \ZZ/p$) and is compatible with the monoidal structure in $X$.
\end{definition}

\begin{definition} \label{def: F orientable}
    Assume that $\dim_{\ZZ/p^r}(F)$ is finite. We say that $F$ is $\ZZ/p^r$-orientable if there exists an isomorphism  $\pi_{n-\dim_{\ZZ/p^r}(F)}F(\B^n\ZZ/p^r) \simeq \ZZ/p^r$ for any $n\geq \dim_{\ZZ/p^r}(F)$. We call a choice of such isomorphisms a $\ZZ/p^r$ orientation of $F$. 
\end{definition}

\begin{example}
    Let $F=\Map(M,-)$, where $M$ is an $(n-t)$-dimensional closed connected manifold. Then $\dim_{\ZZ/p^r}(F)=n-t$, and
    \[
    \pi_{t}(F(\B^{n}\ZZ/p))\simeq H^{n-t}(M,\ZZ/p^r).
    \]
    In particular, $F$ is $\ZZ/p^r$-orientable if and only if $M$ is $\ZZ/p^r$-orientable.  
\end{example}

\begin{lemma}\label{lem: F(B^nZ/pr) splits}
    Assume that $\dim_{\ZZ/p^r}(F)=n-t$ and that it is $\ZZ/p^r$-orientable. Then a choice of an orientation for $F$ is equivalent to a choice of a splitting
    \[
    F(\B^{n}\ZZ/p^r)\simeq \B^{t}\ZZ/p^r \times \tau_{\geq t+1}F(\B^{n}\ZZ/p^r).
    \]
\end{lemma}

\begin{proof}
    Since $F$ commutes with finite coproducts and is symmetric monoidal, $F(\B^{n}\ZZ/p^r)$ is canonically a module over $F(\ZZ/p^r)\simeq \ZZ/p^r$. By assumption, $\pi_t(F(\B^{n}\ZZ/p^r))\simeq \ZZ/p^r$. Giving such an isomorphism is equivalent to choosing a generator, i.e. a map of $\ZZ/p^r$ modules
    \[
    \Sigma^{t}\ZZ/p^r \to F(\B^{n}\ZZ/p^r)
    \]
    which is an isomorphism on $\pi_t$. Such a choice determines a section of the canonical map
    \[
    F(\B^{n}\ZZ/p^r) \to \Sigma^{t}\ZZ/p^r,
    \]
    and any such section splits the canonical fiber sequence
    \[
    \tau_{\ge t+1}F(\B^{n}\ZZ/p^r) \to F(\B^{n}\ZZ/p^r) \to \Sigma^{t}\ZZ/p^r,
    \]
    yielding the desired splitting.
\end{proof}
 
We are now ready to prove that $F_!$ sends higher cyclotomic extensions to higher cyclotomic extensions up to a height-dependent shift. We will first show it behaves well with taking group algebra.

\begin{remark}
Let $R\in \CAlg(\CMoninf(\cC))$ and let $A\in \spcpi$ ($A\in \Sp^{\pi\text{-}\fin,\mrm{cn}}$).
Throughout, the notation $R[A]$ denotes the constant colimit (the associated group algebra) of $A$ in $\CMoninf(\cC)$.
\end{remark}

\begin{corollary}\label{cor: group_algebra_equivariant_equiv}
Let $R\in \CAlg(\CMoninf(\cC))$. Then there is an equivalence in
\[
\Fun(\Sp_{(p)}^{\pi-\fin,\mrm{cn}},\CAlg_{F_!(R)}(\CMoninf(\cC)))
\]
between the functors
\[
M \mapsto F_!(R[M]) \quad \textrm{and} \quad M \mapsto F_!(R)[F(M)].
\]
In particular, if $R$ is of height $n$ and $F$ is of dimension $\dim_{\ZZ/p^r}(F)=n-t$ and $\ZZ/p^r$-oriented, then
\[
F_!(R[\B^{n}(\ZZ/p^r)^\times]^{(-)})\simeq
F_!(R^{(-)})[\B^{t}\ZZ/p^r]
\quad \in\;
\CAlg(\Mod_{F_!(R)}(\CMoninf(\cC)))^{\B(\ZZ/p^r)^\times}.
\]
\end{corollary}

\begin{proof}
By \cref{rem: R to is symmatric monoidal}, for any $A\in \spcpi$ the duality data of $R[A]$ and $R^{A}$ in $\Mod_R(\CMoninf(\cC))$ determines duality data in $\Mod_R(\PCMoninf(\cC))$. Since the functor
\[
F_!\colon \Mod_R(\PCMoninf(\cC))\to \Mod_{F_!(R)}(\CMoninf(\cC))
\]
is symmetric monoidal, the first statement follows from \cref{thm: F! of R A} by taking duals and restricting along the symmetric monoidal functor $\spcpi\to \Span(\spcpi)$.

For the second statement, by \cref{lem: F(B^nZ/pr) splits} we have a splitting
\[
F(\B^{n}\ZZ/p^r)\simeq \B^{t}\ZZ/p^r \times \tau_{\ge t+1}F(\B^{n}\ZZ/p^r).
\]
By \cref{lem: height leq}, the object $F_!(R)$ has height $t$, hence
\[
F_!(R)[\tau_{\ge t+1}F(\B^{n}\ZZ/p^r)]\simeq F_!(R).
\]
Therefore
\[
F_!(R)[F(\B^{n}\ZZ/p^r)]\simeq F_!(R)[\B^{t}\ZZ/p^r],
\]
and this gives the claimed equivalence.
\end{proof}

We now turn to show that the equivalence of \cref{cor: group_algebra_equivariant_equiv} is compatible with the splittings defining the height $n$ and height $t$ $p^r$-th cyclotomic extensions. This is essentially a consequence of the following proposition from \cite{BMCSY-cycloshift}.

\begin{proposition} \cite[Lemma 4.1]{BMCSY-cycloshift}\label{prop: compering decompositions}
    Let $\cC$ be a symmetric monoidal additive category, and let $R \in \CAlg(\cC)$. Assume that we are given two decompositions
    \[
    R \simeq R_1 \times R_2, \quad\quad R \simeq \overline{R_1} \times \overline{R_2}.
    \]
    in $\CAlg(\cC)$ and an isomorphism of $R_1$ and $\overline{R_1}$ under $R$. Then, there is an isomorphism of $R_2$ and $\overline{R_2}$ under R as well, namely an isomorphism of decompositions.
\end{proposition}

\begin{proof}
    The claim in \cite[Lemma 4.1]{BMCSY-cycloshift}, assumes stability but the proof works verbatim only assuming additivity.
\end{proof}

\begin{proposition}\cite[Lemma 4.3]{BMCSY-cycloshift} \label{prop: decomposition ZZ/p suffices}
    Let $I$ be a category with an initial object $i$ and let $\cC \in \CAlg(\PrL)$ be $0$-semiadditive.
    Then, a decomposition of a functor $F: I \to \CAlg(\cC)$ into a product is the same data as a map $\ounit^2_{\cC} \to F(i)$.
\end{proposition}

\begin{theorem}\label{thm:cyclo blue shift}
    Assume $\cC$ is additive, $R\in \CAlg(\CMoninf(\cC))$ of height $n$ and $\dim_{\ZZ/p^r}(F)=n-t$ and that $F$ is $\ZZ/p^r$ orientable. Then there is an  equivalence
    \[
    F_!(R[\omega^{(n)}_{p^r}]) \simeq F_!(R)[\omega^{(t)}_{p^r}]\in \CAlg(\CMoninf(\cC))^{(\ZZ/p^r)\units}.
    \]
\end{theorem}

\begin{proof}
    We need to check that the equivalence from \cref{cor: group_algebra_equivariant_equiv} 
    \[
    (M \mapsto F_!(R[M]^{(-)})\simeq (M \mapsto  F_!(R)[F(M)]^{(-)})
    \]
    of functors 
    \[
    \Sp^{\pi-\fin, \mrm{cn}}_{(\Sigma^n \ZZ/p)/} \to \CAlg_{F_!(R)}(\CMoninf(\cC))
    \]
    respects the decomposition on the source and the target. By $\cref{prop: decomposition ZZ/p suffices}$ it suffices to check this on $\Sigma^n \ZZ/p$.  Applying the equivalence of \cref{cor: group_algebra_equivariant_equiv} to $\Sigma^n \ZZ/p \to 0$ we get the commutative diagram 
    \[
    \begin{tikzcd}
	{F_!(R)[\Sigma^t \ZZ/p]} & {F_!(R[\Sigma^n \ZZ/p])} \\
	{F_!R} & {F_!R}
	\arrow["\sim", from=1-1, to=1-2]
	\arrow[from=1-1, to=2-1]
	\arrow[from=1-2, to=2-2]
	\arrow["\sim"', from=2-1, to=2-2]
    \end{tikzcd}
    \]
    Thus, the two decompositions coincide by \cref{prop: compering decompositions}. 
\end{proof}

\begin{remark}
    Since the composite functor
    \[
    \CMoninf(\cC)\xrightarrow{i}\PCMoninf(\cC)\xrightarrow{F_!}\CMoninf(\cC)
    \]
    does not commute with filtered colimits, there is no a priori reason to expect an equivalence
    \[
    F_!(R[\omega^{(n)}_{p^\infty}]) \simeq F_!(R)[\omega^{(t)}_{p^\infty}].
    \]
    In fact, this equivalence fails in general: we will see in \cref{sec: SK(n) universal character} that it is false for $R=\SS_{K(1)}$ and $F=\L$ the free loop space functor.
\end{remark}

\section{Fourier Transform and the Character of Higher Roots of Unity}\label{sec: Fourier Transform and the Character of Higher Roots of Unity}

In this section we study how $F$-characters interact with the semiadditive Fourier transform. 
We start in \cref{subsec: mul structure} by studying the multiplicative structure of $F$-characters, explaining how multiplicative structure on the ``representation'' induces additional structure on the associated character. In \cref{sec: Rognes-Galois Extensions} we show that the
character of a height $n$ primitive root of unity is again primitive, now of height $t$ (see \cref{cor: character roots of unity}), and we use this to show that a Galois cyclotomic extension of $R$ induces a Galois cyclotomic extension of $F_!(R)$. Finally, in
\cref{sec: Fourier Transform} we explain how to express the semiadditive Fourier transform of an $F$-character of $R$ in terms of the semiadditive Fourier transform of $R$ itself (see
\cref{prop: Fourier in chr}).

We continue under the assumptions of the previous section. Namely, we fix a prime $p$, a presentably symmetric monoidal $p$-local additive category $\cC\in \CAlg(\Pr^{\mrm{L},\mrm{ad}}_{(p)})$, and an exact endofunctor
\[
F\colon \spcpi \to \spcpi
\]
which commutes with finite coproducts.

\subsection{Multiplicative Structure}\label{subsec: mul structure}

We address the following question: given an $F$-character of commutative algebras $R \to F^{*}S$, what structure does a multiplicative structure on a map
\[
  A \to R, \qquad A \in \CMoninf(\spcpi)
\]
induce on the $F$-character of $A$:
\[
  F(A) \to S.
\]

We fix the following notation.

\begin{notation}
    Let $\cC$ be a semiadditive category.
    
    \begin{enumerate}
        
        \item For $R\in \CAlg(\CMoninf(\cC))$ we denote by $\underline{R}_\times\in \CMon(\spc)$ the underlying commutative monoid with respect to multiplication.

        \item For $X\in \CMoninf(\cC)$ we denote by $\underline{X}_+\in \Sp\cn$ the underlying connective spectrum with respect to addition.
        
    \end{enumerate}
\end{notation}

\begin{proposition}\label{prop: L! maps of Sp to maps of Sp}
    Let $R\in \CAlg(\CMoninf(\cC))$ and $f:A \to \underline{R}_\times$ be a map of commutative monoids. Let 
    \[
    \chi \colon R \to F^*S
    \]
    be an $F$-character of commutative algebras. Then the image of $f$ under $\chi$ has a natural structure of a map of commutative monoids 
    \[
    F(A) \to \underline{S}_\times 
    \]
\end{proposition}

\begin{proof}
    Precompose with the symmetric monoidal functor
    \[
        \spcpi \to \Span(\spcpi)
    \]
    which is the identity on objects and sends a morphism $A\to B$ to the span
    \[
        \begin{tikzcd}
            & A \\
            A && B
            \arrow["\id"', from=1-2, to=2-1]
            \arrow[from=1-2, to=2-3].
        \end{tikzcd}
    \]
    After postcomposing with $\Omega^\infty$, the $F$-character
    $\chi\colon R\to F^*S$ induces a morphism
    \[
        \underline{R}_\times \to \underline{F^*S}_\times
        \qquad
        \in \CAlg(\FunDay(\spcpi,\spc)).
    \]

    Applying symmetric monoidal unstraightening
    \cite[Theorem A]{monoidal-unstr}, we obtain a symmetric monoidal functor
    \[
        \mrm{Un}^\otimes(\underline{R}_\times)
        \to
        \mrm{Un}^\otimes(\underline{F^*S}_\times).
    \]
    We identify the source with
    \[
        \mrm{Un}^\otimes(\underline{R}_\times)
        \simeq
        \overCat{\spcpi}{\underline{R}_\times}.
    \]
    The target is identified by base change. Namely, since
    $\underline{F^*S}_\times$ is the restriction of $\underline{S}_\times$
    along
    \[
        F\colon \spcpi \to \spcpi ,
    \]
    its unstraightening fits into a pullback square
    \[
        \begin{tikzcd}
            \mrm{Un}^\otimes(\underline{F^*S}_\times)
            \arrow[r]
            \arrow[d]
            &
            \overCat{\spcpi}{\underline{S}_\times}
            \arrow[d]
            \\
            \spcpi
            \arrow[r,"F"']
            &
            \spcpi .
        \end{tikzcd}
    \]
    Therefore the preceding symmetric monoidal functor may be written as
    \[
        \overCat{\spcpi}{\underline{R}_\times}
        \longrightarrow
        \spcpi
        \times_{\spcpi,F}
        \overCat{\spcpi}{\underline{S}_\times}.
    \]
    Composing with the projection to
    $\overCat{\spcpi}{\underline{S}_\times}$ gives a symmetric monoidal functor
    \[
        \overCat{\spcpi}{\underline{R}_\times}
        \xto{F(-)}
        \overCat{\spcpi}{\underline{S}_\times}.
    \]

    The symmetric monoidal structures on these slice categories are given by
    taking cartesian products of maps and then composing with the multiplication
    on the corresponding commutative algebra. Hence this symmetric monoidal
    functor induces a functor on commutative monoid objects
    \[
        \CMon\!\left(\overCat{\spcpi}{\underline{R}_\times}\right)
        \xto{F(-)}
        \CMon\!\left(\overCat{\spcpi}{\underline{S}_\times}\right),
    \]
    which gives the claim.
\end{proof}

\begin{proposition}
    Let $X\in \CMoninf(\cC)$ and $f:A \to \underline{X}_+$ be a map of commutative monoids in spaces. Let 
    \[
    \chi \colon X \to F^*Y
    \]
    be an $F$-character. Then the image of $f$ under $\chi$ has a natural structure of a map of commutative monoids 
    \[
    F(A) \to \underline{Y}_+ 
    \]
\end{proposition}

\begin{proof}
    Endow $\Span(\spcpi)$ with the symmetric monoidal structure induced by coproducts of
    $\pi$-finite spaces, and endow $\cC$ with the symmetric monoidal structure given by the categorical product. With these choices, the category of symmetric monoidal functors $\Span(\spcpi)\to \cC$ identifies with the full subcategory spanned by the product-preserving functors. It follows that the morphism $X \to F^{*}Y$ defines a map in $\CAlg(\PCMoninf(\cC))$, where $\PCMoninf(\cC)$ is equipped with the Day convolution monoidal structure determined by the above symmetric monoidal structures on the source and target. The remainder of the proof is analogous to the proof of \cref{prop: L! maps of Sp to maps of Sp}.
\end{proof}

\subsection{The Character of Higher Roots of Unity and Galois Extensions}\label{sec: Rognes-Galois Extensions}

In this section we study the character of higher roots of unity, and use it to show that if $R[\omega^{(n)}_{p^r}]$ is a $(\ZZ/p^r)^\times$ Galois extension of $R$, then
$F_!(R)[\omega^{(t)}_{p^r}]$ is a $(\ZZ/p^r)^\times$ Galois extension of $F_!(R)$.

When $R$ is $T(n)$-local, the hypothesis is vacuous and the conclusion is automatic: by \cref{cor: F chr single height T local}, $F_!(R)$ is a $T(t)$-local commutative algebra, and the
desired conclusion holds for any $T(t)$-local commutative algebra by \cite[Proposition 6.11]{BCSY-Fourier}. Accordingly, the only results from this section that will
be used in the remainder of the paper are the constructions from the first part, namely \cref{cons: root of unity candidate} and \cref{cor: character roots of unity}.

Recall that a height $n$ $p^r$-th root of unity in $R$ is a map of spectra
\[
\zeta \colon \Sigma^n \ZZ/p^r \to R\units .
\]
We call $\zeta$ primitive if for every map of commutative algebras $g\colon R \to T$,
the existence of a commutative diagram
\[
\begin{tikzcd}
\Sigma^n \ZZ/p^r \ar[r, "\zeta"] \ar[d] & R\units \ar[d, "g"] \\
\Sigma^n \ZZ/p^{r-1} \ar[r] & T\units
\end{tikzcd}
\]
forces $T \simeq 0$.

Fix $R$ of height $n$. The functor sending an $R$-algebra
$S \in \CAlg_R(\CMoninf(\cC))$ to the space of primitive height $n$ $p^r$-th roots of unity
is represented by $R[\omega^{(n)}_{p^r}]$.

\begin{construction}\label{cons: root of unity candidate}
    Let $R,S \in \CAlg(\CMoninf(\cC))$. Assume that $R$ is of height $n$ and that $F$ is $\ZZ/p^r$-oriented of dimension $\dim_{\ZZ/p^r}(F)=n-t$. Let
    \[
    \chi \colon R \to F^*S
    \]
    be a character map of commutative algebras. Evaluating $\chi$ on $\B^n\ZZ/p^r$ yields a map
    \[
    R^{\B^n\ZZ/p^r} \to S^{F(\B^n\ZZ/p^r)} .
    \]
    Since $S$ is of height $t$ (see \cref{cor: chr single height}) and using
    \cref{lem: F(B^nZ/pr) splits}, we have an equivalence
    \[
    S^{F(\B^n\ZZ/p^r)} \simeq S^{\B^t\ZZ/p^r} .
    \]
    Therefore, by \cref{prop: L! maps of Sp to maps of Sp}, a height $n$ $p^r$-th root of unity
    \[
    \omega^{(n)}_{p^r}\colon \Sigma^n\ZZ/p^r \to R\units
    \]
    determines a height $t$ $p^r$-th root of unity in $S$,
    \[
    \chi(\omega^{(n)}_{p^r}) \colon \Sigma^t \ZZ/p^r \to S\units .
    \]
\end{construction}

It follows directly from the construction of the comparison map in \cref{thm:cyclo blue shift}
that \cref{cons: root of unity candidate} sends primitive roots of unity to primitive roots of unity.

\begin{corollary}\label{cor: character roots of unity}
    Let $0\neq R,S \in \CAlg(\CMoninf(\cC))$, and assume that $\cC$ is stable. Suppose that $R$ has height $n$ and let
    \[
    \chi \colon R \to F^*S
    \]
    be a character map of commutative algebras. Assume that $F$ is $\ZZ/p^r$-oriented of dimension
    $\dim_{\ZZ/p^r}(F)=n-t$. If
    \[
    \omega^{(n)}_{p^r}\colon \Sigma^n\ZZ/p^r \to R\units
    \]
    is a primitive height $n$ $p^r$-th root of unity, then the induced height $t$ $p^r$-th root of unity
    \[
    \chi(\omega^{(n)}_{p^r})\colon \Sigma^t\ZZ/p^r \to S\units
    \]
    constructed in \cref{cons: root of unity candidate} is primitive.
\end{corollary}

\begin{proof}
    We reduce to the universal case. Set $R=\ounit[\omega^{(n)}_{p^r}]$, where $\ounit$ denotes the unit
    object in the full subcategory of height $n$ objects in $\CMoninf(\cC)$, and let
    \[
    S = F_!(\ounit[\omega^{(n)}_{p^r}]).
    \]
    By \cref{thm:cyclo blue shift} there is an equivalence
    \[
    S \simeq F_!(\ounit)[\omega^{(t)}_{p^r}].
    \]
    Moreover, the equivalence of \cref{thm:cyclo blue shift}, is induced by the identification in
    \cref{cor: group_algebra_equivariant_equiv}. Tracing through this identification shows that the image
    under $\chi$ of the canonical map
    \[
    \Sigma^n\ZZ/p^r \to \ounit[\Sigma^n\ZZ/p^r]
    \]
    is identified with the canonical map
    \[
    F(\Sigma^n\ZZ/p^r) \to F_!(\ounit)[F(\Sigma^n\ZZ/p^r)].
    \]
    In particular, the induced root of unity in $S$ is the universal one, hence primitive by definition.
\end{proof}

\begin{remark}\label{rem: char root of unity}
In the special case $\cC=\Sp$ and $R\in \Sp_{T(n)}$, one can prove \cref{cor: character roots of unity} by a more direct argument.

First, one may reduce to the case $r=1$ using the map $\ZZ/p \to \ZZ/p^r$.
Next, using the Chromatic Nullstellensatz, one can show that primitivity of a root of unity in $S$
can be detected after composing with every geometric point $S \to E_t(k)$.
Finally, one shows that the integral of a primitive height $n$ $p^r$-th root of unity vanishes.
It follows that the semiadditive integral of the candidate root of unity constructed in
\cref{cons: root of unity candidate} also vanishes after composing with every geometric point
$S \to E_t(k)$. Since maps $\Sigma^t\ZZ/p \to E_t(k)\units$ are classified by
\[
\ZZ/p \simeq \Map(\Sigma^t\ZZ/p, \I^{(t)}_p),
\]
this vanishing criterion forces the candidate root of unity to be primitive.
\end{remark}

We now turn to show that $F_!(R)[\omega^{(t)}_{p^r}]$ is a Galois extension. We first recall the definition of Galois extension.

\begin{definition}
Let $\B G$ be a pointed connected space, let $\Delta\colon \B G \to \B G\times \B G$ denote the diagonal map, and let $q\colon \B G \to \pt$ denote the canonical map.
A $G$-equivariant commutative algebra
\[
R\colon \B G \to \mathrm{CAlg}(\mathcal{C})
\]
is \emph{$G$-Galois} if it satisfies the following two conditions:
\begin{enumerate}
\item The mate of the unit map $\mathbbm{1}\to q_*R=:R^{hG}$ is an equivalence.
\item The mate of the multiplication map $R\otimes R \to \Delta_*R =: \prod_{G} R$
      is an equivalence.
\end{enumerate}
\end{definition}

We now prove a lemma that will imply the second Galois condition.

\begin{lemma}\label{lem: tensor product of cyc extn}
Let $\cC$ be stable and $R\in \CAlg(\CMoninf(\cC))$. Suppose that $R[\omega^{(n)}_{p^r}]$ is a
$(\ZZ/p^r)\units$ Galois extension of $R$ in $\CAlg(\CMoninf(\cC))$.
Then $R[\omega^{(n)}_{p^r}]$ is also a $(\ZZ/p^r)\units$ Galois extension of $R$ when regarded as an object of $\CAlg(\PCMoninf(\cC))$.
\end{lemma}

\begin{proof}
Write $A = R[\omega^{(n)}_{p^r}]$. By assumption, the Galois condition in $\CAlg(\CMoninf(\cC))$ gives an equivalence
\[
A \widehat{\otimes}_R A \simeq A^{(\ZZ/p^r)^\times},
\]
where the completed tensor product is formed in $\CMoninf(\cC)$.

By \cref{rem: R to is symmatric monoidal}, the inclusion
\[
i\colon \Mod_R(\CMoninf(\cC))\to \Mod_R(\PCMoninf(\cC))
\]
is strong symmetric monoidal on objects of the form $R[\B^n\ZZ/p^r]^{(-)}$.
Since $A$ is a retract of $R[\B^n\ZZ/p^r]^{(-)}$, it follows that $i$ is strong symmetric monoidal
on $A$ as well. Consequently, the canonical comparison map between $A\otimes_R A$ computed in $\CMoninf(\cC)$ and in $\PCMoninf(\cC)$ is an equivalence, so the second Galois condition holds in $\PCMoninf(\cC)$.

The first Galois condition follows because $i$ commutes with limits.
\end{proof}

We now prove two lemmas that will imply the first Galois condition.

\begin{lemma}\label{lem: pullback cyclo extension}
Let $\cC$ be stable and $R\in \CAlg(\CMoninf(\cC))$. Assume that $F$ is $\ZZ/p^r$-orientable of dimension
$\dim_{\ZZ/p^r}(F)=n-t$. Then there is an identification
\[
F^*\bigl(F_!(R)[\omega^{(t)}_{p^r}]\bigr)\simeq F^*F_!(R)\otimes_R R[\omega^{(n)}_{p^r}] \quad \in \PCMoninf(\cC)^{(\ZZ/p^r)\units}.
\]
\end{lemma}

\begin{proof}
Let $e\in R[\B^n\ZZ/p^r]$ be the idempotent whose inversion defines the cyclotomic extension, so
\[
R[\omega^{(n)}_{p^r}] \simeq R[\B^n\ZZ/p^r][e^{-1}].
\]
Using \cref{cor: base change} we compute
\[
F^*F_!(R)\otimes_R R[\omega^{(n)}_{p^r}]
\simeq
F^*F_!(R)\otimes_R R[\B^n\ZZ/p^r][e^{-1}]
\simeq
F^*(F_!(R)[\B^t\ZZ/p^r][e^{-1}]).
\]
Thus it suffices to identify the image of $e$ in $F_!(R)[\B^t\ZZ/p^r]$.

Recall that
\[
e \;=\; 1-|\frac{1}{|\B^n\ZZ/p|}|\int^{R}_{\B^n\ZZ/p}\iota_n,
\]
where $\iota_n\colon \B^n\ZZ/p \to \B^n\ZZ/p^r \to R[\B^n\ZZ/p^r]$.
By \cref{cor: character roots of unity}, the element $e$ is carried to
\[
1-\frac{1}{|F(\B^n\ZZ/p)|}\int^{F_!(R)}_{F(\B^n\ZZ/p)}\chi^{\un}(\iota_n)
\;=\;
1-\frac{1}{|\B^t\ZZ/p|}\int^{F_!(R)}_{\B^t\ZZ/p}\iota_t,
\]
which is precisely the idempotent defining $F_!(R)[\omega^{(t)}_{p^r}]$.
Therefore,
\[
F^*F_!(R)\otimes_R R[\omega^{(n)}_{p^r}]
\simeq
F^*\bigl(F_!(R)[\omega^{(t)}_{p^r}]\bigr),
\]
as claimed.
\end{proof}

\begin{lemma}\label{lem: descent for character of cyclotomic extension}
Let $\cC$ be stable and $R\in \CAlg(\CMoninf(\cC))$. Assume that $R[\omega^{(n)}_{p^r}]$ is a $(\ZZ/p^r)\units$ Galois extension of $R$, and that $F$ is $\ZZ/p^r$-orientable of dimension $\dim_{\ZZ/p^r}(F)=n-t$. Then
\[
(F_!(R)[\omega^{(t)}_{p^r}])^{h(\ZZ/p^r)^\times}\simeq F_!(R).
\]
\end{lemma}

\begin{proof}
Since $F^*\colon \CMoninf(\cC)\to \PCMoninf(\cC)$ commutes with limits and is conservative, it suffices to prove that
\[
(F^*(F_!(R)[\omega^{(t)}_{p^r}]))^{h(\ZZ/p^r)^\times} \simeq F^*F_!(R).
\]
By \cref{lem: pullback cyclo extension} we have an equivalence
\[
F^*(F_!(R)[\omega^{(t)}_{p^r}]) \simeq F^*F_!(R)\otimes_R R[\omega^{(n)}_{p^r}],
\]
so it remains to compute
\[
(F^*F_!(R)\otimes_R R[\omega^{(n)}_{p^r}])^{h(\ZZ/p^r)^\times}.
\]
Because $R[\omega^{(n)}_{p^r}]$ is a retract of $R[\B^n\ZZ/p^r]^{(-)}$, it is dualizable as an
$R$-module in $\Mod_R(\PCMoninf(\cC))$. Therefore the fixed points commute with tensoring with $F^*F_!(R)$, and we obtain
\[
(F^*F_!(R)\otimes_R R[\omega^{(n)}_{p^r}])^{h(\ZZ/p^r)^\times}
\simeq
F^*F_!(R)\otimes_R (R[\omega^{(n)}_{p^r}])^{h(\ZZ/p^r)^\times}
\simeq
F^*F_!(R),
\]
where the last equivalence uses that $R[\omega^{(n)}_{p^r}]$ is $(\ZZ/p^r)\units$ Galois over $R$.
\end{proof}

We finally prove that $F_!(R)[\omega^{(t)}_{p^r}]$ is a Galois extension.

\begin{corollary}\label{cor: galois after character}
Let $\cC$ be stable and let $R\in \CAlg(\CMoninf(\cC))$. Assume that
$R[\omega^{(n)}_{p^r}]$ is a $(\ZZ/p^r)\units$ Galois extension of $R$, and that $F$ is
$\ZZ/p^r$-orientable of dimension $\dim_{\ZZ/p^r}(F)=n-t$. Then
$F_!(R)[\omega^{(t)}_{p^r}]$ is a $(\ZZ/p^r)\units$ Galois extension of $F_!(R)$.
\end{corollary}

\begin{proof}
By \cref{thm:cyclo blue shift}, $F_!(R)[\omega^{(t)}_{p^r}]$ is the image of $R[\omega^{(n)}_{p^r}]$ under $F_!$. As $F_!$ is symmetric monoidal this together with \cref{lem: tensor product of cyc extn}, verifies the tensor product condition
for a Galois extension. The descent condition is \cref{lem: descent for character of cyclotomic extension}.
\end{proof}

\subsection{Fourier Transform}\label{sec: Fourier Transform}

We now turn to relate the Fourier transform arising in character theory to the Fourier transform on $R$. We begin by recalling the notions of orientation and Fourier transform from \cite{BCSY-Fourier}. In that work, Barthel, Carmeli, Schlank, and Yanovski introduce \emph{orientations} for $\infty$-semiadditive categories, in the sense explained below, these play the role of roots of unity in the construction of Fourier transforms.

\begin{definition}
 Let $I^{(n)}_p := \tau_{\ge 0}\Sigma^{\chrHeight} I_{\QQ_p/\Zp}$ be the truncated and shifted $p$-typical Brown-Comenetz dual of the sphere.
\end{definition}

 The spectrum $I^{(n)}_p$ is a higher analogue of the group of $p$-typical roots of unity, in that $\pi_n(I^{(n)}_p)\simeq \mu_{p^\infty}$. Its $\chrHeight$-th homotopy group encodes higher roots of unity as in \cite{CSY-cyclotomic}. In particular, it determines a height $\chrHeight$ version of Pontryagin duality.

 \begin{definition}
 Let $M$ be a $p$-local connective spectrum. Its height $\chrHeight$ Pontryagin dual is
 \[
 I^{(n)}_pM\colon=\tau_{\ge 0}\hom(M,I^{(n)}_p).
 \]
 \end{definition}

 Let $\cD$ be presentably symmetric monoidal, and let $\mathcal{R}\in \CAlg(\cnSp)$. An $(\mathcal{R},\chrHeight)$-\emph{pre-orientation} of $\cD$ is a map
 \[
 \omega\colon I^{(n)}_p\mathcal{R}\to \ounit_{\cD}\units.
 \]
 A pre-orientation determines, for each $M\in \Mod_{\mathcal{R}}^{[0,\chrHeight]\hyphen\mrm{fin}}$, a Fourier transform \cite[\S 3.2]{BCSY-Fourier}
 \[
\mscr{F}_\omega\colon \ounit_{\cD}[M]\to \ounit_{\cD}^{\I^{(n)}_pM}
 \quad \in\; \CAlg(\cD).
 \]
 We say that $\omega$ is an \emph{orientation} if $\mscr{F}_\omega$ is an equivalence for all such $M$.

In $\Sp_{T(n)}$, and more generally in any category whose unit admits a faithful extension that is $(\FF_p,n)$-oriented, the data of a height $n$ $p^r$-th primitive root of unity is the same as an $(\ZZ/p^r,n)$ orientation see \cite[Proposition 6.6]{BCSY-Fourier}.  

We will impose the following standing assumption on $F$ for the remainder of this section.

\begin{assumption}\label{ass: F mapping manifold}
The functor
\[
F\colon \spcpi \to \spcpi
\]
is of the form $F(A)=\Map(X,A)$ for some oriented connected closed manifold $X$ of dimension $n-t$. We also fix a choice of a degree-one map $X \to S^{n-t}$. Note that such a map is unique up to non-canonical homotopy.
\end{assumption}

We now use the Fourier transform on $R$ to build one on $F_!(R)$. 

\begin{construction}\label{cons: candidate Fourier transform}
    Let $R,S\in\CAlg(\tsadi)$, $F$ be $\ZZ/p^r$-oriented of dimension $\dim_{\ZZ/p^r}(F)$ and 
    \[
    \chi\colon R \to F^*S
    \]
    be a character map of commutative algebras. Assume that $R$ is $T(n)$-local and $(\ZZ/p^r,n)$-oriented. For $M\in \Mod_{\ZZ/p^r}^{[0,t]-\fin}$ we define:
    \begin{gather*}
    \varphi_{M,S}^R\colon
    L_{T(t)}(R)[M]
    \xrightarrow{ \mathrm{assem} }
    L_{T(t)}(R[M])
    \simeq
    L_{T(t)}(R^{\I^{(n)}_{\ZZ/p^r}M})
    \to \\
    \to S^{F(\I^{(n)}M)} \to S^{\I^{(t)}M}.
    \end{gather*}
    as the composite of 
    \begin{enumerate}
        \item  The \emph{assembly map}  
            $(L_{T(t)}R)[M]  \to  L_{T(t)}(R[M])$,
        \item The functor $L_{T(t)}$ applied to the equivalence  
            $R[M]\simeq R^{\I^{(n)}_{\ZZ/p^r}M}$, coming from the chromatic Fourier transform corresponding to the $(\ZZ/p^r,n)$ orientation of $R$.
        \item The factorization of $\chi$ through the $T(t)$ localization coming from $\cref{cor: F chr single height T local}$.
        \item Pulling back along
        \[
        I^{(t)}_{\ZZ/p^r}M \to   I^{(t)}_{\ZZ/p^r}M \times \I^{(n)}_{\ZZ/p^r}M \simeq \Map(S^{n-t},\I^{(n)}_{\ZZ/p^r}M)  \to \Map(X,\I^{(n)}_{\ZZ/p^r}M)=F(\I^{(n)}_{\ZZ/p^r}M)
        \]
        where the map $X \to S^{n-t}$ is given by our choice of degree one map.     
    \end{enumerate}
    When $S$ or $R$ are clear from context we will omit them from the notation.
\end{construction}

We will need the following two lemmas. 

\begin{lemma} \label{lem: ZZ/pr action on root of unity}
    Let $\cC\in \CAlg(\Prinfad)$ and 
    \[
    \mathcal{F}_{w^{(n)}_{p^r},(-)}:\ounit_\cC[-] \to \ounit^{\I^{(n)}_{\ZZ/p^r}-}
    \]
    associated to a $p^r$-th height $n$ root of unity $\omega^{(n)}_{p^r}$. Then the map $\mathcal{F}_{w^{(n)}_{p^r},\ZZ/p^r}$ is equivalent to the image of the following composition under the $\ounit[-]\dashv(-)\units$ adjunction
    \[
    \ZZ/p^r \xto{\omega^{(n)}_{p^r}} \hom^{\ZZ/p^r}(\Sigma^{n}\ZZ/p^r,\ounit_\cC\units)\to \Map(\B^{n}\ZZ/p^r,\Omega^{\infty}_\cC\ounit_{\cC}\units)\simeq (\ounit_\cC^{\B^{n}\ZZ/p^r})\units
    \]
    where the first map chooses $\omega^{(n)}_{p^r}$ in the $\pi_0$ of the $\ZZ/{p^r}$ module $\hom^{\ZZ/p^r}(\Sigma^{n}\ZZ/p^r,\ounit_\cC\units)$ and the second map forgets the $\ZZ/p^r$ enrichment to the mapping space. 
\end{lemma}

\begin{proof}
By unwinding the definition of the chromatic Fourier transform as a right Kan extension (see \cite[Remark 3.12]{BCSY-Fourier}), we find that 
$\mathcal{F}_{w^{(n)}_{p^r},\B^{n}\ZZ/p^r}$, when composed with the projection to the $i$-th coordinate in $\prod_{j=0}^{p^r} \ounit_\cC^{\times}$, is the map $(w^{(n)}_{p^r})^{i}$, namely, the action of $i\in \ZZ/p^r$ on 
$\Sigma^{n}\ZZ/p^r$ followed by $w^{(n)}_{p^r}$. 
The claim then follows by taking duals and applying \cite[Corollary 3.26]{BCSY-Fourier}.
\end{proof}

\begin{lemma}\label{lem: fourier equvilance iff}
    Let $R\in \CAlg(\Sp_{T(n)})$  and
    \[
    \mathcal{F}_{\omega_{p^r}^{(n)},(-)}\colon R[-] \to R^{\I^{(n)}(-)}\in \Fun(\Mod_{\ZZ/p^r}^{[0,n]-\fin},\CAlg_R(\Sp_{T(n)}))
    \]
    be a Fourier transform. Then $\mathcal{F}_{\omega_{p^r}^{(n)},(-)}$ is an equivalence if and only if $\mathcal{F}_{\omega_{p^r}^{(n)},\ZZ/p^r}$ is an equivalence.
\end{lemma}

\begin{proof}
    The first direction is immediate. Assume that $\mathcal{F}_{\omega_{p^r}^{(n)}, \ZZ/p^r}$ is an equivalence, that is, the object $\ZZ/p^r \in \Mod_{\ZZ/p^r}^{[0,n]\text{-}\fin}$ is $\omega_{p^r}^{(n)}$-oriented. Note first that by the exact sequence 
    \[
    \ZZ/p \to \ZZ/p^{r} \xto{p} \ZZ/p^{r}
    \]
     and \cite[Proposition 4.12]{BCSY-Fourier} we get that $\ZZ/p \in \Mod_{\ZZ/p^r}^{[0,n]\text{-}\fin}$ is $\omega_{p^r}^{(n)}$-oriented. By \cite[4.39]{BCSY-Fourier}, it suffices to check that the composite
    \[
    \omega_p^{(n)} \colon \B^{n}\ZZ/p \to \B^{n}\ZZ/p^r \to R^{\times}
    \]
    defines an orientation. By \cite[4.4]{BCSY-Fourier}, the module $\ZZ/p$ is $\omega_p^{(n)}$-oriented. Using the exact sequences
    \[
    \Sigma^i\ZZ/p \to 0 \to \Sigma^{i+1}\ZZ/p
    \]
    and \cite[Proposition 4.12]{BCSY-Fourier}, we conclude that each suspension $\Sigma^i\ZZ/p$ is also $\omega_p^{(n)}$-oriented. Furthermore, by \cite[Proposition 4.10]{BCSY-Fourier}, finite sums of $\omega_p^{(n)}$-oriented modules are themselves oriented. We conclude that $\omega_p$ defines an orientation, as required.
\end{proof}

We can now prove the main result of this subsection.

\begin{proposition}\label{prop: Fourier in chr}
    Let $R,S\in\CAlg(\tsadi)$ and 
    \[
    \chi\colon R \to F^*S
    \]
    be a character map of commutative algebras. Assume that $R$ is $T(n)$-local and $(\ZZ/p^r,n)$-oriented. Then the factorization of $\varphi$ from \cref{cons: candidate Fourier transform} through:
    \[
    \bar{\varphi}_{M} \colon L_{T(t)}(R)[M] \widehat{\otimes}_{L_{T(t)}R}S\simeq S[M] \to S^{\I^{(t)}M}
    \]
    is the Fourier transform induced by the height $t$ $p^r$-th root from \cref{cons: root of unity candidate} and in particular by \cref{cor: character roots of unity} an equivalence for every $M$.
\end{proposition}

\begin{proof} 
By \cref{lem: fourier equvilance iff}, it suffices to show that $\bar{\varphi}_{\ZZ/p^r}$ is an equivalence. In that case, the assembly map is an equivalence, and it remains to verify that the resulting map agrees with the Fourier transform of $S$ induced by $\chi(\omega_{p^r}^{(n)})$ evaluated at $\ZZ/p^r$. In other words, we must show that the composite
\[
\ZZ/p^r \to (R[\ZZ/p^r])^{\times} \to (R^{\B^n \ZZ/p^r})^{\times} \to (S^{\B^t \ZZ/p^r})^{\times} \to (S[\ZZ/p^r])^{\times}
\]
is the unit map.

We claim this follows from the construction of the height $t$ root of unity associated to $\omega^{(n)}_{p^r}$ by $\chi$ (see \cref{cons: root of unity candidate}). Indeed, the map
\[
\ZZ/p^r \to (R[\ZZ/p^r])^{\times} \to (R^{\B^n \ZZ/p^r})^{\times}
\]
selects the element $\omega_{p^r}^{(n)} \in (R^{\B^n \ZZ/p^r})^{\times}$, equipped with its height $0$, $p^r$-th root of unity structure coming from \cref{lem: ZZ/pr action on root of unity}. Under the subsequent map
\[
(R^{\B^n \ZZ/p^r})^{\times} \to (S^{\B^t \ZZ/p^r})^{\times},
\]
this element is sent to $\chi(\omega_{p^r}^{(n)})$, which carries the structure of a height $0$, $p^r$-th root of unity in $(S^{\B^t \ZZ/p^r})^{\times}$, coming from \cref{lem: ZZ/pr action on root of unity}.

Finally, composing with the Fourier transform of $S$, namely $\mathcal{F}_{\chi(\omega_{p^r}^{(n)}),\ZZ/p^r}$, yields the unit map
\[
\ZZ/p^r \to (S[\ZZ/p^r])^{\times},
\]
as required.
\end{proof}

\section{Revisiting the Construction of the Splitting Algebra}\label{sec: splitting algebra}

In this section we use the chromatic Fourier transform \cite{BCSY-Fourier}, together with the theory of alternating powers of $p$-divisible groups from \cite{Hopkins-Lurie-2013-ambi}, to
give an alternative perspective on parts of the construction of the $T(t)$-local splitting
algebra developed by Hopkins, Kuhn, Ravenel, Stapleton, and Lurie
\cite{Stapleton-2013-HKR,Hopkins-Kuhn-Ravanel-2000-HKR,Lurie-2019-Elliptic3}.

In \cref{sec: Set-Up} we recall the general set-up for the splitting algebra and show that the problem reduces to finding the initial extension over which a certain map of etale finite flat group schemes becomes an isomorphism. In \cref{sec: Exterior Powers} we show that this condition can be detected on a top exterior power of the relevant group schemes, and we identify the exterior power that appears. Finally, in
\cref{sec: Inverting the Determinant} we invert the resulting determinant map and assemble these ingredients to obtain a new semiadditive construction of the splitting algebra.

\subsection{The Set-Up}\label{sec: Set-Up}

Recall that $\En{}(\overline{\FF}_p,\gamma)=\En$ comes equipped with the universal deformation formal group $\GG_{\En}$\footnote{We use the theory of spectral p-divisible groups as developed in \cite{lurie-2017-elliptic1, Lurie-2018-Elliptic2, Lurie-2019-Elliptic3}}. After $T(t)$-localization, this formal group becomes a $p$-divisible group whose formal part has height $t$. This $p$-divisible group fits into a connected-etale exact sequence
\[
    \GG_{L_{T(t)}\En}^{0} \to \GG_{L_{T(t)}\En} \to \GG_{L_{T(t)}\En}^{\et}.
\]
The splitting algebra, $\widehat{C}_t^n$, is the initial $T(t)$-local $L_{T(t)}\En^0$-algebra over which this sequence splits\footnote{Here, \emph{splits} means that the exact sequence is split and the etale part is trivialized.}. In particular, over $\widehat{C}_t^{n}$ there is a canonical isomorphism of $p$-divisible groups
\[
    \GG_{\widehat{C}_t^{n}} \simeq \GG^0_{\widehat{C}_t^{n}} \times \underline{(\QQ_p/\ZZ_p)}^{\chrHeight-t}.
\]
Recall that a $p$-divisible group is a compatible system of finite flat group schemes given by the $p^r$-torsion. Hence, splitting the above exact sequence is equivalent to splitting the corresponding sequence for the $p^r$-torsion and then taking the colimit over $r$.

We will adopt the following notation:

\begin{notation}
    For $r\ge 1$ and $0\leq t < n$ we write
    \[
    \Zpr := (\ZZ/p^r)^{n-t},
    \qquad
    L_{t,r} := L_{T(t)}(\En^{\B \Zpr}).
    \]
    For an abelian group $A$, we denote its Pontryagin dual by $\widehat{A}$.
\end{notation}


\begin{lemma} \label{lem: Ltr functor of points}
    The functor \[
    \CAlg_{L_{T(t)}\En}(\Sp_{T(t)})\to \Mod_\ZZ, \quad R\mapsto \Hom(\underline{\Zprdbl}_R , \GG_R[p^r])
    \]
    is represented by 
    \[
    \Hom_{\CAlg_{L_{T(t)}\En}}(L_{t,r},-)
    \]  
\end{lemma}

\begin{proof}
    By \cite[Construction 4.6.2]{Lurie-2018-Elliptic2} and \cite[Remark 2.7.6]{Lurie-2019-Elliptic3}, the functor 
    \[
    \CAlg_{L_{T(t)}\En}(\Sp) \to \Mod_\ZZ, \quad R\mapsto \Hom(\underline{\Zprdbl}_R , \GG_R[p^r])
    \]
    is represented by $\tau_{\geq 0}\En{}^{\B\Zpr}$. Since $T(t)$ localization remains unchanged under passage to a connected cover, the result follows. 
\end{proof}

To proceed, note that splitting the $p^r$-th torsion of the exact sequence above is equivalent to specifying a map
\[
\underline{\Zprdbl}\to \GG[p^r] 
\]
that induces an equivalence after composing with
\[
\GG[p^r] \to \GG^\et[p^r]
\]
Using \cref{lem: Ltr functor of points} we see that the identity map
\[
L_{t,r} \xto{\id} L_{t,r}
\]
gives a map of finite flat group schemes 
\[
\underline{\Zprdbl}_{L_{t,r}} \to \GG_{L_{t,r}}^\et[p^r]
\]
and therefore a map of global sections
\[
 \mathcal{O}_{\GG_{L_{t,r}}^\et[p^r]}\to \mathcal{O}_{\underline{\Zprdbl}_{L_{t,r}}} 
\]
These are free modules over $L_{t,r}$ (in the proof of \cite[Proposition 2.5]{Stapleton-2013-HKR}, Stapleton explicitly constructs a basis for $\mathcal{O}_{\GG_{L_{t,r}}}^\et$). Therefore, to ensure that the map is an equivalence, it suffices to invert the determinant\footnote{Here when we say determinant we mean determinant up to a unit.}.

In \cite{Hopkins-Kuhn-Ravanel-2000-HKR, Stapleton-2013-HKR} the determinant of this map is identified through an explicit computation. We will take a different approach, drawing on the frameworks developed in \cite{Hopkins-Lurie-2013-ambi} and \cite{BCSY-Fourier}. As we are looking for an element in $\pi_0L_{T(t)}(\En^{\B\Zpr})$, we will work with $\GG^\heart$ from this point forward see \cite[Remark 2.7.17]{Lurie-2019-Elliptic3}.

\subsection{Exterior Powers}\label{sec: Exterior Powers}

 We will be using extensively the theory of exterior powers of $p$-divisible groups. See \cite{exterior-ofpdiv-group} and \cite[section 3]{Hopkins-Lurie-2013-ambi} for the relevant background.

\begin{lemma} \label{lem: top extrior conservative on etale}
    Let $f:\GG_1,\to \GG_2$ be a map between two etale $p$-divisible groups over $R$ of the same height, $h$. Then $f$ is an isomorphism if and only if  $\Lambda^h f$ is an isomorphism.
\end{lemma}

\begin{proof}
    It suffices to show that the map is an isomorphism after an etale extension. Thus, we may assume that 
    \[
    \GG_1[p^r] = \GG_2[p^r] = (\underline{\ZZ/p^r})^h.
    \]
    In this case, at the level of functors of points, $f$ is given by the map
    \[
    (\ZZ/p^r)^h \to (\ZZ/p^r)^h.
    \]
    By the universal property of the exterior power, the induced map on exterior powers is the map of $\ZZ/p^r$-modules given by multiplication by the determinant.
\end{proof}

\begin{lemma} \label{lem: cup product extrior}
    The global sections of $\Lambda^{n-t}\GG_{L_{T(t)}\En{}}^\heart[p^r]$ are given by:
    \[
    \mathcal{O}_{\Lambda^{n-t}\GG_{L_{T(t)}\En{}}^\heart[p^r]} =\pi_0\L_{T(t)}(\En^{\B^{n-t}\ZZ/p^r})
    \]
    The global sections of the map
    \[
    (\GG_{L_{T(t)}\En{}}^\heart[p^r])^{\times n-t}\to \Lambda^{n-t}\GG_{L_{T(t)}\En{}}^\heart[p^r]
    \]
    are given by the map induced from the cup product
    \[
    \pi_0L_{T(t)}(\En^{\B^{n-t}\ZZ/p^r})\to \pi_0L_{T(t)}(\En^{\B(\Zpr)})=\pi_0(L_{t,r})
    \]
\end{lemma}

\begin{proof}
   By \cite[Theorem 4.4.16]{Lurie-2019-Elliptic3} we have
    \[
    \mathcal{O}_{\Lambda^{n-t}\GG_{\En{}}^\heart[p^r]}  =  \pi_0(\En^{\B^{n-t}\ZZ/p^r}),
    \]
    and that the map induced from the cup product
    \[
    \pi_0(\En^{\B^{n-t}\ZZ/p^r}) \to \pi_0(\En^{\B(\Zpr)})
    \]
    is the map on global sections induced by 
    \[
    (\GG_{\En{}}^\heart[p^r])^{\times (n-t)} \to \Lambda^{n-t}\GG_{\En{}}^\heart[p^r].
    \]
    Since these are free, we have
    \[
    L_{T(t)}(\En^{\B^{n-t}\ZZ/p^r}) 
    \simeq L_{T(t)}\En \otimes_{\En} \En^{\B^{n-t}\ZZ/p^r},
    \quad
    L_{T(t)}(\En^{(\B\ZZ/p^r)^{n-t}}) 
    \simeq L_{T(t)}\En \otimes_{\En} \En^{(\B\ZZ/p^r)^{n-t}}.
    \]
    In particular, both sides of the desired equivalence are obtained from the non $T(t)$ localized statement by base change along $L_{T(t)}\En$.  
    By the universal property the exterior power functor commutes with base change, and hence the claim follows.
\end{proof}

\begin{corollary}\label{cor: maps from const to extrior}
The functor
\[
\CAlg_{L_{T(t)}\En^0}(\Ab)^{\wedge}_{I_t}\to \Mod_\ZZ^\heart,
\qquad
R \mapsto \Hom\!\bigl(\underline{\widehat{\ZZ/p^r}},\, \Lambda^{n-t}\GG_R^\heart[p^r]\bigr)
\]
is represented by $\pi_0 L_{T(t)}\En^{\B^{n-t}\ZZ/p^r}$.

Under this identification, the map of representing objects
\[
\pi_0 L_{T(t)}\En^{\B^{n-t}\ZZ/p^r}\to \pi_0 L_{T(t)}\En^{\B(\Zpr)}
\]
corresponds, on functors of points, to the exterior power map
\[
\Hom_\ZZ(\underline{\widehat{\Zpr}},\,\GG_R^\heart)
\to
\Hom_\ZZ\!\bigl(\Lambda^{n-t}\underline{\widehat{\ZZ/p^r}}^{\,n-t},\, \Lambda^{n-t}\GG_R^\heart\bigr).
\]
\end{corollary}

\begin{proof}
For the representability statement, let
\[
f\colon \underline{\ZZ/p^r}_R \to \Lambda^{n-t}\GG_R^\heart[p^r]
\]
be a morphism. By \cref{lem: cup product extrior}, the induced map on global sections is
\[
f^*\colon \pi_0R\otimes_{\pi_0 L_{t,r}} \pi_0 L_{T(t)}\En^{\B^{n-t}\ZZ/p^r}
\to
\prod_{\ZZ/p^r}\pi_0R.
\]
Choosing a generator of $\widehat{\ZZ/p^r}$ yields an $L_{T(t)}\En$-algebra map
\[
\pi_0 L_{T(t)}\En^{\B^{n-t}\ZZ/p^r} \to \prod_{\ZZ/p^r} R \to R.
\]
The reverse process shows that such a map determines $f$, and this correspondence is functorial in $R$. Hence we
obtain a natural transformation
\[
\Hom(\pi_0 L_{T(t)}\En^{\B^{n-t}\ZZ/p^r},\, (-))
\to
\Hom(\underline{\widehat{\ZZ/p^r}},\, \Lambda^{n-t}\GG_{(-)}^\heart[p^r])
\]
which is an equivalence, proving the first claim.

The final statement follows from the second part of \cref{lem: cup product extrior} together with
the universal property of exterior powers.
\end{proof}

\begin{lemma}  \label{lem: coassembly as map of extrior}
    $(L_{T(k)}\En)^{\B^{t}\ZZ/p^r}$ is a projective $L_{T(k)}\En$ module and the co-assembly map 
    \[
   \pi_0L_{T(k)}(\En^{\B^{t}\ZZ/p^r})\to  \pi_0(L_{T(k)}\En)^{\B^{t}\ZZ/p^r}
    \]
    classifies the map
    \[
    \Lambda^t \GG_{L_{T(k)}\En}^{0,\heart}[p^r]\to \Lambda^t \GG_{L_{T(k)}\En}^\heart[p^r]
    \]
\end{lemma}

\begin{proof}
By \cite[Theorem 4.4.16]{Lurie-2019-Elliptic3} and \cite[Proposition 2.4]{Stapleton-2013-HKR} the map induced by pulling back along the cup product after $T(k)$ localization:
\[
\pi_0(L_{T(k)}\En{})^{\B(\ZZ/p^r)^{t}} \to \pi_0(L_{T(k)}\En{})^{\B^{t}\ZZ/p^r}
\]
classifies the map of finite flat group schemes:
\[
(\GG_{L_{T(k)}\En)}^{0,\heart}[p^r])^{\times t} \to \Lambda^t (\GG_{L_{T(k)}\En)}^{0,\heart}[p^r])
\]
By \cref{lem: cup product extrior} we have that the map induced by pulling back along the cup product:
\[
\pi_0L_{T(k)}(\En^{\B(\ZZ/p^r)^{t}}) \to \pi_0L_{T(k)}(\En^{\B^{t}\ZZ/p^r})
\]
classifies the map of finite flat group schemes
\[
(\GG_{L_{T(k)}\En}^\heart[p^r])^{\times t}\to \Lambda^t \GG_{L_{T(k)}\En}^\heart[p^r]
\]
By \cref{lem: Ltr functor of points} the map
\[
\pi_0(L_{T(k)}\En)^{\B(\ZZ/p^r)^{t}} \to \pi_0L_{T(k)}(\En^{\B(\ZZ/p^r)^{t}})
\]
classifies the map of finite flat group schemes
\[
 (\GG_{L_{T(k)}\En)}^{0,\heart}[p^r])^{t} \to (\GG_{L_{T(k)}\En)}^\heart[p^r])^{t}
\]
So the claim follows from the universal property of exterior power.
\end{proof}

\begin{corollary} \label{cor: assembly as map of alt}
The assembly map
\[
\pi_0L_{T(t)}(\En{})[\B^t\ZZ/p^r] \to \pi_0L_{T(t)}(\En{}[\B^t\ZZ/p^r])
\]
classifies the map of finite flat group schemes:
\[
\Alt^t(\GG_{L_{T(t)}\En}^\heart[p^r]) \to \Alt^t(\GG_{L_{T(t)}\En}^{0,\heart}[p^r])
\]
\end{corollary}

\begin{proof}
    This is the linear dual of \cref{lem: coassembly as map of extrior}.
\end{proof}

\begin{lemma}\label{lem: height of etale part}
    The height of the etale part of
    $\Lambda^{n-t}\GG_{L_{T(t)}\En}$ is $1$.
\end{lemma}

\begin{proof}
    Set $d=n-t$. Base changing to the $T(t)$-local splitting
    algebra $C_t^n$ and choosing $x\in \lvert\Spf_{I_t}(C_t^n)\rvert$,
    we see by \cite[4.4.14]{Lurie-2018-Elliptic2} that it suffices
    to prove the claim after base change to $\kappa(x)$. After a
    further base change, we may assume that $\kappa(x)$ is perfect.
    Over $\kappa(x)$, we have
    \[
        \GG_{\kappa(x)}
        \simeq
        \GG_{\kappa(x)}^0
        \times
        (\QQ_p/\ZZ_p)^d.
    \]

    Let
    \[
        M=\DD(\GG_{\kappa(x)}),
        \qquad
        M^0=\DD(\GG_{\kappa(x)}^0),
        \qquad
        M^\et=\DD((\QQ_p/\ZZ_p)^d),
    \]
    where $\DD$ is the Dieudonne module functor.  We use the contravariant convention for Dieudonne modules, so
    that $F$ is invertible on $M^\et$ and nilpotent modulo $p$ on
    $M^0$.

    By the Dieudonne-module description of exterior powers, there is
    an identification
    \[
        \DD(\Lambda^d\GG_{\kappa(x)})
        \simeq
        \Lambda^d M,
    \]
    under which Frobenius and Verschiebung are given by
    \[
        F_d(a_1\wedge\dots\wedge a_d)
        =
        F(a_1)\wedge\dots\wedge F(a_d)
    \]
    and
    \[
        V_d(a_1\wedge\dots\wedge a_d)
        =
        p^{1-d}
        V(a_1)\wedge\dots\wedge V(a_d)
    \]

    see the proof of \cite[Theorem 3.3.1]{Hopkins-Lurie-2013-ambi} for the derivation of the above. The decomposition $M\simeq M^0\oplus M^\et$ induces an
    $F_d$-stable decomposition
    \[
        \Lambda^dM
        \simeq
        \mathop{\oplus}_{i=0}^d
        \Lambda^iM^0\otimes\Lambda^{d-i}M^\et.
    \]
    Since $F$ is nilpotent modulo $p$ on $M^0$, the restriction of
    $F_d$ to every summand with $i>0$ is nilpotent modulo $p$.
    On the other hand, $F_d$ is invertible on
    \[
        \Lambda^dM^\et.
    \]
    It follows that the maximal submodule of $\Lambda^dM$ on which
    $F_d$ is invertible is precisely $\Lambda^dM^\et$.

    Finally, $M^\et$ has rank $d$, so $\Lambda^dM^\et$ has rank
    one. Thus the etale part of
    $\Lambda^d\GG_{\kappa(x)}$ has height $1$, and hence so does the
    etale part of $\Lambda^{n-t}\GG_{L_{T(t)}\En}$.
\end{proof}

\subsection{Inverting the Determinant} \label{sec: Inverting the Determinant}

We now use the chromatic Fourier Transform to tie everything together and give a new semiadditive construction of the splitting algebra.

\begin{corollary}\label{cor: restriction to et after extrior}
    The map 
\[
\pi_0 L_{T(t)}(\En{})[\B^t\ZZ/p^r] \to \pi_0L_{T(t)}(\En{}[\B^t\ZZ/p^r])
\]
is the global sections of the map of finite flat group schemes:
\[
\Lambda^{n-t}(\GG_{L_{T(t)}\En}[p^r]) \to \Lambda^{n-t}(\GG_{L_{T(t)}\En}^\et[p^r])
\]
in particular the etale part of $\Lambda^{n-t}(\GG_{L_{T(t)}\En}[p^r])$ is the group of height $t$ roots of unity.
\end{corollary}

\begin{proof}
    By \cref{cor: assembly as map of alt} this map yields an epimorphism of finite flat group schemes to an etale group scheme of height $1$. By \cref{lem: height of etale part} the etale part of the source has height $1$.
\end{proof}

\begin{corollary} \label{cor: trivializing etale part of extrior}
The functor $\CAlg_{\pi_0\L_{T(t)}\En}\to \Set$ assigning to an $\pi_0\L_{T(t)}\En$-algebra $R$ the subset
\[
\Iso_R (\underline{\widehat{\ZZ/p^r}}_R , \Lambda^{n-t}\GG_R^{\et,\heart}) \subseteq \hom_R (\underline{\widehat{\ZZ/p^r}}_R , \Lambda^{n-t}\GG_R^\heart )
\]
consisting of morphisms $\underline{\ZZ/p^r}_R \to \Lambda^{n-t}\GG_R^\heart$ for which the
composite morphism
\[
\underline{\widehat{\ZZ/p^r}}_R  \to \Lambda^{n-t}\GG_R^\heart \to \Lambda^{n-t}\GG_R^{\et,\heart}
\]
is an isomorphism is represented by 
\[
\pi_0L_{T(t)}\En{}[\omega^{(t)}_{p^r}]\otimes_{\pi_0 {L_{T(t)}(\En{})[\B^t\ZZ/p^r]} }\pi_0L_{T(t)}(\En^{\B^{n-t}\ZZ/p^r}).
\]
\end{corollary}

\begin{remark}
    The ring 
    \[
    \pi_0L_{T(t)}\En{}[\omega^{(t)}_{p^r}]\otimes_{\pi_0 {L_{T(t)}(\En{})[\B^t\ZZ/p^r]} }\pi_0L_{T(t)}(\En^{\B^{n-t}\ZZ/p^r}),
    \]
    can be obtained from $\pi_0L_{T(t)}(\En^{\B^{n-t}\ZZ/p^r})$ by inverting the image of the idempotent that gives higher cyclotomic extension under the map  
    \[
    \pi_0 L_{T(t)}(\En{})[\B^t\ZZ/p^r] \to \pi_0 L_{T(t)}(\En{}[\B^t\ZZ/p^r])\simeq \pi_0L_{T(t)}(\En^{\B^{n-t}\ZZ/p^r} )
    \]
    where the first map is the assembly map and the second map is the chromatic Fourier transform.
\end{remark}

\begin{proof}(of \cref{cor: trivializing etale part of extrior})
    By \cref{cor: maps from const to extrior} the functor $R\mapsto \hom_R (\underline{\widehat{\ZZ/p^r}}_R , \Lambda^{n-t}\GG_R^\heart )$ is represented by the $\pi_0L_{T(t)}\En$ algebra $\pi_0L_{T(t)}\En^{\B^{n-t}\ZZ/p^r}$. By \cref{cor: restriction to et after extrior} composing with the map 
    \[
    \pi_0L_{T(t)}(\En)[\B^t\ZZ/p^r]  \to \pi_0L_{T(t)}(\En{}[\B^t\ZZ/p^r])\simeq \pi_0L_{T(t)}(\En{}^{\B^{n-t}\ZZ/p^r})
    \] 
    composes with the map
    \[
    \Lambda^{n-t}\GG_R^\heart  \to \Lambda^{n-t}\GG_R^{\et,\heart} =(\mu_{p^r}^{(t)})_R
    \]
    We therefore only need to find the extension of $L_{T(t)}(\En[\B^t\ZZ/p^r])$ over which the universal height $t$ root of unity becomes primitive. That is, we need to extend scalars to $\pi_0L_{T(t)}\En{}[\omega^{(t)}_{p^r}]$
\end{proof}

\begin{corollary}\label{cor: classifying the spliiting algebra on pi0}
The functor $\CAlg_{\pi_0\L_{T(t)}\En}\to \Set$ assigning to a $\pi_0\L_{T(t)}\En$-algebra $R$ the subset
\[
\Iso_R (\underline{\Zprdbl}_R , \GG_R^{\et,\heart}) \subseteq \hom_R (\underline{\Zprdbl}_R ,\GG_R^\heart )
\]
consisting of morphisms $\underline{\Zprdbl}_R \to \GG_R^\heart$ for which the composite morphism
\[
\underline{\Zprdbl}_R  \to \GG_R^\heart \to \GG_R^{\et,\heart}
\]
is an isomorphism is represented by 
\[
\pi_0L_{T(t)}(\En{})[\omega^{(t)}_{p^r}]\otimes_{\pi_0 L_{T(t)}({\En{})[\B^t\ZZ/p^r]} }\pi_0L_{T(t)}(\En^{\B\Zpr}).
\]
\end{corollary}

\begin{remark}\label{rem: classifying the spliiting algebra on pi0}
    The ring 
    \[
   \pi_0L_{T(t)}\En{}[\omega^{(t)}_{p^r}]\otimes_{\pi_0 L_{T(t)}({\En{})[\B^t\ZZ/p^r]} }\pi_0L_{T(t)}(\En^{\B\Zpr}).
    \]
    can be obtained from $\pi_0L_{T(t)}(\En^{\B\Zpr})$ by inverting the image of the idempotent that gives higher cyclotomic extension under the map  
    \[
    \pi_0 L_{T(t)}(\En{})[\B^t\ZZ/p^r] \to \pi_0 L_{T(t)}(\En{}[\B^t\ZZ/p^r])\simeq \pi_0L_{T(t)}(\En^{\B^{n-t}\ZZ/p^r} ) \to \pi_0L_{T(t)}(\En^{\B\Zpr})
    \]
    where the first map is the assembly map, the second map is the chromatic Fourier transform, and the third map is induced from the cup product. 
\end{remark}

\begin{proof}(of \cref{cor: classifying the spliiting algebra on pi0})
    This follows from \cref{cor: trivializing etale part of extrior} and \cref{lem: top extrior conservative on etale}.
\end{proof}

\begin{corollary}\label{cor: classifying the splitting algebra}
    The $T(t)$-local splitting algebra of the $p$-divisible group on $L_{T(t)}\En$ is the $T(t)$-localization of the colimit
    \[
   \colim_r (L_{T(t)}\En{}[\omega^{(t)}_{p^r}]\otimes_{ {L_{T(t)}(\En{})[\B^t\ZZ/p^r]} }L_{T(t)}(\En^{\B\Zpr}))
    \]
\end{corollary}

\begin{proof}
By \cref{cor: classifying the spliiting algebra on pi0} and \cref{rem: classifying the spliiting algebra on pi0} , the map
\[
L_{t,r}\to C_{t,r}^n
\]
becomes an isomorphism on $\pi_0$ after inverting the idempotent arising from the height $t$ cyclotomic extension. Since both the source and target are $2$-periodic and have vanishing $\pi_1$,  it follows that the map becomes an equivalence after inverting the idempotent. We now finish by taking colimit over $r$ on both sides and localizing.
\end{proof}

We note that the above construction of the splitting algebra  can be applied to any $T(n)$-local commutative algebra with all height $n$ roots of unity.

\section{The $K(n)$-Local Universal Character} \label{sec: The $K(n)$-Local Universal Character}

In this section we give a full description of the universal character of a
$K(n)$-local object. We begin in \cref{sec: Trans universal} by showing that the transchromatic character is the universal character of $\En$, see
\cref{cor: indtifying the spliting algebra}.

Next, in \cref{sec: universal character K(n)} we observe that one can run a
construction analogous to the height $t$ splitting algebra for any object
\[
M \in \Mod_{\SS_{K(n)}[\omega^{(n)}_{p^\infty}]}(\Sp_{K(n)}),
\]
which we denote by $C_t(M)$, see \cref{const: Ct}. (In particular,
$C_t(\En)=C_t^n$.) Using descent theory, we deduce that this construction
computes the universal character of any such $M$, namely
\[
(\L_p^{n-t})^*(\L^{n-t}_{p})_!M \simeq C_t(M),
\]
see \cref{thm: k(n) local character with roots of unity}. In particular, for any $\pi$-finite space $A$ we obtain a character isomorphism
\[
C_t(M^A) \simeq C_t(M)^{\L_p^{n-t}A}.
\]
Under additional finiteness assumptions (satisfied for $\En$), this can be
rewritten as a base-change formula, recovering the usual character isomorphism of \cite{Hopkins-Kuhn-Ravanel-2000-HKR, Lurie-2019-Elliptic3, Stapleton-2013-HKR}, see \cref{cor:L!-tensor-RA}.

We then use this to compute the universal character of an arbitrary $K(n)$-local object by adding roots of unity, applying $C_t$, and taking fixed points, see \cref{cor: L!-as-hZ}. Finally, in \cref{sec: GLn-t action} we show that the
functors $(\L_p^{n-t})^*$ and $(\L^{n-t}_p)_!$ carry a natural action of the profinite group $\GL_{n-t}(\ZZ_p)$, and that the universal character map
\[
M^{(-)} \to (\L_p^{n-t})^*(\L^{n-t}_p)_!M
\]
is $\GL_{n-t}(\ZZ_p)$-equivariant when the source is given the trivial action,
see \cref{lem:GL-action-Ln-t}.

In the case $t=0$, this identifies fixed points: for any
$M\in \Sp_{K(n)}$ we have
\[
L_\QQ (M^{(-)}) \simeq ((\L_p^{n})^*(\L_p^{n})_!M)^{h\GL_n(\ZZ_p)},
\]
see \cref{thm: unit-rational-hGL}. When $M=R$ is a $K(n)$-local algebra with all height $n$ roots of unity, the resulting action is particularly well behaved, and the universal character exhibits $(\L_p^{n})^*(\L_p^{n})_!R$ as a $\GL_n(\ZZ_p)$ Galois extension of $L_{\QQ}(R^{(-)})$ in
$\PCMoninf(\Sp_{\QQ})$, see \cref{prop: Rognes-Galois-Ln-PCMon}. Under further finiteness assumptions on $R$ (satisfied for $\En$), the same map evaluated at $\pt$ exhibits $(\L_p^{n})_!R$ as a Galois extension of $L_{\QQ}(R)$ in $\Sp_{\QQ}$, see \cref{prop:L!-R-Galois}. To simplify notation throughout this section we omit the $p$ and write $\L^{n-t}=\L_p^{n-t}$ for the $p$-typical free loop space.

\subsection{Universality of the Transchromatic Character}\label{sec: Trans universal}

In this subsection we show that the transchromatic character agrees with the universal character
of $\En$. We begin by analyzing the $(n-t)$-fold free loop space of the cup product map
\[
(\B\ZZ/p^r)^{n-t} \to \B^{n-t}\ZZ/p^r .
\]
Fix an orientation on $S^1$. This induces a canonical orientation on $(S^1)^{n-t}$. By \cref{lem: F(B^nZ/pr) splits}, such an orientation yields a splitting of $\ZZ/p^r$-modules
\[
\Map((S^1)^{n-t},\B^{n-t}\ZZ/p^r)
\simeq
\ZZ/p^r \times \tau_{\ge 1}\Map((S^1)^{n-t},\B^{n-t}\ZZ/p^r).
\]
Moreover, essentially the same argument gives a splitting of $\ZZ/p^r$-modules
\[
\Map((S^1)^{n-t},\B(\ZZ/p^r)^{n-t})
\simeq
\prod_{i=1}^{n-t}(\ZZ/p^r)^{n-t}
\times
\tau_{\ge 1}\Map((S^1)^{n-t},\B(\ZZ/p^r)^{n-t}).
\]
We now verify that the relevant diagram commutes.

\begin{lemma} \label{lem: orientation diag commutes}
    With the choice of splitting above, the diagram
    \[
    \begin{tikzcd}
	{\ZZ/p^r} & {\Map((S^1)^{n-t},\B^{n-t}\ZZ/p^r)} & {\ZZ/p^r} \\
	{\prod_{i=1}^{n-t}(\ZZ/p^r)^{n-t}} & {\Map((S^1)^{n-t},\B(\ZZ/p^r)^{n-t})} & {\prod_{i=1}^{n-t}(\ZZ/p^r)^{n-t}}
	\arrow[from=1-1, to=1-2]
	\arrow[from=1-2, to=1-3]
	\arrow["{\pi_0(\L^{n-t}\cup^{n-t})}", from=2-1, to=1-1]
	\arrow[from=2-1, to=2-2]
	\arrow["{\L^{n-t}\cup^{n-t}}", from=2-2, to=1-2]
	\arrow[from=2-2, to=2-3]
	\arrow["{\pi_0(\L^{n-t}\cup^{n-t})}"', from=2-3, to=1-3]
    \end{tikzcd}
    \]
    commutes.
\end{lemma}

\begin{proof}
    Set $m:=n-t$, $A:=\ZZ/p^r$, and $T^m:=(S^1)^m$. By \cref{lem: F(B^nZ/pr) splits} applied to $\Map(T^m,-)$,
    the top row splitting
    \[
    \Map(T^m,\B^mA)\simeq A\times \tau_{\ge 1}\Map(T^m,\B^mA)
    \]
    is obtained by choosing the generator in
    $\pi_0\Map(T^m,\B^mA)\simeq H^m(T^m,A)$
    corresponding to the fundamental class of $T^m$.
    Likewise, the bottom row splitting identifies $\pi_0\Map(T^m,(\B A)^m)\simeq (H^1(T^m,A))^m\simeq (A^m)^m$
    using the basis $x_1,\dots,x_m\in H^1(T^m,A)$ coming from the oriented circle factors.
    
    The map $\L^m\cup^m:(\B A)^m\to \B^mA$ classifies the $m$-fold cup product, hence on $\pi_0$ it induces
    \[
    (H^1(T^m,A))^m \xrightarrow{\ \smile\cdots\smile\ } H^m(T^m,A).
    \]
    Under the above identifications, this sends $(x_1,\dots,x_m)$ to
    $x_1\smile\cdots\smile x_m$, which is exactly the top generator determined by the product orientation. Therefore $\L^m\cup^m$ carries the bottom section to the top section, and functoriality of $\pi_0$ gives commutativity of both squares, hence of the whole diagram.
\end{proof}

We now identify this map as the determinant map.

\begin{lemma}\label{lem: Lcup is det}
    Let $m=n-t$ and $A=\ZZ/p^r$. Under the identifications
    \[
    \pi_0\Map(T^m,(\B A)^m)\simeq(H^1(T^m,A))^m \simeq (A^m)^m \simeq M_m(A),
    \qquad
    \pi_0\Map(T^m,\B^mA) \simeq H^m(T^m,A)\simeq A
    \]
    coming from the chosen orientation on $T^m$, the map $\pi_0(\L^m\cup^m): (A^m)^m \to A$ is the determinant map.
\end{lemma}

\begin{proof}
    Let $T^m=(S^1)^m$ and let $x_1,\dots,x_m\in H^1(T^m,A)$ be the standard basis classes
    pulled back from the circle factors (determined by the fixed orientation on $S^1$).
    Then
    \[
    H^*(T^m,A)\simeq \Lambda_A(x_1,\dots,x_m),
    \]
    and the chosen product orientation identifies $H^m(T^m,A)\simeq A$ by sending
    $x_1\smile\cdots\smile x_m$ to $1\in A$.
    
    By definition, $\L^m\cup^m:(\B A)^m\to \B^mA$ classifies the $m$-fold cup product, so on
    $\pi_0$ it induces
    \[
    (H^1(T^m,A))^m \xrightarrow{\ \smile\cdots\smile\ } H^m(T^m,A)\simeq A.
    \]
    Now take $(\alpha_1,\dots,\alpha_m)\in (H^1(T^m,A))^m$, and write
    \[
    \alpha_i=\sum_{j=1}^m a_{ij}x_j \qquad (a_{ij}\in A).
    \]
    Using multilinearity and the exterior-algebra relation (so $x_j\smile x_j=0$ and
    $x_j\smile x_k=-\,x_k\smile x_j$), we expand:
    \[
    \alpha_1\smile\cdots\smile \alpha_m
    =
    \sum_{j_1,\dots,j_m} a_{1j_1}\cdots a_{mj_m}\, x_{j_1}\smile\cdots\smile x_{j_m}.
    \]
    All terms with repeated indices vanish, and reordering each surviving term to
    $x_1\smile\cdots\smile x_m$ contributes the sign of the corresponding permutation. Hence
    \[
    \alpha_1\smile\cdots\smile \alpha_m
    =
    (\sum_{\sigma\in S_m}\mathrm{sgn}(\sigma)\,a_{1\sigma(1)}\cdots a_{m\sigma(m)})
    \,(x_1\smile\cdots\smile x_m)
    =
    \det(a_{ij})\,(x_1\smile\cdots\smile x_m).
    \]
    Under the orientation identification $H^m(T^m,A)\simeq A$, this maps to $\det(a_{ij})\in A$.
    Thus $\pi_0(\L^m\cup^m)$ is the determinant map.
\end{proof}

With this in hand, we turn to showing that the transchromatic character is the universal character of $\En$. We first introduce some notation and then state the key lemma that will be
used in the proof.

\begin{notation}\label{not: evI}
    Given $\cC \in \PrL$, $R \in \cC$, and $n-t\in \NN$, we introduce the following notation:
    \begin{enumerate}
    \item The map
        \[
        \ev_{/p^r} \colon  R^{\L^{\,n-t}(\B\ZZ/p^r)^{\,n-t}} \to R
        \]
    is defined as evaluation at the mod $p^r$ map  $\ZZ_{p}^{\,n-t} \to (\ZZ/p^r)^{\,n-t}$.
    \item The map
        \[
        \ev_{I} \colon R^{M_{n-t}} \to R
        \]
        is defined as evaluation at the unit matrix.
\end{enumerate}

\end{notation}

\begin{remark}\label{rem: why evI}
    The mod $p^r$ morphism $\ZZ_{p}^{\,n-t} \to (\ZZ/p^r)^{\,n-t}$ corresponds to the unit matrix in $\pi_0 \Map(\B\ZZ_{p}^{\,n-t}, (\B\ZZ/p^r)^{\,n-t}) \simeq M_{n-t}(\ZZ/p^r)$.
    Therefore, using the section coming from our chosen orientation we obtain a commutative diagram
    \[
    \begin{tikzcd}
	   {R^{\L^{n-t}(\B\ZZ/p^r)^{n-t}}} & {R^{M_{n-t}(\ZZ/p^r)}} \\
	   & R
	   \arrow["{s^{*}}", from=1-1, to=1-2]
	   \arrow["{\ev_{/p^r}}"', from=1-1, to=2-2]
	   \arrow["{\ev_I}", from=1-2, to=2-2]
    \end{tikzcd}
    \]
    where $\ev_{/p^r}$ and $\ev_I$ are as in \cref{not: evI}.
\end{remark}

Our goal now is to show that the canonical map inverts the relevant idempotent.

\begin{corollary} \label{cor: chooses root of unity}
    Let $R,S\in\CAlg(\tsadi)$ and 
    \[
    \chi\colon R \to (\L^{n-t})^*S
    \]
    be a character map of commutative algebras. Assume that $R$ is $T(n)$-local and $(\ZZ/p^r,n)$-oriented. Let $\varphi =\varphi_{\ZZ/p^r,S}^R$ be from \cref{cons: candidate Fourier transform}. Then the composition
    \[
    L_{T(t)}(R)[\B^t\ZZ/p^r]\xto{\varphi} S^{\I^{t}\B^t\ZZ/p^r}\simeq S^{\ZZ/p^r}\xto{\L^{n-t}\cup_{n-t}^*} S^{\L^{n-t}\B(\ZZ/p^r)^{n-t}}\xto{ev_{/p^r}} S
    \]
    chooses a primitive root of unity under the adjunction $R[-]\dashv (-)\units$.
\end{corollary}

\begin{proof}
    Recall that $\varphi$ factors as
    \[
    L_{T(t)}(R)[\B^t\ZZ/p^r]\to S^{L^{n-t}\I^{n}\B^{t}\ZZ/p^r}
    \simeq S^{L^{n-t}\B^{n-t}\ZZ/p^r}\to S^{\ZZ/p^r},
    \]
    where the last arrow is obtained by collapsing the top cell of $\TT^{\,n-t}$ and using the induced orientation on $S^{n-t}$ (see \cref{cons: candidate Fourier transform}). Hence this last map agrees with the projection map
    determined by our chosen orientation splitting
    \[
    \L^{n-t}\B^{n-t}\ZZ/p^r \simeq \ZZ/p^r \times \tau_{\ge 1}\L^{n-t}\B^{n-t}\ZZ/p^r,
    \]
    i.e.\ the induced map $S^{L^{n-t}\B^{n-t}\ZZ/p^r}\to S^{\ZZ/p^r}$ is exactly the one coming from
    this splitting. Then by \cref{lem: orientation diag commutes}, \cref{lem: Lcup is det} and \cref{rem: why evI} we have a commutative diagram:
    \[
    \begin{tikzcd}
	{L_{T(t)}(R)[\B^t\ZZ/p^r]} & {S^{L^{n-t}\B^{n-t}\ZZ/p^r}} & {S^{\ZZ/p^r}} \\
	{} & {S^{L^{n-t}(\B\ZZ/p^r)^{n-t}}} & {S^{M_{n-t}(\ZZ/p^r)}} \\
	&& S
	\arrow[from=1-1, to=1-2]
	\arrow["\varphi", shift left, curve={height=-12pt}, from=1-1, to=1-3]
	\arrow[from=1-2, to=1-3]
	\arrow["{(L^{n-t}\cup_{n-t})^*}"', from=1-2, to=2-2]
	\arrow["{\det^*}", from=1-3, to=2-3]
	\arrow[from=2-2, to=2-3]
	\arrow["{\ev_{/p^r}}"', from=2-2, to=3-3]
	\arrow["{\ev_I}", from=2-3, to=3-3]
    \end{tikzcd}
    \]
    As the composition $\ev_I\circ\det^*$ is just the projection to the coordinate $1\in \ZZ/p^r$ the claim follows by \cref{prop: Fourier in chr}. 
\end{proof}

\begin{remark}\label{rem: L! module over}
    Note that, by definition of $\L_!$, for any
    \[
    R\in \CAlg(\Mod_{\SS_{T(n)}[\omega^{(n)}_{p^\infty}]}(\Sp_{T(n)}))
    \quad\text{and}\quad r\in \NN,
    \]
    the diagram
    \[
    \begin{tikzcd}
    	{L_{T(t)}(R^{(-)})[\B^t\ZZ/p^r]} &
    	{L_{T(t)}(R[\B^t\ZZ/p^r]^{(-)})\simeq L_{T(t)}(R^{\B^{n-t}\ZZ/p^r\times (-)})} &
    	{\L_!^{n-t}(R)^{\L^{n-t}(\B^{n-t}\ZZ/p^r\times (-))}} \\
    	&
    	{L_{T(t)}(R^{\B(\ZZ/p^r)^{n-t}\times (-)})} &
    	{\L_!^{n-t}(R)^{\L^{n-t}(\B(\ZZ/p^r)^{n-t}\times (-))}} \\
    	&& {\L_!(R)^{\L^{n-t}(-)}}
    	\arrow[from=1-1, to=1-2]
    	\arrow[from=1-2, to=1-3]
    	\arrow[from=1-2, to=2-2]
    	\arrow[from=1-3, to=2-3]
    	\arrow[from=2-2, to=2-3]
    	\arrow[from=2-3, to=3-3]
    \end{tikzcd}
    \]
    commutes, and by \cref{cor: chooses root of unity} the composite chooses a
    height $t$ primitive $p^r$-th root of unity. In particular, $(\L^{n-t})^*\L_!R$ is naturally
    a module over both
    \[
    L_{T(t)}(R^{(-)})[\omega^{(t)}_{p^\infty}]
    \qquad\text{and}\qquad
    L_{T(t)}(R^{\B(\ZZ/p^r)^{n-t}\times (-)})[e_r^{-1}],
    \]
    where $e_r$ is the image of the idempotent splitting off the primitive height $t$ root of unity in $L_{T(t)}(R)[\B^t\ZZ/p^r]$.
    
\end{remark}

To put everything together we will also need the following lemma that follows immediately from the main result of \cite{ben-mosh-Transchromatic}.

\begin{lemma} \label{lem : maps from CtL determined by a point}
Let $M\in\Mod_{L_{T(t)}(\En^{(-)})}\bigl(\PCMoninf(\SpTn[t])\bigr)$.
Then evaluation at the point induces a natural equivalence
\[
        \Map_{\Mod_{L_{T(t)}(\En^{(-)})}(\PCMoninf(\SpTn[t]))}
        \bigl(C_t^{L_p^{n-t}(-)},M\bigr)
        \isoto
        \Map_{\Mod_{L_{T(t)}\En}(\SpTn[t])}
        \bigl(C_t,M(\pt)\bigr).
\]
In particular, every $L_{T(t)}(\En^{(-)})$-linear map
\[
        C_t^{L_p^{n-t}(-)} \longrightarrow M
\]
is determined by its value at $\pt$.
\end{lemma}

\begin{proof}
The forgetful functor is evaluation at the point,
\[
        \ev_{\pt}\colon
        \Mod_{L_{T(t)}(\En^{(-)})}\bigl(\PCMoninf(\SpTn[t])\bigr)
        \longrightarrow
        \Mod_{L_{T(t)}\En}(\SpTn[t]),
        \qquad
        M\longmapsto M(\pt).
\]
Its left adjoint is the free functor
\[
        N\longmapsto
        L_{T(t)}(\En^{(-)})\otimes_{L_{T(t)}\En}N .
\]
Taking $N=C_t$, and using \cite[Theorem A]{ben-mosh-Transchromatic}, we get
an equivalence
\[
        L_{T(t)}(\En^{(-)})\otimes_{L_{T(t)}\En}C_t
        \isoto
        C_t^{L_p^{n-t}(-)}
\]
in $\Mod_{L_{T(t)}(\En^{(-)})} \bigl(\PCMoninf(\SpTn[t])\bigr)$. Thus $C_t^{L_p^{n-t}(-)}$ is the free
$L_{T(t)}(\En^{(-)})$-module precommutative monoid generated on $C_t$ at the
point. Applying the free-forgetful adjunction gives
\[
\begin{aligned}
        \Map_{\Mod_{L_{T(t)}(\En^{(-)})}(\PCMoninf(\SpTn[t]))}
        \bigl(C_t^{L_p^{n-t}(-)},M\bigr)
        &\simeq
        \Map_{\Mod_{L_{T(t)}\En}(\SpTn[t])}
        \bigl(C_t,M(\pt)\bigr).
\end{aligned}
\]
\end{proof}

We now put everything together and identify the universal character of $\En$.

\begin{corollary}\label{cor: indtifying the spliting algebra} \
    The canonical map 
    \[
    \L_!^{n-t}\En \to C^{n}_t
    \]
    that factors the transchromatic character is an equivalence.
\end{corollary}

\begin{proof}
    By \cref{cor: chooses root of unity} and \cref{rem: L! module over}, for every $r$ the unit map 
    \[
    \En \to (\L^{n-t})^*\L_!^{n-t}\En
    \]
    factors through
    \[
     L_{T(t)}(\En^{\B\Zpr \times (-)})[e_r^{-1}] \to (\L^{n-t})^*\L_!^{n-t}\En.
    \]
    By \cref{cor: classifying the splitting algebra}, the fact that $\En^{\B\Zpr}$ is finite and free over $\En$ and \cite[Theorem A]{ben-mosh-Transchromatic} we get that the canonical map
    \[
    \En^{(-)}\to  \colim L_{T(t)}(\En^{\B\Zpr \times (-)})[e_r^{-1}] \simeq C_t^n \otimes_{\En} L_T(t)(\En^{(-)})\simeq (\L^{n-t})^*C_t^n
    \]
    is the transchromatic character. Here the tensor product is $T(t)$-completed.
    Taking colimit over $r$, we obtain a commutative diagram:
    \[
    \begin{tikzcd}
	& \En^{(-)} \\
	{(\L^{n-t})^*C^n_t} & {(\L^{n-t})^*\L_!^{n-t}\En} 
	\arrow["{\chi^\mrm{tch}}"', from=1-2, to=2-1]
	\arrow["{\chi^\un}"{description}, from=1-2, to=2-2]
	\arrow[from=2-1, to=2-2]
    \end{tikzcd} 
    \]
    By \cref{lem : maps from CtL determined by a point} the map $(\L^{n-t})^*C^n_t\to (\L^{n-t})^*\L_!^{n-t}\En$ is $(\L^{n-t})^*$ applied to a map of  $\infty$-commutative monoids and
    by the universal property of the universal character we also have a map
    \[
    (\L^{n-t})^*\L_!^{n-t}\En \to (\L^{n-t})^*C^n_t
    \]
    factoring the transchromatic character. Thus it remains to check that the composition
    \[
    (\L^{n-t})^*C^n_t \to (\L^{n-t})^*\L_!^{n-t}\En \to (\L^{n-t})^*C_t^n
    \]
    is the identity. This follows as by \cref{rem: L! module over} for every $A$ the above map is a map of $\colim L_{T(t)}(\En^{\B\Zpr \times A})$ algebras and $(C^n_t)^{\L^{n-t}A}$ is obtained from $\colim L_{T(t)}(\En^{\B\Zpr \times A})$ by inverting an element.
\end{proof}

\begin{remark}\label{rem:splitting-algebra-proof_PCMon}
    The proof of \cref{cor: indtifying the spliting algebra} in fact shows slightly more. Namely, there are natural equivalences
    \[
    (\L^{n-t})^{*}C_t^{n}
    \;\simeq\;
    \En^{(-)}\widehat{\otimes}_{\En} C_t^{n}
    \;\simeq\;
    \colim_r L_{T(t)}(\En^{\B(\ZZ/p^{r})^{n-t}\times(-)})
    \;\simeq\;
    (\L^{n-t})^{*}\L^{\,n-t}_{!}\En ,
    \]
    where the colimit is taken in $\PCMoninf(\Sp_{T(t)})$. The same proof also gives an analogous formula for the universal character of $\En^{A\times(-)}$.
\end{remark}

\subsection{The $K(n)$-Local Universal Character} \label{sec: universal character K(n)}

We now bootstrap our computation of the universal character of $\En{}$ to any $K(n)$-local object.

We start by extending the construction of the splitting algebra to any element in $\Mod_{\SS_{K(n)}[\omega^{(n)}_{p^{\infty}}]}(\Sp_{K(n)})$.

\begin{construction}\label{const: Ct}
    We define
    \[
    C_{r,t} \colon \Mod_{\SS_{K(n)}[\omega^{(n)}_{p^{\infty}}]}(\Sp_{K(n)}) \to \PCMoninf(\Sp_{T(t)})
    \]
    as the following composite. First, we apply fixed points for $\B(\ZZ/p^r)^{n-t}$:
    \[
    \Mod_{\SS_{K(n)}[\omega^{(n)}_{p^{\infty}}]}(\Sp_{K(n)})
    \to
    \Mod_{\SS_{K(n)}[\omega^{(n)}_{p^{\infty}}]^{\B(\ZZ/p^r)^{n-t}}}(\Sp_{K(n)}),
    \qquad
    M \mapsto M^{\B(\ZZ/p^r)^{n-t}}.
    \]
    Next, we restrict scalars along
    \[
    \SS_{K(n)}[\omega^{(n)}_{p^{\infty}}][\B^{n-t}\ZZ/p^r]
    \simeq
    \SS_{K(n)}[\omega^{(n)}_{p^{\infty}}]^{\B^{n-t}\ZZ/p^r}
    \to
    \SS_{K(n)}[\omega^{(n)}_{p^{\infty}}]^{\B(\ZZ/p^r)^{n-t}},
    \]
    where the equivalence is the semiadditive Fourier transform and the arrow is induced by restriction along the cup product
    \[
    \B(\ZZ/p^r)^{n-t} \to \B^{n-t}\ZZ/p^r.
    \]
    This yields a functor
    \[
    \Mod_{\SS_{K(n)}[\omega^{(n)}_{p^{\infty}}]^{\B(\ZZ/p^r)^{n-t}}}(\Sp_{K(n)})
    \to
    \Mod_{\SS_{K(n)}[\omega^{(n)}_{p^{\infty}}][\B^{n-t}\ZZ/p^r]}(\Sp_{K(n)}).
    \]
    We then $T(t)$-localize pointwise and restrict scalars along the assembly map
    \[
    L_{T(t)}(\SS_{K(n)}[\omega^{(n)}_{p^{\infty}}]^{(-)})[\B^{t}\ZZ/p^r]
    \to
    L_{T(t)}(\SS_{K(n)}[\omega^{(n)}_{p^{\infty}}][\B^{t}\ZZ/p^r]^{(-)}),
    \]
    obtaining a functor
    \[
    \Mod_{\SS_{K(n)}[\omega^{(n)}_{p^{\infty}}][\B^{t}\ZZ/p^r]}(\Sp_{K(n)})
    \to
    \Mod_{L_{T(t)}(\SS_{K(n)}[\omega^{(n)}_{p^{\infty}}]^{(-)})[\B^{t}\ZZ/p^r]}(\PCMoninf
    (\Sp_{T(t)})).
    \]
    Finally, we base change along
    \[
    L_{T(t)}(\SS_{K(n)}[\omega^{(n)}_{p^{\infty}}]^{(-)})[\B^{t}\ZZ/p^r]
    \to
    L_{T(t)}(\SS_{K(n)}[\omega^{(n)}_{p^{\infty}}]^{(-)})[\omega^{(t)}_{p^\infty}],
    \]
    and then apply the forgetful functor to $\Sp_{T(t)}$, yielding
    \[
    \Mod_{L_{T(t)}(\SS_{K(n)}[\omega^{(n)}_{p^{\infty}}])[\B^{t}\ZZ/p^r]}(\Sp_{T(t)})
    \to
    \PCMoninf(\Sp_{T(t)}).
    \]
    By construction we have natural maps $C_{r,t} \to C_{r+1,t}$, and we set
    \[
    C_t := \colim_r C_{r,t},
    \]
    where the colimit is taken in $\PCMoninf(\Sp_{T(t)})$.
\end{construction}

\begin{remark}
    Unwinding the definition of $C_t$ in \cref{const: Ct}, we may identify the functor
    \[
    C_t \colon \Mod_{\SS_{K(n)}[\omega^{(n)}_{p^{\infty}}]}(\Sp_{K(n)}) \to \PCMoninf(\Sp_{T(t)})
    \]
    as
    \[
    C_t(M)\simeq \colim_{r}\, L_{T(t)}\!(M^{\B(\ZZ/p^r)^{n-t}\times (-)})[e_r^{-1}],
    \]
    where $e_r$ is the idempotent in $L_{T(t)}(\SS_{K(n)}[\omega^{(n)}_{p^{\infty}}])[\B^t\ZZ/p^r]$ splitting off the
    primitive root of unity. The
    $L_{T(t)}(\SS_{K(n)}[\omega^{(n)}_{p^{\infty}}])[\B^t\ZZ/p^r]$-module structure on $M$ is obtained by restriction of scalars along the composite
    \begin{align*}
    L_{T(t)}(\SS_{K(n)}[\omega^{(n)}_{p^{\infty}}])[\B^t\ZZ/p^r]
    &\to L_{T(t)}\!\bigl(\SS_{K(n)}[\omega^{(n)}_{p^{\infty}}][\B^t\ZZ/p^r]) \\
    &\simeq L_{T(t)}\!\bigl(\SS_{K(n)}[\omega^{(n)}_{p^{\infty}}]^{\B^{n-t}\ZZ/p^r}) \\
    &\to L_{T(t)}\!(\SS_{K(n)}[\omega^{(n)}_{p^{\infty}}]^{\B(\ZZ/p^r)^{n-t}}).
    \end{align*}
\end{remark}

We note that by \cref{cor: classifying the splitting algebra}  $C_{t}(\En)=C^n_t$. We will now show that there is always a map $C_t(M) \to (\L^{n-t})^*\L_!^{n-t}(M) $.

\begin{lemma}\label{lem: eq Ft L! candidate}
    There is a natural transformation $C_t(-) \to (\L^{n-t})^*\L_!^{n-t}(-)$ extending the equivalence $\L_!\En \simeq C^n_t$. Here we abuse notation and write $\L_!^{n-t}(-)$ also for the restriction of $\L_!^{n-t}(-)$ to $\Mod_{\SS_{K(n)}[\omega^{(n)}_{p^{\infty}}]}(\Sp_{K(n)})$ and think of $(\L^{(n-t)})^*$ as having values in $\PCMoninf(\Sp_{T(t)})$.
\end{lemma}

\begin{proof}
    By \cref{rem: L! module over}, the functor $(\L^{n-t})^*\L^{n-t}_!$ is a module over
    \[
    L_{T(t)}(\SS_{K(n)}[\omega^{(n)}_{p^{\infty}}]^{\B(\ZZ/p^r)^{n-t}\times (-)})[e_r^{-1}].
    \]
    In particular, the unit natural transformation
    \[
    \id \to (\L^{n-t})^*\L^{n-t}_! \in \Mod_{L_{T(t)}\SS_{K(n)}[\omega^{(n)}_{p^{\infty}}]}(\PCMoninf(\Sp))
    \]
    factors through extension of scalars along $L_{T(t)}(\SS_{K(n)}[\omega^{(n)}_{p^{\infty}}])\to L_{T(t)}(\SS_{K(n)}[\omega^{(n)}_{p^{\infty}}]^{\B(\ZZ/p^r)^{n-t}\times (-)})[e_r^{-1}]$. Concretely,
    this gives a natural map
    \[
    L_{T(t)}(\SS_{K(n)}[\omega^{(n)}_{p^{\infty}}]^{\B(\ZZ/p^r)^{n-t}\times (-)})[e_r^{-1}]
    \otimes_{L_{T(t)}(\SS_{K(n)}[\omega^{(n)}_{p^{\infty}}])} \id
    \to
    (\L^{n-t})^*\L^{n-t}_!.
    \]
    By \cref{cor: base change} we can identify the source with $C_{t,r}$. Taking colimit over $r$ gives the required natural transformation.
\end{proof}

We note that both functors
\[
C_t,\ (\L^{n-t})^*\L_!^{n-t}\colon \Mod_{\SS_{K(n)}[\omega^{(n)}_{p^{\infty}}]}(\Sp_{K(n)}) \to \PCMoninf(\Sp_{T(t)})
\]
commute with finite limits. Our goal is to show that they are equivalent. We will do this by first proving that they agree on $\En$-modules, and then applying descent theory. We will start by defining a variant of $C^t$ on $\Mod_{\En^{(-)}}(\PCMoninf(\Sp))$.

\begin{construction}
    We define a functor
    \[
    \overline{C}_{t,r} \colon \Mod_{\En^{(-)}}(\PCMoninf(\Sp))
    \to
    \Mod_{L_{T(t)}(\En^{(-)})[\omega^{(t)}_{p^\infty}]}(\PCMoninf(\Sp_{T(t)}))
    \]
    as follows. We first evaluate at $\B(\ZZ/p^r)^{n-t}$, obtaining a functor
    \[
    \Mod_{\En^{(-)}}(\PCMoninf(\Sp))
    \to
    \Mod_{\En^{\B(\ZZ/p^r)^{n-t}\times (-)}}(\PCMoninf(\Sp)).
    \]
    Next, we apply $T(t)$-localization and then restrict scalars along the composite
    \[
    L_{T(t)}(\En^{(-)})[\B^{t}\ZZ/p^r]
    \to
    L_{T(t)}(\En{} [\B^{t}\ZZ/p^r]^{(-)})
    \simeq
    L_{T(t)}\En^{\B^{n-t}\ZZ/p^r\times (-)}
    \to
    L_{T(t)}\En^{\B(\ZZ/p^r)^{n-t} \times (-)},
    \]
    thereby obtaining a functor
    \[
    \Mod_{\En^{\B(\ZZ/p^r)^{n-t}\times (-)}}(\PCMoninf(\Sp))
    \to
    \Mod_{L_{T(t)}(\En^{(-)})[\B^{t}\ZZ/p^r]}(\PCMoninf(\Sp_{T(t)})).
    \]
    Finally, we base change along
    \[
    L_{T(t)}(\En^{(-)})[\B^{t}\ZZ/p^r]
    \to
    L_{T(t)}(\En^{(-)})[\omega^{(t)}_{p^\infty}],
    \]
    which yields a functor
    \[
    \Mod_{L_{T(t)}(\En^{(-)})[\B^{t}\ZZ/p^r]}(\PCMoninf(\Sp_{T(t)}))
    \to
    \Mod_{L_{T(t)}(\En^{(-)})[\omega^{(t)}_{p^\infty}]}(\PCMoninf(\Sp_{T(t)})).
    \]
    The composite is $\overline{C}_{t,r}$. We then set
    \[
    \overline{C}_{t} := \colim_r \overline{C}_{t,r}.
    \]
\end{construction}

We record two immediate consequences of the construction.

\begin{lemma}\label{lem: alpha omega and Ft with and without bar}
\begin{enumerate}
    \item A choice of a compatible family $\{\omega^{(n)}_{p^r}\}_r$ of height $n$ primitive roots of unity in $\En$ determines a natural transformation
    \[
     \overline{C}_t(-)\to (\L^{n-t})^*L_{T(t)}^\tsadi\L_!^{n-t}(-)
    \quad\in\quad
    \Fun\bigl(\Mod_{\En^{(-)}}(\PCMoninf(\Sp)),\ \Mod_{L_{T(t)}(\En^{(-)})[\omega^{(t)}_{p^\infty}]}(\PCMoninf(\Sp_{T(t)})\bigr)
    \]
    extending the equivalence $L_!\En \simeq C^n_t$.  
    \item For $M\in \Mod_{\En} (\Sp_{K(n)})\subseteq \Mod_{\En^{(-)}}(\PCMoninf(\Sp_{K(n)}))$ we have $\overline{C}_t(M)=C_t(M)$.
\end{enumerate}
\end{lemma}

\begin{proof}
    Part (2) is immediate from the definition of $\overline{C_t}$, and the proof of (1) is essentially the same as the proof of \cref{lem: eq Ft L! candidate}.
\end{proof}

\begin{lemma}\label{lem: L*-colimit}
The functor
\[
(\L^{n-t})^* \colon \Sp_{T(t)} \simeq \CMoninf(\Sp_{T(t)}) \to \PCMoninf(\Sp_{T(t)})
\]
commutes with colimits.
\end{lemma}

\begin{proof}
Recall that $\CMoninf(\Sp_{T(t)})$ is the full subcategory of
$\PCMoninf(\Sp_{T(t)})$ spanned by those functors that commute with limits indexed by $p$-local $\pi$-finite spaces. Thus it suffices to note that in $\Sp_{T(t)}$ limits and colimits indexed by a $p$-local $\pi$-finite space agree.
It follows that the inclusion
\[
\CMoninf(\Sp_{T(t)}) \to \PCMoninf(\Sp_{T(t)})
\]
preserves colimits, hence so does $(\L^{n-t})^*$.
\end{proof}

\begin{lemma}\label{lem: alpha omega is eq}
    The natural transformation $\overline{C}_t(-)\to (\L^{n-t})^*L_{T(t)}^\tsadi\L^{n-t}_!(-)$ from \cref{lem: alpha omega and Ft with and without bar}~(1) is an equivalence.
\end{lemma}

\begin{proof}
    By construction, $\overline{C}_t$ commutes with colimits. Moreover, by
    \cref{lem: L*-colimit}, the composite $(\L^{n-t})^* L_{T(t)}^\tsadi\L^{n-t}_!$ commutes with colimits. Hence it suffices to check that the map is an equivalence on objects of the form
    \[
    \En^{(-)}\otimes \Map(A,-)\simeq \En^{A\times (-)},
    \]
    where the displayed equivalence is \cref{cor: base change}.
    
    Using \cref{cor: base change}, the functor $\overline{C}_{t,r}$ is identified with extension of scalars, so that
    \[
    \overline{C}_{t,r}(\En^{A\times(-)})
    \simeq
    L_{T(t)}(\En^{\B(\ZZ/p^r)^{n-t}\times (-)})[e_r^{-1}]
    \otimes_{L_{T(t)}(\En)}
    L_{T(t)}(\En^{A\times(-)}).
    \]
    On the other hand, by \cref{cor: indtifying the spliting algebra} we have
    \[
    (\L^{n-t})^*\L^{n-t}_!(\En)
    \simeq
   \colim_r L_{T(t)}(\En^{\B(\ZZ/p^r)^{n-t}\times (-)})[e_r^{-1}].
    \]
    Therefore, using \cref{thm: F! of R A} \footnote{One can use \cref{rem:splitting-algebra-proof_PCMon} instead}, we compute
    \begin{align*}
    (\L^{n-t})^*\L^{n-t}_!(\En^{A\times(-)})
    &\simeq
    (\L^{n-t})^*\L^{n-t}_!(\En)\otimes_{L_{T(t)}(\En)} L_{T(t)}(\En^{A\times(-)}) \\
    &\simeq
    \colim_r L_{T(t)}(\En^{\B(\ZZ/p^r)^{n-t}\times (-)})[e_r^{-1}]
    \otimes_{L_{T(t)}(\En)}
    L_{T(t)}(\En^{A\times(-)}),
    \end{align*}
    which agrees with the formula for $\overline{C}_t(\En^{A\times(-)}\bigr)$.
    This identifies the natural transformation on these generators with an
    equivalence, and hence it is an equivalence on all objects.
\end{proof}

We now recall the following result from descendability theory.

\begin{lemma}\label{lem: dedent for fix point by close subgroup}
    Let $H \subseteq \GG_n$ be a closed subgroup of the Morava stabilizer group. Then the essential image of the forgetful functor
    \[
    \Mod_{\En} \to \Mod_{\En^{h_cH}}^{\wedge}
    \]
    generates $\Mod_{\En^{h_cH}}^{\wedge}$ under finite limits. Here $(-)^{h_cH}$ denotes continuous fixed points (see \cite{behrens_Davis_profinte_gal} for the relevant definitions).
\end{lemma}

\begin{proof}
    In the terminology of \cite{Mathew-2016-Galois}, the claim is that
    the map
    \[
    \En^{h_cH}\longrightarrow\En
    \]
    is descendable. This is precisely
    \cite[Corollary~3.2.10(1)]{Li-Zhang-profinite-descent}.
\end{proof}

We now assemble the preceding results to obtain an explicit formula for the universal character of a $\SS_{K(n)}[\omega^{(n)}_{p^\infty}]$-module.

\begin{theorem}\label{thm: k(n) local character with roots of unity}
    The map $C_t \to (\L^{n-t})^*\L^{n-t}_!$ from \cref{lem: eq Ft L! candidate} is an equivalence. 
\end{theorem}

\begin{proof}
    Since both $C_t$ and $(\L^{n-t})^*\L_!^{n-t}$ commute with finite limits by \cref{lem: dedent for fix point by close subgroup} it suffices to show that the map is an equivalence on $\En$-modules. This follows by \cref{lem: alpha omega and Ft with and without bar} (2) and \cref{lem: alpha omega is eq}.
\end{proof}

\begin{remark}\label{rem: k(n) local character with roots of unity in corrdinates}
    Unwinding the definition of $C_t$, \cref{thm: k(n) local character with roots of unity} states that for $M\in \Mod_{\SS_{K(n)}[\omega^{(n)}_{p^{\infty}}]}(\Sp_{K(n)})$ there is a canonical identification
    \[
    (\L^{n-t})^*\L^{n-t}_!(M)\simeq \colim_{r}\, L_{T(t)}(M^{\B(\ZZ/p^r)^{n-t}\times (-)})[e_r^{-1}],
    \]
    where $e$ is the idempotent in
    $L_{T(t)}(\SS_{K(n)}[\omega^{(n)}_{p^{\infty}}])[\B^t\ZZ/p^r]$ that splits off the primitive root of unity and the $L_{T(t)}(\SS_{K(n)}[\omega^{(n)}_{p^{\infty}}])[\B^t\ZZ/p^r]$ module
    structure on $M$ is obtained by restriction of scalars along the composite
    \begin{align*}
    L_{T(t)}(\SS_{K(n)}[\omega^{(n)}_{p^{\infty}}])[\B^t\ZZ/p^r]
    &\to L_{T(t)}(\SS_{K(n)}[\omega^{(n)}_{p^{\infty}}][\B^t\ZZ/p^r]) \\
    &\simeq L_{T(t)}(\SS_{K(n)}[\omega^{(n)}_{p^{\infty}}]^{\B^{n-t}\ZZ/p^r}) \\
    &\to L_{T(t)}(\SS_{K(n)}[\omega^{(n)}_{p^{\infty}}]^{\B(\ZZ/p^r)^{n-t}}).
    \end{align*}
\end{remark}

As a direct consequence of the above, we have:

\begin{corollary}\label{cor: char-iso}
    Let $M\in \Mod_{\SS_{K(n)}[\omega^{(n)}_{p^{\infty}}]}(\Sp_{K(n)})$. Then for every $0\le t\le n$, and every $p$-local $\pi$-finite space $A$, there is a natural equivalence
    \[
    C_t(M^A)\simeq C_t(M)^{\L^{n-t}A}
    \]
\end{corollary}

\begin{proof}
    This is just \cref{thm: k(n) local character with roots of unity} and \cref{thm: F! of R A}.
\end{proof}

\begin{remark}
    One can directly deduce \cref{cor: char-iso} and \cref{thm: F! of R A}, in the case relevant here, namely for $F=\L^{n-t}$ and $R\in \Sp_{K(n)}$, from the character isomorphism of \cite[Theorem A]{ben-mosh-Transchromatic}
    \[
    C_t^n \otimes_{\En} \En^A \simeq (C_t^n)^{\L^{n-t}A}.
    \]
\end{remark}

In some cases this can be written as a base change formula.

\begin{corollary}\label{cor:L!-tensor-RA}
    Let $R\in \CAlg_{\SS_{K(n)}[\omega^{(n)}_{p^\infty}]}(\Sp_{K(n)})$. Assume that for each $r\ge 1$ the object $R^{\B(\ZZ/p^r)^{n-t}}$ is a perfect $R$-module. Then for any $p$-local $\pi$-finite space $A$ and  $M\in \Mod_R(\Sp_{T(n)})$ there is a natural equivalence
    \[
    \L^{n-t}_!(R)\otimes_R M^A
    \simeq
    C_t(R) \otimes_R M^A
    \simeq
    C_t(M)^{\L^{n-t}A}.
    \]
    where the tensor product is implicitly $T(t)$-localized.
\end{corollary}

\begin{proof}
    As $R^{\B(\ZZ/p^r)^{n-t}}$ is perfect 
    \[
    \colim_r L_{T(t)}(R^{\B(\ZZ/p^r)^{n-t}})[e_r^{-1}]\otimes_R M^A\simeq \colim_r L_{T(t)}(M^{\B(\ZZ/p^r)^{n-t}\times A})[e_r^{-1}]
    \]
    so the claim follows by \cref{cor: char-iso}.
\end{proof}

We use \cref{thm: k(n) local character with roots of unity} and cyclotomic descent to deduce a general formula for the $K(n)$-local universal character. 

\begin{corollary}\label{cor: L!-as-hZ}
    Let $M\in \Sp_{K(n)}$, and set
    \[
    M[\omega^{(n)}_{p^\infty}] := M\otimes \SS_{K(n)}[\omega^{(n)}_{p^\infty}].
    \]
    If $p$ is odd, then
    \[
    \L^{n-t}_!(M) \simeq (C_t(M[\omega^{(n)}_{p^\infty}]))^{h\ZZ},
    \]
    where $\ZZ$ acts via a choice of topological generator of $\ZZ_p^\times$.
    
    If $p=2$, then
    \[
     \L^{n-t}_!(M)[i^{(t)}] \simeq \L^{n-t}_!(M[i^{(n)}])\simeq
    (C_t(M[\omega^{(n)}_{p^\infty}]))^{h\ZZ},
    \]
    where $\ZZ$ acts via a choice of topological generator of the subgroup
    $1+4\ZZ_2\subseteq \ZZ_2^\times$. Here $i^{(n)}$ and $i^{(t)}$ are the height $n$ and height $t$ primitive fourth roots of unity.
\end{corollary}

\begin{proof}
    For $p$ odd, the continuous $\ZZ_p^\times$-action is procyclic, so its continuous fixed points are computed as the fiber of $g-1$ for a choice of topological generator $g\in \ZZ_p^\times$ and the functor $\L_!^{n-t}$ commutes with finite limits. For $p=2$ the same description holds after restricting to the procyclic subgroup $1+4\ZZ_2\subseteq \ZZ_2^\times$ and choosing a topological generator.
\end{proof}

We can also get a formula for the universal character that only sees the finite level cyclotomic extensions.

\begin{corollary}\label{lem:L!-as-finite-quotient-hfps}
    Let $M\in \Sp_{K(n)}$.
    Then there is a canonical equivalence
        \[
        \L^{n-t}_!(M)
        \simeq
        \bigl(C_t(M[\omega^{(n)}_{p^\infty}])\bigr)^{h_c\ZZ_p\units}
        \simeq
        \colim_r\,
        \Bigl(
        L_{T(t)}(M[\omega^{(n)}_{p^r}]^{\B(\ZZ/p^r)^{n-t}})[e_r^{-1}]
        \Bigr)^{h(\ZZ/p^r)\units}.
         \]
\end{corollary}

\begin{proof}
    Assume $p$ is odd, the proof for $p$ even is analogous. For every $r\ge 1$ let
    \[
    G_r:=\ker\!\bigl(\ZZ_p\units \to (\ZZ/p^r)\units\bigr).
    \]
    By the same descent argument as in \cref{cor: L!-as-hZ}, but without choosing
    a topological generator, one gets a canonical equivalence $\L^{n-t}_!(M)\simeq
    \bigl(C_t(M[\omega^{(n)}_{p^\infty}])\bigr)^{h_c\ZZ_p\units}$.
    It therefore remains to identify these continuous fixed points.

    For every $r\ge 1$, set
    \[
    X_r=
    L_{T(t)}\!\bigl(M[\omega^{(n)}_{p^\infty}]^{\B(\ZZ/p^r)^{n-t}}\bigr)[e_r^{-1}].
    \]
    Then by \cref{rem: k(n) local character with roots of unity in corrdinates}, $C_t(M[\omega^{(n)}_{p^\infty}])\simeq \colim_r X_r$.

    The subgroup $G_r$ acts trivially on $\B(\ZZ/p^r)^{n-t}$ and
    only on the coefficient object $M[\omega^{(n)}_{p^\infty}]$. Therefore, by definition of the action, taking $G_r$-fixed points on the $r$-th stage descends the coefficients from all height $n$ roots of unity to the height $n$ primitive $p^r$-th roots of unity:
    \[
    X_r^{hG_r}
    \simeq
    L_{T(t)}\!\bigl(M[\omega^{(n)}_{p^r}]^{\B(\ZZ/p^r)^{n-t}}\bigr)[e_r^{-1}].
    \]
    Here we use also that the idempotent $e_r$ is $G_r$-invariant, so it descends under fixed points. Since the subgroups $G_r$
    form a cofinal system of open normal subgroups of $\ZZ_p\units$, the formula for continuous fixed points gives
    \[
    \bigl(C_t(M[\omega^{(n)}_{p^\infty}])\bigr)^{h_c\ZZ_p\units}
    \simeq
    \colim_r \Bigl(\bigl(C_t(M[\omega^{(n)}_{p^\infty}])\bigr)^{hG_r}\Bigr)^{h(\ZZ_p\units/G_r)}.
    \]
    Now write $C_t(M[\omega^{(n)}_{p^\infty}])\simeq \colim_s X_s$.
    Using this presentation, we get a bifiltered diagram. Noting that $G_r$ is procyclic and that in the $T(t)$-local category fixed point of finite groups commute with colimits, we may interchange the two filtered colimits and then pass to the diagonal, since the diagonal
    \[
    \NN \to \NN\times \NN,\qquad r\mapsto (r,r)
    \]
    is cofinal. Hence
    \[
    \bigl(C_t(M[\omega^{(n)}_{p^\infty}])\bigr)^{h_c\ZZ_p\units}
    \simeq
    \colim_r \bigl(X_r^{hG_r}\bigr)^{h(\ZZ_p\units/G_r)}.
    \]
    Since $\ZZ_p\units/G_r\simeq (\ZZ/p^r)\units$, and since
    \[
    X_r^{hG_r}\simeq
    L_{T(t)}\!\bigl(M[\omega^{(n)}_{p^r}]^{\B(\ZZ/p^r)^{n-t}}\bigr)[e_r^{-1}],
    \]
    we obtain
    \[
    \bigl(C_t(M[\omega^{(n)}_{p^\infty}])\bigr)^{h_c\ZZ_p\units}
    \simeq
    \colim_r
    \Bigl(
    L_{T(t)}\!\bigl(M[\omega^{(n)}_{p^r}]^{\B(\ZZ/p^r)^{n-t}}\bigr)[e_r^{-1}]
    \Bigr)^{h(\ZZ/p^r)\units}.
    \]
\end{proof}

\subsection{The $\mrm{GL}_n(\ZZ_p)$ Action and its Fixed Points} \label{sec: GLn-t action}

In this section we prove that for any $M\in \Sp_{T(n)}$ the object $\L^{n-t}_!M$ carries a natural action of the profinite group $\GL_{n-t}(\ZZ_p)$. If $M\in \Sp_{K(n)}$, then the fixed points recover rationalization:
\[
(\L^{n}_!M)^{h\GL_{n}(\ZZ_p)} \simeq L_{\QQ}M.
\]
We also show that the induced map
\[
L_{\QQ}(M^{(-)}) \to (\L^{n})^*\L^{n}_!M
\]
exhibits $(\L^{n})^*\L^{n}_!M$ as a $\GL_{n}(\ZZ_p)$ Galois
extension of $L_{\QQ}(M^{(-)})$ in $\PCMoninf(\Sp)$ when $M$ is a commutative ring that has all roots of unity.

We start by defining the action.

\begin{lemma}\label{lem:GL-action-Ln-t}
    For each $0\le t\le n$, the profinite group $\GL_{n-t}(\ZZ_p)$ acts on $(\L^{n-t})^*$ and $\L^{n-t}_!(M)$. Moreover, the unit of the
    adjunction
    \[
    \eta\colon \id \to (\L^{n-t})^* \L^{n-t}_!
    \]
    is $\GL_{n-t}(\ZZ_p)$-equivariant.
\end{lemma}

\begin{proof}
    By \cite[Proposition 3.4.7]{Lurie-2019-Elliptic3}, the endofunctor
    \[
    \L^{n-t}\colon \spcpi\to \spcpi
    \]
    admits a presentation
    \[
    \L^{n-t}\simeq \colim_r\, \Map(\B(\ZZ/p^r)^{n-t},-).
    \]
    From this we get an action of the profinite group  $\GL_{n-t}(\ZZ_p)$, and therefore on the associated pullback functor $(\L^{n-t})^*$.
    
    Transporting this action across the adjunction, the left adjoint $\L^{n-t}_!$ also inherits a natural $\GL_{n-t}(\ZZ_p)$-action.
    
    It remains to check equivariance of the unit map. The action of $\GL_{n-t}(\ZZ_p)$ on $\L^{n-t}_!$ induces an action on the mapping space $\Map(\L^{n-t}_!,\L^{n-t}_!)$ by conjugation. Under the adjunction isomorphism
    \[
    \Map(\L^{n-t}_!,\L^{n-t}_!) \simeq \Map(\id,(\L^{n-t})^*\L^{n-t}_!),
    \]
    the identity transformation $\id_{\L^{n-t}_!}$ corresponds to the unit
    $\eta\colon \id \to (\L^{n-t})^*\L^{n-t}_!$. Since $\id_{\L^{n-t}_!}$ is fixed under conjugation, its image $\eta$ is fixed as well. Hence $\eta$ is $\GL_{n-t}(\ZZ_p)$-equivariant.
\end{proof}

\begin{remark}
    In the case $t=0$ of \cref{lem:GL-action-Ln-t}, for $M\in \Sp_{K(n)}$, the action of the profinite group $\GL_n(\ZZ_p)$ on $\L_!^{n}(M)$ is taken with respect to the discrete topology on $\L_!^{n}(M)$. This differs from the action of $(\ZZ_p)^\times$ on
    \[
        C_0(M[\omega_p^{(n),\infty}]) = \L_!^n(M[\omega_p^{(n),\infty}]),
    \]
    where this object is instead regarded as carrying its $p$-adic topology.
\end{remark}

We turn to understand the fixed points for $\En$-modules.

\begin{lemma}\label{lem: hGL-over-En}
    Let $M\in \ModEn$. Then the unit of the adjunction induces an equivalence
    \[
    L_{\QQ}(M^{(-)}) \simeq ((\L^{n})^*\L^{n}_!M)^{h\GL_{n}(\ZZ_p)}
    \]
    in $\PCMoninf(\Sp)$.
\end{lemma}

\begin{proof}
    By \cref{thm: k(n) local character with roots of unity} and
    \cref{rem: k(n) local character with roots of unity in corrdinates}, there is an equivalence
    \[
    \L^{n}_!(M) \simeq C_n(M) \simeq
    \colim_r\,L_{\QQ}(M^{\B(\ZZ/p^r)^n\times(-)})[e_r^{-1}].
    \]
    As $\En^{\B(\ZZ/p^r)^n}$ is a free $\En$-module, evaluating at a $\pi$-finite space $A$ gives
    \[
    \colim_r L_{\QQ}(M^A)\otimes_{L_{\QQ}(\En)}
    L_{\QQ}(\En^{\B(\ZZ/p^r)^n})[e_r^{-1}]
    \simeq
    L_{\QQ}(M^A)\otimes_{L_{\QQ}(\En)}\L^n_!\En.
    \]
    By \cite[Remark~2.7.10]{Lurie-2019-Elliptic3}, $\L^n_!\En$ is a
    $\GL_n(\ZZ_p)$-Galois extension of $L_{\QQ}(\En)$. Galois descent therefore identifies the homotopy fixed points of the last expression with $L_{\QQ}(M^A)$. This identification is natural in $A$ and is induced by the unit, proving the claim.
\end{proof}

Using \cref{lem: dedent for fix point by close subgroup} we deduce the corresponding fixed point statement for any $M\in \Sp_{K(n)}$.

\begin{theorem}\label{thm: unit-rational-hGL}
    Let $M\in \Sp_{K(n)}$. Then the unit map induces an equivalence
    \[
    L_{\QQ}(M^{(-)}) \simeq  ((\L^{n})^*\L^n_!M)^{h\GL_{n}(\ZZ_p)}
    \]
    in $\PCMoninf(\Sp)$. In particular, evaluation at $\pt$ gives an equivalence of spectra
    \[
    L_{\QQ}M \simeq (\L^{n}_!M)^{h\GL_{n}(\ZZ_p)}.
    \]
\end{theorem}

\begin{proof}
    As both sides commute with finite limits this follows from \cref{lem: dedent for fix point by close subgroup} and \cref{lem: hGL-over-En}. 
\end{proof}

Using the formula for the universal character of a $K(n)$-local object with all height $n$ roots of unity, together with the base-change result, we deduce that the resulting action gives rise to a Galois extension in $\PCMoninf(\Sp_{\QQ})$.

\begin{proposition}\label{prop: Rognes-Galois-Ln-PCMon}
    Let $R\in \CAlg_{\SS_{K(n)}[\omega^{(n)}_{p^\infty}]}(\Sp_{K(n)})$. Then the map
    \[
    L_{\QQ}(R^{(-)}) \to (\L^{n})^*\L^{n}_!R
    \]
    exhibits $(\L^{n})^*\L^{n}_!R$ as a $\GL_n(\ZZ_p)$ Galois extension of $L_{\QQ}(R^{(-)})$ in $\PCMoninf(\Sp_{\QQ})$.
\end{proposition}

\begin{proof}
    The first Galois condition, namely
    \[
    \bigl((\L^{n})^*\L^{n}_!R\bigr)^{h\GL_n(\ZZ_p)} \simeq L_{\QQ}(R^{(-)}),
    \]
    is \cref{thm: unit-rational-hGL}. By \cref{thm: k(n) local character with roots of unity} we have
    \[
    (\L^{n})^*\L^{n}_!R \simeq C_n(R) \simeq \colim_r\, L_{\QQ}(R^{\B(\ZZ/p^r)^{n}\times (-)})[e_r^{-1}],
    \]
    Using \cref{cor: base change} and that tensor products commute with colimits, we obtain
    \[
    C_n(R)\otimes_{L_{\QQ}(R^{(-)})} C_n(R)
    \simeq
    \colim_{r,s}\, L_{\QQ}(R^{\B(\ZZ/p^r)^{n}\times \B(\ZZ/p^s)^{n}\times (-)})[e_r^{-1}][e_s^{-1}].
    \]
    Applying \cref{thm: k(n) local character with roots of unity} again, the right-hand side may be rewritten as
    \[
    \colim_{s}\,
    C_n(R)^{\L^n(\B(\ZZ/p^s)^{n}\times (-))}[e_s^{-1}].
    \]
    Under the identification $\pi_0(\L^n(\B(\ZZ/p^s)^{n}))\simeq M_n(\ZZ/p^s)$,
    inverting $e_s$ restricts to the invertible matrices, hence to $\GL_n(\ZZ/p^s)$.
    Therefore the previous expression identifies with
    \[
    \colim_s\, \Map(\GL_n(\ZZ/p^s),C_n(R))
    \simeq
    \Map_c(\GL_n(\ZZ_p),C_n(R)).
    \]
\end{proof}

Under suitable finiteness hypotheses (satisfied, for example, by $\En$), the same argument produces a Galois extension in $\Sp_{\QQ}$.

\begin{proposition}\label{prop:L!-R-Galois}
    Let $R\in \CAlg_{\SS_{K(n)}[\omega^{(n)}_{p^\infty}]}(\Sp_{K(n)})$. Assume that for each $r\ge 1$ the object $R^{\B(\ZZ/p^r)^{n}}$ is a perfect $R$-module. Then the map
    \[
    L_{\QQ}(R) \to \L^{n}_!R
    \]
    exhibits $\L^{n}_!R$ as a $\GL_n(\ZZ_p)$ Galois extension of $L_{\QQ}(R)$ in $\Sp_{\QQ}$.
\end{proposition}

\begin{proof}
    The first Galois condition
    \[
    (\L^{n}_!R)^{h\GL_n(\ZZ_p)}\simeq L_{\QQ}(R)
    \]
    follows from \cref{thm: unit-rational-hGL} by evaluation at $\pt$.

    For the second Galois condition, the perfect $R$-module
    $R^{\B(\ZZ/p^r)^n}$ is dualizable, so tensoring with it preserves
    limits. Consequently, for every $r,s\geq 1$, the canonical map
    \[
    R^{\B(\ZZ/p^r)^n}\otimes_R R^{\B(\ZZ/p^s)^n}
    \longrightarrow
    R^{\B(\ZZ/p^r)^n\times\B(\ZZ/p^s)^n}
    \]
    is an equivalence. 

    By \cref{thm: k(n) local character with roots of unity}, we have
    $\L^{n}_!R\simeq\colim_r
    L_{\QQ}(R^{\B(\ZZ/p^r)^n})[e_r^{-1}]$. Since tensor products
    commute with colimits, the preceding equivalence gives
    \[
    \L^{n}_!R\otimes_{L_{\QQ}(R)}\L^{n}_!R
    \simeq
    \colim_{r,s}
    L_{\QQ}
    (R^{\B(\ZZ/p^r)^n\times\B(\ZZ/p^s)^n})
    [e_r^{-1}][e_s^{-1}].
    \]
    This is the evaluation at $\pt$ of the double-colimit expression
    appearing in the proof of
    \cref{prop: Rognes-Galois-Ln-PCMon}. The remainder of the
    computation there therefore identifies the canonical Galois
    comparison map with an equivalence
    \[
    \L^{n}_!R\otimes_{L_{\QQ}(R)}\L^{n}_!R
    \simeq
    \Map_c(\GL_n(\ZZ_p),\L^{n}_!R).
    \]
    Thus both Galois conditions hold.
\end{proof}

\section{Character Theory for the $K(n)$-Local Sphere} \label{sec: Character Theory for the K(n) Local Sphere}

In this section we compute the $n$-fold universal character of the $K(n)$-local sphere and deduce from it a formula for the rationalization of $\SS_{K(n)}^A$ for any $\pi$-finite space $A$.

The general strategy is based on \cref{lem:L!-as-finite-quotient-hfps}, which identifies
\[
    \L_!^n(\SS_{K(n)})
    \simeq
    \colim_r\,
    \Bigl(
    L_{\QQ}\bigl((\SS_{K(n)}[\omega^{(n)}_{p^r}])^{\B(\ZZ/p^r)^{n}}\bigr)[e_r^{-1}]
    \Bigr)^{h(\ZZ/p^r)\units}.
\]
Thus the main task is to compute, for varying $r$, the rational spectra
\[
L_{\QQ}\bigl((\SS_{K(n)}[\omega^{(n)}_{p^r}])^{\B(\ZZ/p^r)^{n}}\bigr)[e_r^{-1}].
\]

To do this, we use the two-towers isomorphism and follow the general strategy of
\cite{Barthe-Schlank-Stapleton-Weinstein-rationalization}. In \cref{sec: Lubin-Tate Side}, we study the rigid-analytic Lubin-Tate cover with full level-$p^r$ structure and compute its rational condensed global sections. Taking continuous fixed points by a suitable subgroup of the Morava stabilizer group will identify the resulting rational algebra object with
\[
L_{\QQ}\bigl((\SS_{K(n)}[\omega^{(n)}_{p^r}])^{\B(\ZZ/p^r)^{n}}\bigr)[e_r^{-1}],
\]
up to adjoining a square-zero class $\epsilon$ in cohomological degree $1$.

In \cref{sec:The Drinfeld Side}, we carry out the parallel computation on the Drinfeld side. More precisely, we analyze the rigid-analytic $(\ZZ/p^r)\units$-cover of Drinfeld upper half-space and compute its rational condensed global sections. In \cref{sec: two towers} we use the two-towers
isomorphism, to show that the corresponding continuous fixed points under an appropriate subgroup of $\GL_n(\ZZ_p)$ recover 
\[
L_{\QQ}\bigl((\SS_{K(n)}[\omega^{(n)}_{p^r}])^{\B(\ZZ/p^r)^{n}}\bigr)[e_r^{-1}],
\]
again up to the same square-zero class $\epsilon$ in degree $1$.

In \cref{sec: SK(n) universal character}, we combine these computations to determine the universal $n$-fold character of $\SS_{K(n)}$ and deduce the rationalization of $\SS_{K(n)}^A$ for any $\pi$-finite space $A$.

Finally, in \cref{sec: Rational K(n) Local Power Operations} we use our computation to identify the ring of rational $K(n)$-local power operations.

Throughout this section, we adopt the notation and conventions of \cite{Barthe-Schlank-Stapleton-Weinstein-rationalization}.

\subsection{The Lubin-Tate Side}\label{sec: Lubin-Tate Side}

We begin by introducing the main objects that will appear throughout this subsection.

\begin{notation}\label{ass:setup}
Fix the height $n$ Honda formal group $\Gamma$ over $\overline{\FF}_p$, and let $A$ be its
Lubin--Tate deformation ring. Thus
\[
A \simeq W(\overline{\FF}_p)[[u_1,\dots,u_{n-1}]]
\]
non-canonically. Set
\[
K:=W(\overline{\FF}_p)[1/p],\qquad \LT:=\Spf A,
\]
and write
\[
\LT_K:=(\LT)_\eta=(\Spf A)_\eta
\]
for the rigid-analytic generic fiber of $\LT$.

For each $r\ge 1$, let $D_r$ be the Drinfeld full level-$p^r$ deformation ring, and set
\[
\Level_r:=\Spf D_r,
\qquad
\LevelG:=(\Level_r)_\eta=(\Spf D_r)_\eta.
\]
The structure map $A\to D_r$ induces morphisms
\[
\Level_r\to \LT
\qquad\text{and}\qquad
f_r\colon \LevelG\to \LT_K .
\]

By \cite[Thm.~2.2(ii)]{StrauchLT}, $D_r$ is finite flat over $A$, and
$D_r[1/p]$ is finite etale over $A[1/p]$, in particular,
\[
f_r\colon \LevelG\to \LT_K
\]
is a finite etale morphism. Since $A$ is local, $D_r$ is finite free over $A$.

 Let $\GG_n$ be the Morava stabilizer group acting continuously on $A$.
Hence $\GG_n$ acts on $\LT=\Spf A$, and therefore on its generic fiber $\LT_K$.
The group $\GL_n(\ZZ/p^r)$ acts trivially on $\LT$ and $\LT_K$, and acts on $\Level_r$ and $\LevelG$. Thus $f_r$ is equivariant for
\[
J_r:=\GG_n\times \GL_n(\ZZ/p^r).
\]
\end{notation}

Our goal in this subsection is to compute the $J_r$-equivariant condensed global sections of $\LevelG$. We do this by bounding the $p$-torsion of the cofiber of
\[
(D_r\otimes_A \Ohatp_{\LTG})_{\mathrm{cond}}
\to
Rf_{r,\proet,*}\Ohatpc_{\LevelG}.
\]
This reduces the computation to \cite[Thm.~3.9.1]{Barthe-Schlank-Stapleton-Weinstein-rationalization}. The argument is local on the pro-etale site: we verify the claim after evaluating on an affinoid perfectoid basis of $(\LTG)_{\proet}$. We therefore need to understand how the cover $\LevelG\to\LTG$ behaves on affinoid perfectoid objects.

We start by recording how a map of affinoids is pulled back under a finite etale map.

\begin{lemma}\label{lem:affinoid-fet-basechange}
Let
\[
X=\Spa(A,A^+),\qquad Y=\Spa(B,B^+),\qquad Z=\Spa(C,C^+)
\]
be affinoid adic spaces over a complete nonarchimedean field $K$, and assume that
\[
X\to Y
\]
is finite etale. Then
\[
X\times_Y Z \simeq \Spa(D,D^+),
\]
where
\[
D=A\otimes_B C
\]
and $D^+$ is the integral closure of $C^+$ in $D$. In particular, $D$ is finite etale over $C$.
\end{lemma}

\begin{proof}
If $K$ has characteristic $p$, this is exactly
\cite[Lem.~7.3(iv)]{ScholzePerfectoid}. In characteristic $0$, it follows from
\cite[Prop.~7.10]{ScholzePerfectoid}, which states that parts (iii) and (iv) of \cite[Lem.~7.3]{ScholzePerfectoid} remain valid in characteristic $0$.
\end{proof}

The next step is to choose standard closed balls in $\LTG$ together with their formal models.

\begin{lemma}\label{lem:closed-ball-model}
Let
\[
B=\Spa(C,C^+)\into \LTG
\]
be a standard closed ball in the chosen Lubin--Tate coordinates, and let
\[
\mathfrak B=\Spf(C^+)\to \LT
\]
be the corresponding formal model of topologically finite type over $\ZZ_p$, whose generic
fiber is $B$. Let
\[
Y:=\LevelG\times_{\LTG} B
\]
with projection $f\colon Y\to B$. Set
\[
S^+:=D_r\otimes_A C^+,
\qquad
S:=S^+[1/p]\simeq D_r[1/p]\otimes_{A[1/p]} C.
\]
Then $S^+$ is finite free over $C^+$, and if $\widetilde S^+$ denotes the integral closure of
$C^+$ in $S$, then
\[
Y\simeq \Spa(S,\widetilde S^+).
\]
\end{lemma}

\begin{proof}
For a standard closed ball in the chosen coordinates, the corresponding formal model
\[
\mathfrak B=\Spf(C^+)\to \LT=\Spf(A)
\]
is an affine formal scheme of topologically finite type over $W(\overline{\FF}_p)$.

Now form the fiber product on the formal side:
\[
\mathfrak Y:=\Level_r\times_{\LT}\mathfrak B.
\]
So  $\mathfrak Y\simeq \Spf(D_r\widehat\otimes_A C^+)$.
Since $D_r$ is finite over $A$, the completed tensor product agrees with the ordinary tensor
product, so
\[
\mathfrak Y\simeq \Spf(D_r\otimes_A C^+)=\Spf(S^+).
\]
Because $D_r$ is finite free over the local ring $A$, it follows that $S^+$ is finite free
over $C^+$.

Moreover, $C^+$ is topologically of finite type over $W(\overline{\FF}_p)$, and $S^+$ is finite over
$C^+$, hence $S^+$ is again topologically of finite type over $W(\overline{\FF}_p)$. Therefore
$\mathfrak Y=\Spf(S^+)$ is an affine admissible formal $\ZZ_p$-scheme in the sense of \cite[Def. 3(iii) and Def. 1]{BoschFormalRigid}.

We now explain the generic fiber of $\mathfrak Y$. By the explicit description of the adic generic fiber of an affine admissible formal scheme,
\[
(\Spf(S^+))_\eta \simeq \Spa(S,\widetilde S^{\,+}_{\mathrm{int}}), \footnote{Bosch identifies the rigid generic fiber of the admissible affine formal scheme
$\Spf(S^+)$ with the affinoid rigid space $\Sp(S)$, where $S=S^+[1/p]$, see
\cite[Prop.~3]{BoschFormalRigid}. To pass from this to adic notation, we use Huber's
functor $\operatorname{rk}$ from rigid analytic varieties to adic spaces: for an affinoid
rigid space $\Sp(S)$, the proof of \cite[Prop.~4.3]{HuberFormalRigid}, together with
\cite[Lem.~4.4]{HuberFormalRigid}, identifies $\operatorname{rk}(\Sp(S))$ with
$\Spa(S,S^\circ)$.}
\]
where $S=S^+[1/p]$ and $\widetilde S^{\,+}_{\mathrm{int}}$ is the integral closure of $S^+$ in $S$.
Since $S^+$ is integral over $C^+$, the integral closure of $S^+$ in $S$ coincides with the
integral closure of $C^+$ in $S$. Thus
\[
\mathfrak Y_\eta \simeq \Spa(S,\widetilde S^+).
\]

Finally, the rigid generic fiber functor commutes with finite limits. Hence
\[
\mathfrak Y_\eta
\simeq
(\Level_r\times_{\LT}\mathfrak B)_\eta
\simeq
(\Level_r)_\eta\times_{\LT_\eta}\mathfrak B_\eta
=
\LevelG\times_{\LTG} B
=
Y.
\]
Combining the two identifications gives $Y\simeq \Spa(S,\widetilde S^+)$, as claimed. 
\end{proof}

We now build our comparison map.

\begin{construction}\label{constr:comparison-morphism}
We construct a natural morphism in the derived category of sheaves of condensed abelian groups on
$(\LTG)_{\proet}$
\[
\beta_{r,\mathrm{cond}}\colon
\bigl(D_r\otimes_A \Ohatp_{\LTG}\bigr)_{\cond}
\to
Rf_{r,\proet,*}\Ohatpc_{\LevelG}.
\]

We first define an underived morphism of sheaves
\[
\beta_r'\colon
D_r\otimes_A \Ohatp_{\LTG}
\to
f_{r,\proet,*}\Ohatpc_{\LevelG}
\]
on the affinoid perfectoid basis of $(\LTG)_{\proet}$.

Let
\[
U=\Spa(R,R^+)\to \LTG
\]
be an affinoid perfectoid object. Since $\LTG$ is the open unit ball, the image of the
quasicompact space $U$ is contained in some standard closed ball
\[
B=\Spa(C,C^+)\into \LTG.
\]
Let
\[
Y:=\LevelG\times_{\LTG} B
\qquad\text{and}\qquad
f\colon Y\to B
\]
be the base change of $f_r$. By \cref{lem:closed-ball-model}, if we set
\[
S^+:=D_r\otimes_A C^+,
\qquad
S:=S^+[1/p],
\]
then $Y\simeq \Spa(S,\widetilde S^+)$, where $\widetilde S^+$ is the integral closure of $C^+$ in $S$. Now form the pullback $Y_U:=Y\times_B U$.
Since $Y\to B$ is finite etale, \cref{lem:affinoid-fet-basechange} shows that $Y_U$ is
again affinoid perfectoid, say $Y_U\simeq \Spa(T,T^+)$, where
\[
T\simeq S\otimes_C R
\simeq D_r[1/p]\otimes_{A[1/p]} R
\]
and $T^+$ is the integral closure of $R^+$ in $T$. On the object $U$, the source evaluates to
\[
\bigl(D_r\otimes_A \Ohatp_{\LTG}\bigr)(U)
=
D_r\otimes_A R^+,
\]
while the target evaluates to
\[
\bigl(f_{r,\proet,*}\Ohatpc_{\LevelG}\bigr)(U)
=
H^0\bigl((Y_U)_{\proet},\Ohatpc_{\LevelG}\bigr)
=
T^+.
\]
The natural map
\[
D_r\otimes_A R^+
=
S^+\otimes_{C^+}R^+
\to
S\otimes_C R
=
T
\]
lands in $T^+$, because $D_r\otimes_A R^+$ is finite over $R^+$ and hence its image
consists of elements integral over $R^+$. Thus we obtain a map
\[
D_r\otimes_A R^+\to T^+.
\]
This map is independent of the chosen closed ball $B$, and therefore defines the desired morphism of sheaves
\[
\beta_r'\colon
D_r\otimes_A \Ohatp_{\LTG}
\to
f_{r,\proet,*}\Ohatpc_{\LevelG}.
\]

Passing to condensed sheaves and composing with the canonical morphism
\[
f_{r,\proet,*}\Ohatpc_{\LevelG}
\to
Rf_{r,\proet,*}\Ohatpc_{\LevelG},
\]
we obtain the morphism
\[
\beta_{r,\mathrm{cond}}\colon
\bigl(D_r\otimes_A \Ohatp_{\LTG}\bigr)_{\cond}
\to
Rf_{r,\proet,*}\Ohatpc_{\LevelG}.
\]
\end{construction}

We study the relevant cofiber restricted to the closed ball.

\begin{lemma}\label{lem:local-comparison-ball}
Let
\[
B=\Spa(C,C^+)\into \LTG
\]
be a standard closed ball, let
\[
f\colon Y:=\LevelG\times_{\LTG} B \to B
\]
be the base change of $f_r$, and let
\[
S^+:=D_r\otimes_A C^+,
\qquad
S:=S^+[1/p].
\]
Then there exists an integer $N\ge 0$, depending only on the finite free $A$-algebra $D_r$, such that for every affinoid perfectoid object $U=\Spa(R,R^+) \to B$
the cofiber
\[
K_U:=
\cofib\!\Bigl(
R\Gamma\bigl(U,(S^+\otimes_{C^+}\Ohatp_B)_{\cond}\bigr)
\to
R\Gamma\bigl(U,Rf_{\proet,*}\Ohatpc_Y\bigr)
\Bigr),
\]
has $p^N$-torsion cohomology. That is
\[
p^N\cdot H^i(K_U)=0
\qquad \text{for all } i\ge 0.
\]
\end{lemma}

\begin{proof}
Since $D_r$ is finite free over $A$, the trace pairing over the generic fiber defines an
$A$-linear map
\[
\tau_A\colon D_r \to D_r\dual:=\Hom_A(D_r,A),
\qquad
x\mapsto \bigl(y\mapsto \Tr_{D_r[1/p]/A[1/p]}(xy)\bigr).
\]
After inverting $p$, this becomes an isomorphism because $D_r[1/p]$ is finite etale over
$A[1/p]$. Hence there exists an integer $N_{\mathrm{tr}}\ge 0$ such that
\[
p^{N_{\mathrm{tr}}}\cdot \coker(\tau_A)=0.
\]
By
\cref{lem:affinoid-fet-basechange,lem:closed-ball-model}, the pullback $Y_U:=Y\times_B U$ is again affinoid perfectoid, say
\[
Y_U\simeq \Spa(T,T^+),
\qquad
T=S\otimes_C R,
\]
where $T^+$ is the integral closure of $R^+$ in $T$. By definition of derived pushforward,
\[
R\Gamma\bigl(U,Rf_{\proet,*}\Ohatpc_Y\bigr)
\simeq
R\Gamma\bigl((Y_U)_{\proet},\Ohatpc_Y\bigr).
\]

Because $S^+$ is finite free over $C^+$, we have an isomorphism 
\[
R\Gamma\bigl(U,(S^+\otimes_{C^+}\Ohatp_B)_{\cond}\bigr)
\simeq
R\Gamma(U,\Ohatp_B)^{\oplus d}.
\]

Since $U=\Spa(R,R^+)$ is affinoid perfectoid, $H^0(U,\Ohatp_B)=R^+$ and for every $j>0$, the cohomology $H^j(U,\Ohatp_B)$ is annihilated by the ideal of topologically nilpotent elements of $R^+$, in particular, it is
killed by $p$, see \cite[Lemma 4.10. (v)]{scholze2013adic}. Likewise, since $Y_U\simeq \Spa(T,T^+)$ is affinoid perfectoid,
\[
H^0\bigl((Y_U)_{\proet},\Ohatpc_Y\bigr)=T^+
\]
and for every $j>0$, the cohomology $H^j\bigl((Y_U)_{\proet},\Ohatpc_Y\bigr)$ is annihilated by $p$. 

Set $M:=S^+\otimes_{C^+}R^+$. Since $S^+$ is finite free over $C^+$, the $R^+$-module $M$ is finite free. Base-changing
$\tau_A$ along $A\to C^+\to R^+$ gives
\[
\tau_U\colon M \to M\dual:=\Hom_{R^+}(M,R^+)
\]
with
\[
p^{N_{\mathrm{tr}}}\cdot \coker(\tau_U)=0.
\]
Equivalently, $p^{N_{\mathrm{tr}}}M\dual\subset M$. We claim that $T^+\subset M\dual$. Indeed, let $t\in T^+$ and $m\in M$. Since both $t$ and $m$ are integral over $R^+$,
their product $tm$ is integral over $R^+$. Therefore its trace $\Tr_{T/R}(tm)$ is integral
over $R^+$, as $R^+$ is integrally closed in $R$, it follows that $\Tr_{T/R}(tm)\in R^+$. Thus $t$ defines an $R^+$-linear functional on $M$, proving the claim.

Combining the two inclusions, we obtain
\[
p^{N_{\mathrm{tr}}}T^+ \subset p^{N_{\mathrm{tr}}}M\dual \subset M \subset T^+.
\]
Hence the cokernel of $M \to T^+$ is killed by $p^{N_{\mathrm{tr}}}$.

Now consider the long exact cohomology sequence attached to the defining triangle
\[
R\Gamma\bigl(U,(S^+\otimes_{C^+}\Ohatp_B)_{\cond}\bigr)
\to
R\Gamma\bigl(U,Rf_{\proet,*}\Ohatpc_Y\bigr)
\to
K_U.
\]

In degree $0$, the exact segment
\[
H^0\Bigl(R\Gamma\bigl(U,(S^+\otimes_{C^+}\Ohatp_B)_{\cond}\bigr)\Bigr)
\to
H^0\Bigl(R\Gamma\bigl(U,Rf_{\proet,*}\Ohatpc_Y\bigr)\Bigr)
\to
H^0(K_U)
\to
H^1\Bigl(R\Gamma\bigl(U,(S^+\otimes_{C^+}\Ohatp_B)_{\cond}\bigr)\Bigr)
\]
shows that $H^0(K_U)$ is an extension of a group killed by
$p^{N_{\mathrm{tr}}}$, by a subgroup of
\[
H^1\Bigl(R\Gamma\bigl(U,(S^+\otimes_{C^+}\Ohatp_B)_{\cond}\bigr)\Bigr),
\]
which is killed by $p$. Therefore $p^{N_{\mathrm{tr}}+1}\cdot H^0(K_U)=0$.

For $i>0$, the exact segment
\[
H^i\Bigl(R\Gamma\bigl(U,(S^+\otimes_{C^+}\Ohatp_B)_{\cond}\bigr)\Bigr)
\to
H^i\Bigl(R\Gamma\bigl(U,Rf_{\proet,*}\Ohatpc_Y\bigr)\Bigr)
\to
H^i(K_U)
\to
H^{i+1}\Bigl(R\Gamma\bigl(U,(S^+\otimes_{C^+}\Ohatp_B)_{\cond}\bigr)\Bigr)
\]
shows that $H^i(K_U)$ is an extension of $p$-torsion groups, hence
\[
p^2\cdot H^i(K_U)=0
\qquad i>0.
\]

Thus, if we set $N:=\max\{N_{\mathrm{tr}}+1,2\}$,
then
\[
p^N\cdot H^i(K_U)=0
\qquad \text{for all } i\ge 0,
\]
as claimed.
\end{proof}

As every affinoid factors through the closed ball, we can globalize the previous lemma. 

\begin{proposition}\label{thm:finite-level-rational}
There exists a $J_r$-equivariant morphism of algebra objects
\[
\alpha_{\Gamma,r}\colon
D_r[\epsilon]\to \RGammacond((\LevelG)_{\proet},\Ohatp_{\LevelG})
\]
in $D(\Solid_{J_r})$, where $\epsilon$ is square-zero in cohomological degree $1$
and $J_r$ acts trivially on $\epsilon$.

Moreover, if $C$ denotes the cofiber of $\alpha_{\Gamma,r}$, then for every $i\in \ZZ$
there exists an integer $M_i\ge 0$ such that
\[
p^{M_i}\cdot H^i(C)=0.
\]
\end{proposition}

\begin{proof}
By \cite[Thm.~3.9.4(1)]{Barthe-Schlank-Stapleton-Weinstein-rationalization}, there is a
$\GG_n$-equivariant morphism of algebra objects
\[
\alpha_{\Gamma}\colon
A[\epsilon]\to \RGammacond((\LTG)_{\proet},\Ohatp_{\LTG})
\]
in $D(\Solid_{\GG_n})$, with trivial $\GG_n$-action on $\epsilon$, whose cofiber is
annihilated by a power of $p$. We may regard $\alpha_\Gamma$ as a $J_r$-equivariant morphism
by letting $\GL_n(\ZZ/p^r)$ act trivially on both source and target. 
Tensoring $\alpha_\Gamma$ with $D_r$ yields a $J_r$-equivariant morphism
\[
\alpha_\Gamma'\colon
D_r[\epsilon]
=
D_r\otimes_A A[\epsilon]
\to
D_r\otimes_A \RGammacond((\LTG)_{\proet},\Ohatp_{\LTG}),
\]
whose cofiber is also annihilated by a power of $p$. Since $D_r$ is finite free over $A$, there is a natural
$J_r$-equivariant identification
\[
D_r\otimes_A \RGammacond((\LTG)_{\proet},\Ohatp_{\LTG})
\simeq
\RGammacond((\LTG)_{\proet},D_r\otimes_A \Ohatp_{\LTG}).
\]

Let
\[
\beta_{r,\mathrm{cond}}\colon
(D_r\otimes_A \Ohatp_{\LTG})_{\mathrm{cond}}
\to
Rf_{r,\proet,*}\Ohatpc_{\LevelG}
\]
be the natural morphism. Applying derived condensed global sections and using the adjunction
isomorphism
\[
\RGammacond\bigl((\LTG)_{\proet},Rf_{r,\proet,*}\Ohatpc_{\LevelG}\bigr)
\simeq
\RGammacond((\LevelG)_{\proet},\Ohatp_{\LevelG}),
\]
we obtain a $J_r$-equivariant morphism of algebra objects
\[
\gamma_r\colon
\RGammacond((\LTG)_{\proet},D_r\otimes_A \Ohatp_{\LTG})
\to
\RGammacond((\LevelG)_{\proet},\Ohatp_{\LevelG}).
\]

We claim that for every $i\ge 0$ there exists $N_i\ge 0$ such that
\[
p^{N_i}\cdot H^i(\cofib(\gamma_r))=0.
\]
To see this, choose a cover $\{U_j\}_{j\in I}$ of $\LTG$ by affinoid perfectoid objects such
that every finite intersection
\[
U_{j_0\cdots j_q}:=U_{j_0}\times_{\LTG}\cdots\times_{\LTG}U_{j_q}
\]
is again affinoid perfectoid. Consider the induced map of cosimplicial objects
\[
\alpha_{\mathcal U}\colon
\prod_{f\in I^{[q]}}
R\Gamma\Bigl(U_{\operatorname{im}(f)},
(D_r\otimes_A \Ohatp_{\LTG})_{\cond}\Bigr)
\to
\prod_{f\in I^{[q]}}
R\Gamma\Bigl(U_{\operatorname{im}(f)},
Rf_{r,\proet,*}\Ohatpc_{\LevelG}\Bigr).
\]
Its limit is precisely $\gamma_r$. By \cref{lem:local-comparison-ball}, there exists an integer $N\ge 0$, depending only on
$D_r/A$, such that for every affinoid perfectoid $U\to \LTG$, the cofiber
\[
\cofib\!\Bigl(
R\Gamma\bigl(U,(D_r\otimes_A \Ohatp_{\LTG})_{\cond}\bigr)
\to
R\Gamma\bigl(U,Rf_{r,\proet,*}\Ohatpc_{\LevelG}\bigr)
\Bigr)
\]
has cohomology annihilated by $p^N$ in all degrees. In particular, for every $q\ge 0$ and
every $s\ge 0$, the $s$-cohomology of the cofiber of the $q$-th level of
$\alpha_{\mathcal U}$ is annihilated by $p^N$.

We obtain a spectral sequence
\[
E_1^{s,q}
=
\prod_{f\in I^{[q]}}
H^s\!\left(
\cofib\!\Bigl(
R\Gamma\bigl(U_{\operatorname{im}(f)},
(D_r\otimes_A \Ohatp_{\LTG})_{\cond}\bigr)
\to
R\Gamma\bigl(U_{\operatorname{im}(f)},
Rf_{r,\proet,*}\Ohatpc_{\LevelG}\bigr)
\Bigr)\right)
\to
H^{s+q}(\cofib(\gamma_r)).
\]
Each $E_1^{s,q}$ is killed by $p^N$, hence so is each $E_\infty^{s,q}$. For fixed total
degree $i\ge 0$, there are at most $i+1$ pairs $(s,q)$ with $s+q=i$, so the induced
filtration on $H^i(\cofib(\gamma_r))$ has at most $i+1$ graded pieces, each killed by
$p^N$. Therefore
\[
p^{N(i+1)}\cdot H^i(\cofib(\gamma_r))=0
\qquad\text{for all } i\ge 0.
\]

Now define
\[
\alpha_{\Gamma,r}:=\gamma_r\circ \alpha_\Gamma'
\colon
D_r[\epsilon]\to \RGammacond((\LevelG)_{\proet},\Ohatp_{\LevelG}).
\]
This is a $J_r$-equivariant morphism of algebra objects.

Finally, let $C$ be the cofiber of $\alpha_{\Gamma,r}$. Since $\alpha_{\Gamma,r}=\gamma_r\circ \alpha_\Gamma'$, there is a cofiber sequence
\[
\cofib(\alpha_\Gamma')\to C\to \cofib(\gamma_r)\to \Sigma\cofib(\alpha_\Gamma').
\]
The cofiber of $\alpha_\Gamma'$ is annihilated by a power of $p$, while for each $i\ge 0$
the group $H^i(\cofib(\gamma_r))$ is annihilated by $p^{N(i+1)}$. It follows that for every
$i\in \ZZ$ there exists an integer $M_i\ge 0$ such that
\[
p^{M_i}\cdot H^i(C)=0.
\]
This proves the claim.
\end{proof}

We now study how this morphism interacts with taking fixed points.

\begin{proposition}\label{prop:finite-level-rational-open-subgroups}
Let
\[
\alpha_{\Gamma,r}\colon
D_r[\epsilon]\to \RGammacond((\LevelG)_{\proet},\Ohatp_{\LevelG})
\]
be the morphism of \cref{thm:finite-level-rational}. Then for every open subgroup $H\subset \GG_n$, the induced morphism
\[
\alpha_{\Gamma,r}^{hH}\colon
D_r[\epsilon]^{hH}\to
\RGammacond((\LevelG)_{\proet},\Ohatp_{\LevelG})^{hH}
\]
becomes an equivalence after inverting $p$. 
\end{proposition}

\begin{proof}
By \cref{thm:finite-level-rational}, for every $b\ge 0$ there exists an integer $N_b\ge 0$
such that $p^{N_b}\cdot H^b(K)=0$.
where $K:=\cofib(\alpha_{\Gamma,r})$. Since $\GG_n$ has finite virtual cohomological dimension, so does the open subgroup $H$.
Choose an open normal subgroup $U\triangleleft H$ of finite cohomological dimension. 

We first show that each cohomology group of $K^{hU}$ is $p$-power torsion. Apply \cite[Prop.~3.3.6]{Barthe-Schlank-Stapleton-Weinstein-rationalization} to the closed normal
subgroup $1\subset U$. This gives a spectral sequence
\[
E_2^{a,b}
=
H^a_{\cts}(U,H^b(K))
\to
H^{a+b}(K^{hU}).
\]
For each fixed $b$, the coefficient module $H^b(K)$ is killed by $p^{N_b}$, hence every
continuous cochain group computing $H^a_{\cts}(U,H^b(K))$ is also killed by $p^{N_b}$. It
follows that every $E_2^{a,b}$ is $p$-power torsion. Since $U$ has finite cohomological
dimension, only finitely many values of $a$ can contribute to a fixed total degree. Therefore each group $H^m(K^{hU})$ admits a finite filtration with $p$-power torsion graded pieces, and hence is itself $p$-power torsion.

Now apply \cite[Prop.~3.3.6]{Barthe-Schlank-Stapleton-Weinstein-rationalization} to the closed normal subgroup $U\subset H$. We obtain a Lyndon--Hochschild--Serre spectral sequence
\[
E_2^{a,b}
=
H^a_{\cts}\bigl(H/U,\,H^b(K^{hU})\bigr)
\to
H^{a+b}(K^{hH}).
\]
Since $K^{hU}$ is connective, and each $H^b(K^{hU})$ is $p$-power torsion, every $E_2^{a,b}$ is again $p$-power torsion. For each total degree only finitely many values of $a$ occur, so
each $H^m(K^{hH})$ admits a finite filtration with $p$-power torsion graded pieces. Hence $H^m(K^{hH})$ is
$p$-power torsion for every $m$.

Therefore
\[
H^m(K^{hH}[1/p])=0
\qquad\text{for all }m,
\]
so $K^{hH}[1/p]\simeq 0$. This proves that $\alpha_{\Gamma,r}^{hH}$ becomes an equivalence
after inverting $p$.
\end{proof}

\subsection{The Drinfeld Side} \label{sec:The Drinfeld Side}

We begin by introducing the main objects that will appear throughout this subsection.

\begin{definition}\label{def:H-cyclotomic-cover}
Let
\[
\GG_{n,r}
:=\ker\bigl(\GG_n\to (\ZZ/p^r)^\times\bigr).
\]
Let $X$ be the common infinite-level perfectoid space, equipped with commuting
actions of $\GG_n$ and $\GL_n(\ZZ_p)$, together with a $\GG_n$-torsor $X^\diamond \to \HDr^\diamond$.

We define the diamond
\[
(\HCyc)^\diamond:=X^\diamond/\GG_{n,r}.
\]
Then the natural map
\[
g_r^\diamond\colon (\HCyc)^\diamond\to \HDr^\diamond
\]
is a finite etale $(\ZZ/p^r)^\times$-torsor. Since $\HDr$ is a rigid-analytic space and finite etale morphisms of rigid
spaces are equivalent to finite etale morphisms of their associated diamonds,
there exists a unique finite etale rigid-analytic space
\[
g_r\colon \HCyc\to \HDr
\]
whose associated diamond is canonically identified with
$(\HCyc)^\diamond$ .
We call $\HCyc$ the cyclotomic level-$p^r$ cover of $\HDr$.
\end{definition}

Our goal in this subsection is to compute the equivariant condensed global sections of $\HCyc$. We begin by identifying the cover $(\HCyc)^\diamond$ with the cyclotomic pullback.

\begin{proposition}\label{prop:H-cyclotomic-pullback}
There is a canonical equivalence of diamonds
\[
(\HCyc)^\diamond
\simeq
\HDr^\diamond\times_{\Spd \QQ_p^\diamond}\Spd(\QQ_p(\omega_{p^r}))^\diamond.
\]
Under this identification, the action of $(\ZZ/p^r)^\times$ on
$(\HCyc)^\diamond$ is the natural Galois action on the second factor.
\end{proposition}

\begin{proof}
By \cite[Thm.~3.7.1]{Barthel-Schlank-Stapleton-Weinstein-2024-Picard},
there is a determinant morphism $\det\colon M_n \to M_1$ which is compatible with both the Lubin--Tate and Drinfeld presentations.
After pulling back along the map
\[
\HDr^\diamond \to [\HDr^\diamond/\GL_n(\ZZ_p)],
\]
the cover $(\HCyc)^\diamond \to \HDr^\diamond$ is therefore obtained by pulling back the corresponding height $1$ cover along the
composite
\[
\HDr^\diamond \to [\HDr^\diamond/\GL_n(\ZZ_p)]
\xrightarrow{\det} M_1.
\]

Thus the claim reduces to the height $1$ case. In height $1$, the Drinfeld space is $\mathcal H^0 \simeq \Spa(\QQ_p,\ZZ_p)$,
and the relevant Lubin--Tate formal group is the multiplicative formal group
$\widehat{\mathbb G}_m$. Its $p^r$-torsion is $\mu_{p^r}$, so the corresponding
finite etale cover is the cyclotomic extension
\[
\Spd(\QQ_p(\omega_{p^r}))^\diamond \to \Spd \QQ_p^\diamond \simeq \mathcal H^0 .
\]
Hence pulling back along the structural morphism $\HDr^\diamond \to \Spd \QQ_p^\diamond$ gives
\[
(\HCyc)^\diamond
\simeq
\HDr^\diamond\times_{\Spd \QQ_p^\diamond}\Spd(\QQ_p(\omega_{p^r}))^\diamond,
\]
as claimed. 
\end{proof}

We record how this cover behaves on affinoid perfectoid objects.

\begin{lemma}\label{lem:H-affinoid-evaluation}
Let
\[
g_r\colon \mathcal H_r^\times \to \mathcal H
\]
be the finite etale $(\ZZ/p^r)^\times$-torsor from
\cref{def:H-cyclotomic-cover}. Then there is a natural morphism of condensed sheaves on $\mathcal H_{\proet}$
\[
\gamma_r\colon
(\ZZ_p[\omega_{p^r}]\otimes_{\ZZ_p}\Ohatp_{\mathcal H})_{\mathrm{cond}}
\to
g_{r,\proet,*}(\Ohatpc_{\mathcal H_r^\times})_{\mathrm{cond}}.
\]
such that for $U=\Spa(R,R^+)\to \mathcal H_{\proet}$ an affinoid perfectoid object we have
\[
V:=\mathcal H_r^\times\times_{\mathcal H}U.
\]
Then $V$ is affinoid perfectoid, say $V=\Spa(S,S^+)$, with:
\begin{enumerate}
\item $S$ is finite etale over $R$ and there is a canonical identification
\[
S\simeq \QQ_p(\omega_{p^r})\otimes_{\QQ_p}R,
\]
\item $S^+$ is the integral closure of $R^+$ in $S$,
\item the morphism $\gamma_r$ evaluates on $U$ as the condensed enhancement of the natural ring map
\[
\ZZ_p[\omega_{p^r}]\otimes_{\ZZ_p}R^+\to S^+.
\]
\end{enumerate}
Moreover, the construction is functorial in $U$, the action of $(\ZZ/p^r)^\times$ is the Galois action on $\ZZ_p[\omega_{p^r}]$, and the whole construction is $(\ZZ/p^r)^\times\times \GL_n(\ZZ_p)$-equivariant.
\end{lemma}

\begin{proof}
By \cref{prop:H-cyclotomic-pullback}, 
\[
(\mathcal H_r^\times)^\diamond
\simeq
\mathcal H^\diamond\times_{\Spd \QQ_p^\diamond}\Spd(\QQ_p(\omega_{p^r}))^\diamond,
\]
and the action of $(\ZZ/p^r)^\times$ is the Galois action on the second factor. Since $g_r$ is finite etale and $U$ is affinoid perfectoid, finite etale base change for affinoid perfectoids implies that $V=\mathcal H_r^\times\times_{\mathcal H}U$ is affinoid perfectoid, say $V=\Spa(S,S^+)$, with $S\simeq \QQ_p(\omega_{p^r})\otimes_{\QQ_p}R$,
where $S$ is finite etale over $R$ and $S^+$ is the integral closure of $R^+$ in $S$. By definition of condensed enhancement,
\[
(\ZZ_p[\omega_{p^r}]\otimes_{\ZZ_p}\Ohatp_{\mathcal H})_{\mathrm{cond}}(U)
=
\underline{\ZZ_p[\omega_{p^r}]\otimes_{\ZZ_p}R^+},
\qquad
(g_{r,\proet,*}(\Ohatpc_{\mathcal H_r^\times})_{\mathrm{cond}})(U)
=
\underline{S^+},
\]
and the map is the condensed enhancement of the evident ring map
\[
\ZZ_p[\omega_{p^r}]\otimes_{\ZZ_p}R^+\to S^+.
\]
\end{proof}

The desired comparison will follow from the above by verifying it on an affinoid perfectoid cover.

\begin{proposition}\label{prop:H-derived-comparison}
Let
\[
\gamma_r\colon
(\ZZ_p[\omega_{p^r}]\otimes_{\ZZ_p}\Ohatp_{\mathcal H})_{\mathrm{cond}}
\to
g_{r,\proet,*}(\Ohatpc_{\mathcal H_r^\times})_{\mathrm{cond}}
\]
be the morphism of \cref{lem:H-affinoid-evaluation}. Then there exists an integer $N\ge 0$, depending only on $r$, such that for every affinoid perfectoid object $U=\Spa(R,R^+)\to \mathcal H$, the cofiber
\[
K_U:=
\cofib\!\Bigl(
R\Gamma\bigl(U,(\ZZ_p[\omega_{p^r}]\otimes_{\ZZ_p}\Ohatp_{\mathcal H})_{\cond}\bigr)
\to
R\Gamma\bigl(U,g_{r,\proet,*}(\Ohatpc_{\mathcal H_r^\times})_{\cond}\bigr)
\Bigr)
\]
has $p^N$-torsion cohomology. That is,
\[
p^N\cdot H^i(K_U)=0
\qquad\text{for all } i\ge 0.
\]
Moreover, $\gamma_r$ is $(\ZZ/p^r)^\times\times \GL_n(\ZZ_p)$-equivariant.
\end{proposition}

\begin{proof}
Since $\ZZ_p[\omega_{p^r}]$ is finite free over $\ZZ_p$, multiplication defines a trace pairing
\[
\tau\colon \ZZ_p[\omega_{p^r}]
\to
\Hom_{\ZZ_p}(\ZZ_p[\omega_{p^r}],\ZZ_p),
\qquad
x\mapsto \bigl(y\mapsto \Tr_{\QQ_p(\omega_{p^r})/\QQ_p}(xy)\bigr).
\]
After inverting $p$, this becomes an isomorphism because
$\QQ_p(\omega_{p^r})/\QQ_p$ is finite etale. Hence there exists an integer
$N_{\mathrm{tr}}\ge 0$ such that
\[
p^{N_{\mathrm{tr}}}\cdot \coker(\tau)=0.
\]

Let $U=\Spa(R,R^+)\to \mathcal H$ be an affinoid perfectoid object. By \cref{lem:H-affinoid-evaluation},
\[
V:=\mathcal H_r^\times\times_{\mathcal H}U \simeq \Spa(S,S^+),
\qquad
S\simeq \QQ_p(\omega_{p^r})\otimes_{\QQ_p}R,
\]
where $S^+$ is the integral closure of $R^+$ in $S$. By definition of derived pushforward,
\[
R\Gamma\bigl(U,g_{r,\proet,*}(\Ohatpc_{\mathcal H_r^\times})_{\cond}\bigr)
\simeq
R\Gamma\bigl(V_{\proet},(\Ohatpc_{\mathcal H_r^\times})_{\cond}\bigr).
\]

Because $\ZZ_p[\omega_{p^r}]$ is finite free over $\ZZ_p$, say of rank $d$, we have
\[
R\Gamma\bigl(U,(\ZZ_p[\omega_{p^r}]\otimes_{\ZZ_p}\Ohatp_{\mathcal H})_{\cond}\bigr)
\simeq
R\Gamma(U,\Ohatp_{\mathcal H})^{\oplus d}.
\]

Since $U=\Spa(R,R^+)$ is affinoid perfectoid, $H^0(U,\Ohatp_{\mathcal H})=R^+$, and for every $j>0$, the cohomology $H^j(U,\Ohatp_{\mathcal H})$ is annihilated by the ideal
of topologically nilpotent elements of $R^+$, in particular, it is killed by $p$, see \cite[Lemma 4.10.(v)]{scholze2013adic}. 

Likewise, since $V=\Spa(S,S^+)$ is affinoid perfectoid, $H^0\bigl(V_{\proet},\Ohatpc_{\mathcal H_r^\times}\bigr)=S^+$,
and for every $j>0$, the cohomology $H^j\bigl(V_{\proet},\Ohatpc_{\mathcal H_r^\times}\bigr)$ is annihilated by $p$.

Set $M:=\ZZ_p[\omega_{p^r}]\otimes_{\ZZ_p}R^+$. Base-changing $\tau$ along $\ZZ_p\to R^+$ gives
\[
\tau_U\colon M\to M\dual:=\Hom_{R^+}(M,R^+)
\]
with
\[
p^{N_{\mathrm{tr}}}\cdot \coker(\tau_U)=0.
\]
Equivalently, $p^{N_{\mathrm{tr}}}M\dual\subset M$. We claim that $S^+\subset M\dual$. Indeed, if $s\in S^+$ and $m\in M$, then both $s$ and $m$ are integral over $R^+$, hence so is $sm$. Therefore $\Tr_{S/R}(sm)$ is integral over $R^+$, and since $R^+$ is integrally closed in
$R$, we get $\Tr_{S/R}(sm)\in R^+$. Thus $s$ defines an $R^+$-linear functional on $M$, proving the claim.

Combining the two inclusions, we obtain
\[
p^{N_{\mathrm{tr}}}S^+\subset p^{N_{\mathrm{tr}}}M\dual\subset M\subset S^+.
\]
Hence the cokernel of
\[
M=\ZZ_p[\omega_{p^r}]\otimes_{\ZZ_p}R^+\to S^+
\]
is killed by $p^{N_{\mathrm{tr}}}$.

Now consider the long exact cohomology sequence attached to the defining triangle
\[
R\Gamma\bigl(U,(\ZZ_p[\omega_{p^r}]\otimes_{\ZZ_p}\Ohatp_{\mathcal H})_{\cond}\bigr)
\to
R\Gamma\bigl(U,g_{r,\proet,*}(\Ohatpc_{\mathcal H_r^\times})_{\cond}\bigr)
\to
K_U.
\]

In degree $0$, the exact segment
\begin{align*}
H^0\Bigl(R\Gamma\bigl(U,(\ZZ_p[\omega_{p^r}]\otimes_{\ZZ_p}\Ohatp_{\mathcal H})_{\cond}\bigr)\Bigr)
\to
H^0\Bigl(R\Gamma\bigl(U,g_{r,\proet,*}(\Ohatpc_{\mathcal H_r^\times})_{\cond}\bigr)\Bigr)
\to \\
\to H^0(K_U)
\to
H^1\Bigl(R\Gamma\bigl(U,(\ZZ_p[\omega_{p^r}]\otimes_{\ZZ_p}\Ohatp_{\mathcal H})_{\cond}\bigr)\Bigr)
\end{align*}

shows that $H^0(K_U)$ is an extension of a group killed by
$p^{N_{\mathrm{tr}}}$, by a subgroup of a $p$-torsion group. Therefore $p^{N_{\mathrm{tr}}+1}\cdot H^0(K_U)=0$.

For $i>0$, the exact segment
\begin{align*}
H^i\Bigl(R\Gamma\bigl(U,(\ZZ_p[\omega_{p^r}]\otimes_{\ZZ_p}\Ohatp_{\mathcal H})_{\cond}\bigr)\Bigr)
\to
H^i\Bigl(R\Gamma\bigl(U,g_{r,\proet,*}(\Ohatpc_{\mathcal H_r^\times})_{\cond}\bigr)\Bigr)
\to \\
\to
H^i(K_U)
\to
H^{i+1}\Bigl(R\Gamma\bigl(U,(\ZZ_p[\omega_{p^r}]\otimes_{\ZZ_p}\Ohatp_{\mathcal H})_{\cond}\bigr)\Bigr)
\end{align*}
shows that $H^i(K_U)$ is an extension of $p$-torsion groups, hence
\[
p^2\cdot H^i(K_U)=0
\qquad i>0.
\]

Thus, if we set $N:=\max\{N_{\mathrm{tr}}+1,2\}$, then
\[
p^N\cdot H^i(K_U)=0
\qquad\text{for all } i\ge 0,
\]
as claimed. The equivariance of $\gamma_r$ follows from \cref{lem:H-affinoid-evaluation}.
\end{proof}

We now use this and \cite[Theorem 3.9.1]{Barthe-Schlank-Stapleton-Weinstein-rationalization} to compute the global sections.

\begin{lemma}\label{lem:H-cyclotomic-cohomology}
There exists a $(\ZZ/p^r)^\times\times \GL_n(\ZZ_p)$-equivariant morphism of algebra objects
\[
\alpha_{\HDr,r}^{\times}\colon
\ZZ_p[\omega_{p^r}][\epsilon]
\to
\RGammacond\bigl((\mathcal H_r^\times)_{\proet},\Ohatp_{\mathcal H_r^\times}\bigr)
\]
in $D\bigl(\Solid_{(\ZZ/p^r)^\times\times \GL_n(\ZZ_p)}\bigr)$, where
$(\ZZ/p^r)^\times$ acts through its Galois action on $\ZZ_p[\omega_{p^r}]$ and trivially on
$\epsilon$. Moreover, if $C$ denotes the cofiber of $\alpha_{\HDr,r}^{\times}$, then for
every $i\in \ZZ$ there exists an integer $M_i\ge 0$ such that
\[
p^{M_i}\cdot H^i(C)=0.
\]
\end{lemma}

\begin{proof}
By \cite[Thm.~3.5.3(2)]{Barthe-Schlank-Stapleton-Weinstein-rationalization}, there is a
$\GL_n(\ZZ_p)$-equivariant morphism of algebra objects
\[
\alpha_{\mathcal H}\colon
\ZZ_p[\epsilon]\to
\RGammacond(\mathcal H_{\proet},\Ohatp_{\mathcal H})
\]
whose cofiber is annihilated by a power of $p$. We regard $\alpha_{\mathcal H}$ as a
$(\ZZ/p^r)^\times\times \GL_n(\ZZ_p)$-equivariant morphism by letting
$(\ZZ/p^r)^\times$ act trivially on both source and target.

Tensoring $\alpha_{\mathcal H}$ with $\ZZ_p[\omega_{p^r}]$ yields a
$(\ZZ/p^r)^\times\times \GL_n(\ZZ_p)$-equivariant morphism
\[
\alpha_{\mathcal H}'\colon
\ZZ_p[\omega_{p^r}][\epsilon]
=
\ZZ_p[\omega_{p^r}]\otimes_{\ZZ_p}\ZZ_p[\epsilon]
\to
\ZZ_p[\omega_{p^r}]\otimes_{\ZZ_p}\RGammacond(\mathcal H_{\proet},\Ohatp_{\mathcal H}),
\]
whose cofiber is also annihilated by a power of $p$.

Since $\ZZ_p[\omega_{p^r}]$ is finite free over $\ZZ_p$, there is a natural
$(\ZZ/p^r)^\times\times \GL_n(\ZZ_p)$-equivariant identification
\[
\ZZ_p[\omega_{p^r}]\otimes_{\ZZ_p}\RGammacond(\mathcal H_{\proet},\Ohatp_{\mathcal H})
\simeq
\RGammacond\bigl(\mathcal H_{\proet},
\ZZ_p[\omega_{p^r}]\otimes_{\ZZ_p}\Ohatp_{\mathcal H}\bigr).
\]

Let
\[
\gamma_r\colon
(\ZZ_p[\omega_{p^r}]\otimes_{\ZZ_p}\Ohatp_{\mathcal H})_{\mathrm{cond}}
\to
g_{r,\proet,*}(\Ohatpc_{\mathcal H_r^\times})_{\mathrm{cond}}
\]
be the morphism of \cref{prop:H-derived-comparison}. Applying derived condensed global sections
and using the adjunction isomorphism
\[
\RGammacond\bigl(\mathcal H_{\proet},
g_{r,\proet,*}(\Ohatpc_{\mathcal H_r^\times})_{\mathrm{cond}}\bigr)
\simeq
\RGammacond\bigl((\mathcal H_r^\times)_{\proet},\Ohatp_{\mathcal H_r^\times}\bigr),
\]
we obtain a $(\ZZ/p^r)^\times\times \GL_n(\ZZ_p)$-equivariant morphism of algebra objects
\[
\gamma_r'\colon
\RGammacond\bigl(\mathcal H_{\proet},
\ZZ_p[\omega_{p^r}]\otimes_{\ZZ_p}\Ohatp_{\mathcal H}\bigr)
\to
\RGammacond\bigl((\mathcal H_r^\times)_{\proet},\Ohatp_{\mathcal H_r^\times}\bigr).
\]

We claim that for every $i\ge 0$ there exists $N_i\ge 0$ such that
\[
p^{N_i}\cdot H^i(\cofib(\gamma_r'))=0.
\]
To see this, choose a cover $\{U_j\}_{j\in I}$ of $\mathcal H$ by affinoid perfectoid objects
such that every finite intersection
\[
U_{j_0\cdots j_q}:=
U_{j_0}\times_{\mathcal H}\cdots\times_{\mathcal H}U_{j_q}
\]
is again affinoid perfectoid. Consider the induced map of cosimplicial objects
\[
\alpha_{\mathcal U}\colon
\prod_{f\in I^{[q]}}
R\Gamma\Bigl(U_{\operatorname{im}(f)},
(\ZZ_p[\omega_{p^r}]\otimes_{\ZZ_p}\Ohatp_{\mathcal H})_{\mathrm{cond}}\Bigr)
\to
\prod_{f\in I^{[q]}}
R\Gamma\Bigl(U_{\operatorname{im}(f)},
g_{r,\proet,*}(\Ohatpc_{\mathcal H_r^\times})_{\mathrm{cond}}\Bigr).
\]
Its totalization computes $\gamma_r'$.

By \cref{prop:H-derived-comparison}, there exists an integer $N\ge 0$, depending only on
$r$, such that for every affinoid perfectoid $U\to \mathcal H$, the cofiber
\[
\cofib\!\Bigl(
R\Gamma\bigl(U,(\ZZ_p[\omega_{p^r}]\otimes_{\ZZ_p}\Ohatp_{\mathcal H})_{\mathrm{cond}}\bigr)
\to
R\Gamma\bigl(U,g_{r,\proet,*}(\Ohatpc_{\mathcal H_r^\times})_{\mathrm{cond}}\bigr)
\Bigr)
\]
has cohomology annihilated by $p^N$ in all degrees. In particular, for every $q\ge 0$ and
every $s\ge 0$, the cohomology of the cofiber of the $q$-th level of $\alpha_{\mathcal U}$
is annihilated by $p^N$.

We therefore obtain a spectral sequence
\[
E_1^{s,q}
=
\prod_{f\in I^{[q]}}
H^s\!\left(
\cofib\!\Bigl(
R\Gamma\bigl(U_{\operatorname{im}(f)},
(\ZZ_p[\omega_{p^r}]\otimes_{\ZZ_p}\Ohatp_{\mathcal H})_{\mathrm{cond}}\bigr)
\to
R\Gamma\bigl(U_{\operatorname{im}(f)},
g_{r,\proet,*}(\Ohatpc_{\mathcal H_r^\times})_{\mathrm{cond}}\bigr)
\Bigr)\right)
\to
H^{s+q}(\cofib(\gamma_r')).
\]
Each $E_1^{s,q}$ is killed by $p^N$, hence so is each $E_\infty^{s,q}$. For fixed total
degree $i\ge 0$, there are at most $i+1$ pairs $(s,q)$ with $s+q=i$, so the induced
filtration on $H^i(\cofib(\gamma_r'))$ has at most $i+1$ graded pieces, each killed by
$p^N$. Therefore
\[
p^{N(i+1)}\cdot H^i(\cofib(\gamma_r'))=0
\qquad\text{for all } i\ge 0.
\]

Now define
\[
\alpha_{\HDr,r}^{\times}:=\gamma_r'\circ \alpha_{\mathcal H}'\colon
\ZZ_p[\omega_{p^r}][\epsilon]
\to
\RGammacond\bigl((\mathcal H_r^\times)_{\proet},\Ohatp_{\mathcal H_r^\times}\bigr).
\]
This is a $(\ZZ/p^r)^\times\times \GL_n(\ZZ_p)$-equivariant morphism of algebra objects.

Finally, let $C$ be the cofiber of $\alpha_{\HDr,r}^{\times}$. Since $\alpha_{\HDr,r}^{\times}=\gamma_r'\circ \alpha_{\mathcal H}'$, there is a cofiber sequence
\[
\cofib(\alpha_{\mathcal H}')\to C\to \cofib(\gamma_r')\to \Sigma\cofib(\alpha_{\mathcal H}').
\]
The cofiber of $\alpha_{\mathcal H}'$ is annihilated by a power of $p$, while for each
$i\ge 0$ the group $H^i(\cofib(\gamma_r'))$ is annihilated by $p^{N(i+1)}$. It follows
that for every $i\in \ZZ$ there exists an integer $M_i\ge 0$ such that
\[
p^{M_i}\cdot H^i(C)=0.
\]
This proves the claim.
\end{proof}

We now study how this morphism interacts with taking fixed points.

\begin{proposition}\label{prop:H-cyclotomic-open-subgroups}
Let
\[
\alpha_{\HDr,r}^{\times}\colon
\ZZ_p[\omega_{p^r}][\epsilon]\to
\RGammacond\bigl((\mathcal H_r^\times)_{\proet},\Ohatp_{\mathcal H_r^\times}\bigr)
\]
be the morphism of \cref{lem:H-cyclotomic-cohomology}. Then for every open subgroup $H\subseteq \GL_n(\ZZ_p)$, the induced morphism
\[
(\alpha_{\HDr,r}^{\times})^{hH}\colon
\bigl(\ZZ_p[\omega_{p^r}][\epsilon]\bigr)^{hH}\to
\RGammacond\bigl((\mathcal H_r^\times)_{\proet},\Ohatp_{\mathcal H_r^\times}\bigr)^{hH}
\]
becomes an equivalence after inverting $p$.
\end{proposition}

\begin{proof}
By \cref{lem:H-cyclotomic-cohomology}, for every $b\ge 0$ there exists an integer $N_b\ge 0$
such that
\[
p^{N_b}\cdot H^b(K)=0,
\]
where $K:=\cofib(\alpha_{\HDr,r}^{\times})$.

Since $\GL_n(\ZZ_p)$ has finite virtual cohomological dimension, so does the open subgroup
$H$. Choose an open normal subgroup $U\triangleleft H$
of finite cohomological dimension.

We first show that each cohomology group of $K^{hU}$ is $p$-power torsion. Apply
\cite[Prop.~3.3.6]{Barthe-Schlank-Stapleton-Weinstein-rationalization} to the closed normal
subgroup $1\subset U$. This gives a spectral sequence
\[
E_2^{a,b}
=
H^a_{\cts}(U,H^b(K))
\to
H^{a+b}(K^{hU}).
\]
For each fixed $b$, the coefficient module $H^b(K)$ is killed by $p^{N_b}$, hence every
continuous cochain group computing $H^a_{\cts}(U,H^b(K))$ is also killed by $p^{N_b}$. It
follows that every $E_2^{a,b}$ is $p$-power torsion. Since $U$ has finite cohomological
dimension, only finitely many values of $a$ can contribute to a fixed total degree. Therefore
each group $H^m(K^{hU})$ admits a finite filtration with $p$-power torsion graded pieces, and hence is itself
$p$-power torsion.

Now apply \cite[Prop.~3.3.6]{Barthe-Schlank-Stapleton-Weinstein-rationalization} to the closed
normal subgroup $U\subset H$. We obtain a Lyndon--Hochschild--Serre spectral sequence
\[
E_2^{a,b}
=
H^a_{\cts}\bigl(H/U,\,H^b(K^{hU})\bigr)
\to
H^{a+b}(K^{hH}).
\]
Since $K^{hU}$ is connective and each $H^b(K^{hU})$ is $p$-power torsion, every $E_2^{a,b}$ is
again $p$-power torsion. For each total degree only finitely many values of $a$ occur, so
each
\[
H^m(K^{hH})
\]
admits a finite filtration with $p$-power torsion graded pieces. Hence $H^m(K^{hH})$ is
$p$-power torsion for every $m$.

Therefore
\[
H^m(K^{hH}[1/p])=0
\qquad\text{for all }m,
\]
so $K^{hH}[1/p]\simeq 0$. This proves that $(\alpha_{\HDr,r}^{\times})^{hH}$ becomes an
equivalence after inverting $p$.
\end{proof}

\subsection{Computing With the Two-Towers Isomorphism}\label{sec: two towers}

In this subsection we use the two-towers isomorphism
\[
\begin{tikzcd}
& X \arrow[dl, swap, "\GL_n(\ZZ_p)" ]
     \arrow[dr, "\GG_n"] & \\
\LTG & & \HDr
\end{tikzcd}
\]

We now introduce the subgroups with respect to which we will take  fixed points.

\begin{notation}
For $r\ge 1$, let
\[
\GG_{n,r}:=\ker\!\bigl(\GG_n\to (\ZZ/p^r)^\times\bigr),
\qquad
\GL_{n,r}:=\ker\!\bigl(\GL_n(\ZZ_p)\to \GL_n(\ZZ/p^r)\bigr),
\]
where the first map is the reduction modulo $p^r$ of the determinant character
$\GG_n\to \ZZ_p^\times$.
\end{notation}

Finally, we will use the two-towers isomorphism to compute $D_r[1/p]=\pi_0(C_{0,r})$.

\begin{proposition}\label{prop:hfp-comparison-finite-det-level}
There is an equivalence of algebra objects
\[
(\pi_0(C_{0,r})[\epsilon])^{h_c\GG_{n,r}}
\;\xrightarrow{\ \sim\ }\;
(\QQ_p(\omega_{p^r})[\epsilon])^{h_c\GL_{n,r}}.
\]
In particular,
\[
\pi_0(C_{0,r})^{h_c\GG_{n,r}}
\;\xrightarrow{\ \sim\ }\;
\QQ_p(\omega_{p^r})^{h_c\GL_{n,r}}.
\]
\end{proposition}

\begin{proof}
By the compatibility of the two-towers isomorphism with determinants
(\cite[Thm.~3.6.1 and Thm.~3.7.1]{Barthel-Schlank-Stapleton-Weinstein-2024-Picard}), the equivalence
\[
[\LTG^\diamond/\GG_n]\simeq [\HDr^\diamond/\GL_n(\ZZ_p)]
\]
pulls back along the level-$p^r$ cyclotomic cover on the height $1$ side to an equivalence
of quotient diamonds
\[
[\LevelG^\diamond/\GG_{n,r}]
\;\xrightarrow{\ \sim\ }\;
[(\mathcal H_r^\times)^\diamond/\GL_{n,r}].
\]
Applying condensed pro-etale cohomology with coefficients in $\Ohatp$ and using \cite[Proposition 3.6.3]{Barthe-Schlank-Stapleton-Weinstein-rationalization} we get an equivalence of algebra objects
\[
\RGammacond\bigl((\LevelG)_{\proet},\Ohatp_{\LevelG}\bigr)^{h\GG_{n,r}}
\;\xrightarrow{\ \sim\ }\;
\RGammacond\bigl((\HCyc)_{\proet},\Ohatp_{\HCyc}\bigr)^{h\GL_{n,r}}.
\]

By \cref{thm:finite-level-rational},
\[
\RGammacond\bigl((\LevelG)_{\proet},\Ohatp_{\LevelG}\bigr)[1/p]
\simeq
\pi_0(C_{0,r})[\epsilon],
\]
and by \cref{lem:H-cyclotomic-cohomology},
\[
\RGammacond\bigl((\HCyc)_{\proet},\Ohatp_{\HCyc}\bigr)[1/p]
\simeq
\QQ_p(\omega_{p^r})[\epsilon].
\]
Taking fixed points under $\GG_{n,r}$ and $\GL_{n,r}$, respectively and applying \cref{prop:H-cyclotomic-open-subgroups} and \cref{prop:finite-level-rational-open-subgroups}, yields
\[
(\pi_0(C_{0,r})[\epsilon])^{h\GG_{n,r}} \simeq D_r[\epsilon]^{h\GG_{n,r}}[1/p]
\;\xrightarrow{\ \sim\ }\;
\ZZ_p[\omega_p^r][\epsilon]^{h\GL_{n,r}}[1/p]\simeq (\QQ_p(\omega_{p^r})[\epsilon])^{h\GL_{n,r}}.
\]

Finally, the preceding equivalences are compatible with the canonical $\QQ_p[\epsilon]$-algebra structures.
Thus it is $\QQ_p[\epsilon]$-linear, hence carries $\epsilon$ to $\epsilon$. Passing to the quotient by
$(\epsilon)$ yields
\[
\pi_0(C_{0,r})^{h_c\GG_{n,r}}
\;\xrightarrow{\ \sim\ }\;
\QQ_p(\omega_{p^r})^{h_c\GL_{n,r}}.
\]
\end{proof}

\subsection{The $n$-Fold Universal Character of $\SS_{K(n)}$}\label{sec: SK(n) universal character}

In this subsection we tie everything together to compute the $n$-fold universal character of $\SS_{K(n)}$ and get a formula for $L_\QQ(\SS_{K(n)}^A)$ for any $\pi$-finite space $A$. 

\begin{lemma}\label{lem:C0r-hGGnr-spectrum}
    We have a $(\ZZ/p^r)^\times$-equivariant equivalence 
    \[
    \pi_0(C_{0,r})^{h_c\GG_{n,r}}\simeq C_{0,r}^{h_c\GG_{n,r}}\simeq L_\QQ(\SS_{K(n)}[\omega^{(n)}_{p^r}]^{\B(\ZZ/p^r)^n})[e_r^{-1}],
    \]
    where $e_r$ is the idempotent splitting the primitive root of unity.
\end{lemma}

\begin{proof}
    Consider the Devinatz--Hopkins homotopy fixed point spectral sequence for
    $\En^{\B(\ZZ/p^r)^n}$:
    \[
    E_2^{s,t}=H^s_{\cts}\!\bigl(\GG_{n,r},\pi_t(\En^{\B(\ZZ/p^r)^n})\bigr)
    \to
    \pi_{t-s}\!\bigl((\En^{\B(\ZZ/p^r)^n})^{h_c\GG_{n,r}}\bigr).
    \]
    This spectral sequence is strongly convergent and has a horizontal vanishing line. Moreover, since $\GG_{n,r}$ acts trivially on $\B(\ZZ/p^r)^n$, the abutment identifies
    with
    \[
    \pi_{t-s}\!\bigl((\SS_{K(n)}[\omega^{(n)}_{p^r}])^{\B(\ZZ/p^r)^n}\bigr).
    \]
    Rationalizing therefore gives another strongly convergent spectral sequence
    \[
    \QQ\otimes E_2^{s,t}
    \simeq
    H^s_{\cts}\!\bigl(\GG_{n,r},\QQ\otimes\pi_t(\En^{\B(\ZZ/p^r)^n})\bigr)
    \to
    \pi_{t-s}\!\bigl(L_\QQ((\SS_{K(n)}[\omega^{(n)}_{p^r}])^{\B(\ZZ/p^r)^n})\bigr).
    \]

    By construction,
    \[
    e_r\in \pi_0\!\bigl(L_\QQ(\En^{\B(\ZZ/p^r)^n})\bigr)
    \]
    is $\GG_{n,r}$-invariant. Therefore multiplication by $e_r$ defines an idempotent endomorphism of the
    rational $\GG_{n,r}$-spectrum $L_\QQ(\En^{\B(\ZZ/p^r)^n})$. By naturality of the homotopy fixed point spectral sequence, this idempotent acts on every page of the rationalized spectral sequence. Passing to the image of this idempotent on each page gives a direct summand spectral sequence, still strongly
    convergent, whose abutment is the $e_r$-summand
    \[
    \pi_{t-s}\!\bigl(L_\QQ((\SS_{K(n)}[\omega^{(n)}_{p^r}])^{\B(\ZZ/p^r)^n})[e_r^{-1}]\bigr),
    \]
    
    Thus we obtain a strongly convergent spectral sequence
    \[
    E_{2,e_r}^{s,t}
    =
    H^s_{\cts}\!\bigl(\GG_{n,r},\pi_t(C_{0,r})\bigr)
    \to
    \pi_{t-s}\!\bigl(L_\QQ((\SS_{K(n)}[\omega^{(n)}_{p^r}])^{\B(\ZZ/p^r)^n})[e_r^{-1}]\bigr).
    \]

    We claim that
    \[
    H^s_{\cts}(\GG_{n,r},\pi_t(C_{0,r}))=0
    \qquad
    (t\neq0).
    \]
    Choose $N$ sufficiently large so that the central scalar subgroup
    \[
    Z:=1+p^N\ZZ_p
    \]
    is inside $\GG_{n,r}$ when $p=2$, we also choose $N\geq2$. Let
    $\gamma\in Z$ be a topological generator. The odd homotopy groups of
    $C_{0,r}$ vanish, while $\gamma$ acts on $\pi_{2m}(C_{0,r})$ through
    the scalar $\gamma^m$. If $m\neq0$, then $\gamma^m-1$ is a nonzero
    element of $\QQ_p$ and hence acts invertibly on
    $\pi_{2m}(C_{0,r})$. The complex computing the continuous cohomology
    of $Z$ is therefore
    \[
    \pi_{2m}(C_{0,r})
    \xto{\gamma-1}
    \pi_{2m}(C_{0,r}),
    \]
    and is acyclic. Thus
    \[
    H^*_{\cts}(Z,\pi_{2m}(C_{0,r}))=0
    \qquad
    (m\neq0).
    \]
    Since $Z$ is central, it is a closed normal subgroup of $\GG_{n,r}$.
    The continuous Lyndon--Hochschild--Serre spectral sequence for
    \[
    1\to Z\to\GG_{n,r}\to\GG_{n,r}/Z\to1
    \]
    takes the form
    \[
    E_2^{a,b}
    =
    H^a_{\cts}\left(
    \GG_{n,r}/Z,
    H^b_{\cts}(Z,\pi_t(C_{0,r}))
    \right)
    \to
    H^{a+b}_{\cts}(\GG_{n,r},\pi_t(C_{0,r})).
    \]
    For $t\neq0$, the inner continuous cohomology groups vanish in every
    degree, so the entire $E_2$-page vanishes. Therefore
    \[
    H^s_{\cts}(\GG_{n,r},\pi_t(C_{0,r}))=0
    \qquad
    (t\neq0).
    \]
    Consequently, the above spectral sequence is concentrated in the row
    $t=0$ and therefore collapses at the $E_2$-page, without extension
    problems. Hence
    \[
    \pi_{-*}\left(
    L_\QQ((\SS_{K(n)}[\omega^{(n)}_{p^r}])^{\B(\ZZ/p^r)^n})
    [e_r^{-1}]
    \right)
    \simeq
    H^*_{\cts}(\GG_{n,r},\pi_0(C_{0,r})) \simeq \pi_{-*}(\pi_0(C_{0,r})^{h_c\GG_{n,r}}).
    \]
\end{proof}

By \cref{lem:L!-as-finite-quotient-hfps}, we can compute the universal character of $\SS_{K(n)}$ using \cref{lem:C0r-hGGnr-spectrum}.
 
\begin{corollary}\label{cor: universal character SK(n)}
    The universal character of  $\SS_{K(n)}$ is its rationalization. That is 
    \[
    (\L_p^n)_!(\SS_{K(n)})\simeq L_\QQ(\SS_{K(n)}),
    \]
    with trivial $\GL_n(\ZZ_p)$-action.
\end{corollary}

\begin{proof}
    By \cref{lem:L!-as-finite-quotient-hfps}, we have
    \[
    (\L_p^n)_!(\SS_{K(n)})
    \simeq
    \colim_r
    \left(
    L_\QQ((\SS_{K(n)}[\omega^{(n)}_{p^r}])^{\B(\ZZ/p^r)^n})
    [e_r^{-1}]
    \right)^{h(\ZZ/p^r)\units}.
    \]
    By \cref{lem:C0r-hGGnr-spectrum} and
    \cref{prop:hfp-comparison-finite-det-level}, the $r$-th term is
    equivalent to
    \[
    \left(
    \QQ_p(\omega_{p^r})^{h_c\GL_{n,r}}
    \right)^{h(\ZZ/p^r)\units}.
    \]
    
    The group $\GL_{n,r}$ acts trivially on $\QQ_p(\omega_{p^r})$.
    Consequently, the continuous homotopy fixed point spectral sequence
    gives
    \[
    \pi_{-*}\left(
    \QQ_p(\omega_{p^r})^{h_c\GL_{n,r}}
    \right)
    \simeq
    H^*_{\cts}(\GL_{n,r},\QQ_p(\omega_{p^r}))
    \simeq
    \QQ_p(\omega_{p^r})
    \otimes_{\QQ_p}
    H^*_{\cts}(\GL_{n,r},\QQ_p).
    \]
    Under this identification, $(\ZZ/p^r)\units$ acts through its Galois
    action on $\QQ_p(\omega_{p^r})$ and trivially on the second factor.
    Since $(\ZZ/p^r)\units$ is finite and we are working rationally, its
    higher cohomology vanishes. Taking homotopy fixed points therefore
    gives
    \[
    \pi_{-*}
    \left(
    \left(
    \QQ_p(\omega_{p^r})^{h_c\GL_{n,r}}
    \right)^{h(\ZZ/p^r)\units}
    \right)
    \simeq
    H^*_{\cts}(\GL_{n,r},\QQ_p).
    \]
    
    The Lazard comparison identifies
    \[
    H^*_{\cts}(\GL_{n,r},\QQ_p)
    \simeq
    H^*(\mathfrak{gl}_n(\QQ_p),\QQ_p)
    \simeq
    \Lambda_{\QQ_p}(x_1,x_3,\dots,x_{2n-1}),
    \qquad
    |x_{2i-1}|=2i-1.
    \]
    These identifications are natural with respect to inclusions of open
    subgroups. In particular, the transition map from the $r$-th term to
    the $(r+1)$-st term induces the restriction map
    \[
    H^*_{\cts}(\GL_{n,r},\QQ_p)
    \xto{\mathrm{res}}
    H^*_{\cts}(\GL_{n,r+1},\QQ_p).
    \]
    The Lie algebras of $\GL_{n,r}$ and $\GL_{n,r+1}$ are both canonically
    $\mathfrak{gl}_n(\QQ_p)$, and under the Lazard comparison the
    restriction map is identified with the identity of
    $H^*(\mathfrak{gl}_n(\QQ_p),\QQ_p)$. It is therefore an isomorphism.
    Hence every transition map in the colimit diagram is an equivalence.
    
    The same argument identifies the map from the zeroth term
    $L_\QQ(\SS_{K(n)})$ to the $r$-th term with the restriction map
    \[
    H^*_{\cts}(\GL_n(\ZZ_p),\QQ_p)
    \to
    H^*_{\cts}(\GL_{n,r},\QQ_p),
    \]
    which is again identified by Lazard with the identity of
    $H^*(\mathfrak{gl}_n(\QQ_p),\QQ_p)$. Thus the canonical map
    \[
    L_\QQ(\SS_{K(n)})
    \to
    (\L_p^n)_!(\SS_{K(n)})
    \]
    is an equivalence.
    
    Finally, the residual $\GL_n(\ZZ_p)$-action corresponds under the
    Lazard comparison to the adjoint action on
    $H^*(\mathfrak{gl}_n(\QQ_p),\QQ_p)$. This action is trivial because
    inner automorphisms act trivially on Lie algebra cohomology. Therefore
    the induced $\GL_n(\ZZ_p)$-action on
    $(\L_p^n)_!(\SS_{K(n)})$ is trivial.
\end{proof}

We conclude.

\begin{theorem}\label{thm: rationalization SK(n)A}
Let $A$ be a $\pi$-finite space. The profinite group $\GL_n(\ZZ_p)$ acts by precomposition on
\[
\Map(\B\ZZ_p^n,A)
\simeq
\colim_{r}\Map(\B(\ZZ/p^r)^n,A),
\]
where the equivalence follows from \cite[Proposition~3.4.7]{Lurie-2019-Elliptic3}. Write
\[
S :=
\pi_0\Map(\B\ZZ_p^n,A)/\GL_n(\ZZ_p).
\]
Then
\[
L_{\QQ}(\SS_{K(n)}^A)
\simeq
L_{\QQ}(\SS_{K(n)})^S.
\]
\end{theorem}

\begin{proof}
    This follows immediately by \cref{cor: universal character SK(n)} and \cref{thm: unit-rational-hGL} as the $\GL_n(\ZZ_p)$ action on $(\L_p^n)_!(\SS_{K(n)}) \simeq L_{\QQ}(\mathbb S_{K(n)}) $ is trivial.  
\end{proof}

We will show that this formula is especially nice when $A=\B G$ for $G$ a finite group. We start with a preliminary lemma.

\begin{lemma}\label{lem: count orbits as subgroups}
Let $G$ be a finite group, and let
\[
S_{n,p}(G):=\Hom_{\cts}(\ZZ_p^n,G)/G,
\]
where $G$ acts by conjugation on the target. Then the
$\GL_n(\ZZ_p)$-orbits on $S_{n,p}(G)$ are naturally in bijection with
the conjugacy classes of abelian $p$-subgroups $A\leq G$ such that
$d(A)\leq n$, where $d(A)$ is the minimal number of generators of
$A$.
\end{lemma}

\begin{proof}
A continuous homomorphism $\phi\colon\ZZ_p^n\to G$ has finite image.
Since $\ZZ_p^n$ is abelian and pro-$p$, its image is a finite abelian
$p$-group generated by at most $n$ elements. Thus
\[
[\phi]\longmapsto[\phi(\ZZ_p^n)]
\]
defines a map from the $\GL_n(\ZZ_p)$-orbits on $S_{n,p}(G)$ to the
conjugacy classes of abelian $p$-subgroups of $G$ generated by at most
$n$ elements. Every such subgroup occurs as the image of a continuous
homomorphism $\ZZ_p^n\to G$, so this map is surjective.

It remains to prove injectivity. Suppose that $\phi,\psi\colon\ZZ_p^n
\to G$ have conjugate images. After conjugating $\psi$, we may assume
that both are surjections onto the same finite abelian $p$-subgroup
$A\leq G$.

Reducing modulo $p$ gives surjections
\[
\overline{\phi},\overline{\psi}\colon
(\FF_p)^n\to A/pA.
\]
The group $\GL_n(\FF_p)$ acts transitively on the set of surjections
from $(\FF_p)^n$ onto the fixed vector space $A/pA$. Hence there is
some $\overline{u}\in\GL_n(\FF_p)$ such that
\[
\overline{\phi}\circ\overline{u}=\overline{\psi}.
\]
Choose a lift $u_0\in\GL_n(\ZZ_p)$ of $\overline{u}$. For the standard
basis $e_1,\dots,e_n$ of $\ZZ_p^n$, we have
\[
\psi(e_i)-\phi(u_0e_i)\in pA.
\]
Choose $a_i\in A$ such that
\[
pa_i=\psi(e_i)-\phi(u_0e_i),
\]
and then choose $w_i\in\ZZ_p^n$ such that $\phi(w_i)=a_i$. Define
$u\colon\ZZ_p^n\to\ZZ_p^n$ by
\[
u(e_i):=u_0(e_i)+pw_i.
\]
Since $u$ reduces to $\overline{u}$ modulo $p$, we have
$u\in\GL_n(\ZZ_p)$. Moreover,
\[
\phi(u(e_i))
=
\phi(u_0e_i)+p\phi(w_i)
=
\psi(e_i)
\]
for every $i$, and therefore $\psi=\phi\circ u$.
\end{proof}
We conclude.

\begin{corollary}\label{cor: rationalization SK(n)BG}
    Let $G$ be a finite group and write
    \[
    S:=\{ \textrm{abelian} \ p-\textrm{subgroups of }G \  \textrm{generated by at most } n \ \textrm{elements}\}/\textrm{conjugation}.
    \]
    Then
    \[
    L_{\QQ}(\mathbb S_{K(n)}^{\B G})\simeq L_{ \QQ}(\mathbb S_{K(n)})^S.
    \]
\end{corollary}
 
\begin{proof}
    This is immediate from \cref{thm: rationalization SK(n)A} and \cref{lem: count orbits as subgroups}.
\end{proof}

\subsection{Rational $K(n)$-Local Power Operations} \label{sec: Rational K(n) Local Power Operations}

In this section we identify the ring of rational $K(n)$-local power operations.  

\begin{definition}\label{def:LQCAlg}
Consider the functor
\[
    \pi_0^\QQ\colon \CAlg_{\Kn}\to \Set,
    \qquad
    A\mapsto \pi_0(L_\QQ A),
\]
where we regard $\pi_0(L_\QQ A)$ as its underlying set. Pointwise addition
and multiplication endow
\[
    \pi_0\Nat(\pi_0^\QQ,\pi_0^\QQ)
\]
with a ring structure. We call this the \emph{ring of rational
$K(n)$-local power operations}. Note that this ring is naturally graded by weight.
\end{definition}

\begin{remark}
Composition of natural transformations endows
$\pi_0\Nat(\pi_0^\QQ,\pi_0^\QQ)$ with a natural composition law.
More generally, the multivariable rational $K(n)$-local power
operations carry a natural plethystic structure. In this section, we
identify only the ring of unary operations and do not study this
additional structure.
\end{remark}

We now identify the ring of rational $K(n)$-local power operations with
the rationalization of the free commutative algebra on one generator.

\begin{lemma}\label{lem:yoneda-identifies-power-operations}
There is a canonical isomorphism of graded rings
\[
    \pi_0\Nat(\pi_0^\QQ,\pi_0^\QQ)
    \simeq
    \pi_0L_\QQ\left(
        \bigoplus_{k\geq 0}\SSKn^{\B\Sn[k]}
    \right)
    \simeq
    \bigoplus_{k\geq 0}
    \pi_0L_\QQ\left(\SSKn^{\B\Sn[k]}\right).
\]
\end{lemma}

\begin{proof}
Set
\[
    R:=
    \bigoplus_{k\geq 0}\SSKn^{\B\Sn[k]}
    \in \CAlg(\SpKn).
\]
The standard formula for the free commutative algebra, together with
ambidexterity in $\SpKn$, gives
\[
    \operatorname{Sym}_{\SpKn}(\SSKn)
    \simeq
    \bigoplus_{k\geq 0}
    (\SSKn^{\otimes k})_{h\Sn[k]}
    \simeq
    \bigoplus_{k\geq 0}\SSKn^{\B\Sn[k]}
    =R.
\]
Thus, $R$ is the free commutative algebra on one generator. In particular,
\[
    h_R(A):=
    \pi_0\Map_{\CAlg(\SpKn)}(R,A)\simeq \pi_0(A)
\]

Let
\[
    [p]\colon R\to R
\]
be the endomorphism that multiplies the universal generator by $p$.
Under the identification $h_R(A)\simeq\pi_0(A)$, precomposition with
$[p]$ agrees with multiplication by $p$.
Consequently, there is a natural isomorphism
\[
    \pi_0^\QQ
    \simeq
    \colim\left(
        h_R\xto{[p]^*}h_R
        \xto{[p]^*}h_R
        \xto{[p]^*}\cdots
    \right).
\]

Mapping out of this colimit and applying the Yoneda lemma gives
\[
\begin{aligned}
    \pi_0\Nat(\pi_0^\QQ,\pi_0^\QQ)
    &\simeq
    \lim\left(
        \pi_0^\QQ(R)
        \xgets{\pi_0^\QQ([p])}
        \pi_0^\QQ(R)
        \xgets{\pi_0^\QQ([p])}
        \cdots
    \right).
\end{aligned}
\]
The endomorphism $[p]$ acts on the summand of weight $k$ in $R$ by
multiplication by $p^k$. It therefore becomes an equivalence after
rationalization. Hence projection onto the initial term identifies the
last limit with $\pi_0^\QQ(R)=\pi_0L_\QQ R$. Remembering the grading on everything then gives the claim.
\end{proof}

We now identify the ring $\bigoplus_{m\geq 0} \QQ_p^{\L^n_p\B\Sigma_m}$ as a polynomial ring.

\begin{lemma}\label{lem: QLB as a poly ring}
There is an isomorphism of graded rings
\[
\bigoplus_{m\geq 0}\QQ_p^{\L_p^n\B\Sn[m]}
\simeq
\QQ_p[x_H\mid H\leq_o\ZZ_p^n],
\]
where the left-hand side is graded by $m$, and the right-hand side is
graded by $|x_H|=[\ZZ_p^n:H]$. The multiplication on the left is given
by the transfers induced by the block-sum maps
$\B\Sn[m]\times\B\Sn[k]\to\B\Sn[m+k]$.
\end{lemma}

\begin{proof}
For every $m\geq 0$, there is a canonical identification
\[
\pi_0(\L_p^n\B\Sn[m])
\cong
\Hom_{\cts}(\ZZ_p^n,\Sn[m])/\Sn[m],
\]
where $\Sn[m]$ acts by conjugation. As $\Sn[m]$ is a finite group every homomorphism
$\ZZ_p^n\to\Sn[m]$ is automatically continuous. Such a homomorphism is equivalently an action of $\ZZ_p^n$ on a set of cardinality $m$. Thus
$\pi_0(\L_p^n\B\Sn[m])$ is naturally identified with the set of
isomorphism classes of finite $\ZZ_p^n$-sets of cardinality $m$. It follows that $\QQ_p^{\L_p^n\B\Sn[m]}$ has a basis $\{\delta_X\}$ indexed by the
isomorphism classes of finite $\ZZ_p^n$-sets $X$ of cardinality $m$,
where $\delta_X$ is the characteristic function of the component
corresponding to $X$.

Let $X$ and $Y$ have cardinalities $m$ and $k$. On the corresponding
components, the block-sum map is induced by the inclusion
\[
\Aut_{\ZZ_p^n}(X)\times\Aut_{\ZZ_p^n}(Y)
\to
\Aut_{\ZZ_p^n}(X\sqcup Y).
\]
The transfer on degree-zero cohomology therefore gives
\[
\delta_X\delta_Y
=
[\Aut_{\ZZ_p^n}(X\sqcup Y):
\Aut_{\ZZ_p^n}(X)\times\Aut_{\ZZ_p^n}(Y)]
\delta_{X\sqcup Y}.
\]

Every finite $\ZZ_p^n$-set decomposes uniquely as
$X\simeq\coprod_Ha_H(\ZZ_p^n/H)$, where $H$ ranges over the open
subgroups of $\ZZ_p^n$ and only finitely many $a_H$ are nonzero. Since $\ZZ_p^n$ is abelian, two such sets are isomorphic
if and only if the corresponding open subgroups are equal.

If
\[
X\simeq\coprod_Ha_H(\ZZ_p^n/H)
\qand
Y\simeq\coprod_Hb_H(\ZZ_p^n/H),
\]
then
\[
\Aut_{\ZZ_p^n}(X)
\cong
\prod_H
\left((\ZZ_p^n/H)^{a_H}\rtimes\Sn[a_H]\right).
\]
Consequently,
\[
[\Aut_{\ZZ_p^n}(X\sqcup Y):
\Aut_{\ZZ_p^n}(X)\times\Aut_{\ZZ_p^n}(Y)]
=
\prod_H\binom{a_H+b_H}{a_H},
\]
and hence
\[
\delta_X\delta_Y
=
\prod_H\binom{a_H+b_H}{a_H}\delta_{X\sqcup Y}.
\]

Define
\[
\widetilde{\delta}_X
:=
\prod_Ha_H!\,\delta_X.
\]
Since every factorial is invertible in $\QQ_p$, the elements
$\widetilde{\delta}_X$ again form a basis. The preceding formula gives
$\widetilde{\delta}_X\widetilde{\delta}_Y
=\widetilde{\delta}_{X\sqcup Y}$. Thus, in the basis
$\{\widetilde{\delta}_X\}$, the ring
$\bigoplus_{m\geq 0}\QQ_p^{\L_p^n\B\Sn[m]}$ is the monoid algebra over
$\QQ_p$ of the commutative monoid of finite $\ZZ_p^n$-sets under
disjoint union.

This monoid is freely generated by the transitive sets $\ZZ_p^n/H$ for
$H\leq_o\ZZ_p^n$. We therefore obtain an isomorphism
\[
\QQ_p[x_H\mid H\leq_o\ZZ_p^n]
\simeq
\bigoplus_{m\geq 0}\QQ_p^{\L_p^n\B\Sn[m]}
\]
sending $x_H$ to $\delta_{\ZZ_p^n/H}$. More generally, if
$X\simeq\coprod_Ha_H(\ZZ_p^n/H)$, then $\prod_Hx_H^{a_H}$ is sent to
$\widetilde{\delta}_X=\prod_Ha_H!\,\delta_X$. Since
$|\ZZ_p^n/H|=[\ZZ_p^n:H]$, this isomorphism respects the grading.
\end{proof}

We now put things together to compute the ring of rational $K(n)$-local power operations.

\begin{proposition} \label{prop: K(n) local rational power operations}
    The ring of rational $K(n)$-local power operations is 
    \[
    \pi_0 \Nat(\pi_0^\QQ,\pi_0^\QQ) \simeq \bigoplus_{k \geq 0} \pi_0 L_\QQ(\SSKn^{\B\Sn[k]}) \simeq \QQ_p[x_H\mid H\leq_o \ZZ_p^n]^{\GL_n(\ZZ_p)},
    \]
    where the grading on the right-hand side is induced by declaring $\deg(x_H)=[\ZZ_p^n: H]$ before taking fixed points.
\end{proposition}

\begin{proof}
    The first equivalence is \cref{lem:yoneda-identifies-power-operations}. It remains to identify the graded ring $\bigoplus_{k \geq 0} \pi_0 L_\QQ\SSKn^{\B\Sn[k]}$.
    
    By \cref{thm: unit-rational-hGL}, this graded ring may be computed from the graded commutative algebra
    \[
    \bigoplus_{k\geq 0}\SSKn^{\B\Sn[k]}
    \in 
    \CAlg\bigl(\Fun^{\Day}(\NN,\Sp_{K(n)})\bigr)
    \]
    by applying $\L^n_!$ degreewise and then taking $\GL_n(\ZZ_p)$-fixed points. By \cref{cor: universal character SK(n)} together with \cref{thm: F! of R A}, applying $\L^n_!$ to $\bigoplus_{k\geq 0}\SSKn^{\B\Sn[k]}$ yields
    \[
    \bigoplus_{k\geq 0}(L_\QQ\SS_{K(n)})^{\L_p^n\B\Sn[k]}
    \in
    \CAlg\bigl(\Fun^{\Day}(\NN,\Mod_\QQ)\bigr).
    \]
    By \cref{lem: QLB as a poly ring}, there is an isomorphism of graded rings
    \[
    \bigoplus_{k\geq 0}\QQ_p^{\L_p^n\B\Sn[k]}
    \simeq 
    \QQ_p[x_H \mid H\leq_o \ZZ_p^n].
    \]
    Moreover, this isomorphism is $\GL_n(\ZZ_p)$-equivariant. Passing to $\pi_0$ and $\GL_n(\ZZ_p)$-fixed points therefore gives
    \[
    \bigoplus_{k \geq 0} \pi_0L_\QQ(\SSKn^{\B\Sn[k]})
    \simeq 
    \QQ_p[x_H \mid H\leq_o \ZZ_p^n]^{\GL_n(\ZZ_p)}.
    \]
\end{proof}

\bibliographystyle{alpha}
\phantomsection\addcontentsline{toc}{section}{\refname}
\bibliography{references}

\end{document}